\documentclass[11pt, reqno]{amsart}
\usepackage{graphicx, epsfig, psfrag}
\usepackage{amsfonts}
\usepackage{amssymb}
\usepackage{amsmath}
\usepackage{cases}
\usepackage{xcolor}
\usepackage{float}
\usepackage{hyperref}
\usepackage{theorem}
\usepackage{mathrsfs}
\usepackage{pdfsync}
\usepackage{subfig}
\usepackage{booktabs}

\allowdisplaybreaks[4]

\newcounter{theorem}
\def\openthm#1#2{\refstepcounter{theorem}\bigskip

{\noindent\bf#1~\thetheorem\if#2!{. }\else{ (#2).}\fi}
\it}

\def\thmskip{}
\newenvironment{theorem}[1][!]{\openthm{Theorem}{#1}}{\thmskip}
\newenvironment{lemma}[1][!]{\openthm{Lemma}{#1}}{\thmskip}

\newenvironment{proposition}[1][!]{\openthm{Proposition}{#1}}{\thmskip}

\newcounter{assumption}
\def\openass#1#2{\refstepcounter{assumption}\bigskip

{\noindent\bf#1~\theassumption\if#2!{. }\else{ (#2).}\fi}
\it}
\def\thmskip{}
\newenvironment{assumption}[1][!]{\openass{Assumption}{#1}}{\thmskip}

\newcounter{remark}
\def\openrem#1#2{\refstepcounter{remark}\bigskip
{\noindent \it \bfseries#1~\theremark\if#2!{. }\else{ (#2). }\fi}}
\newenvironment{remark}[1][!]{\openrem{Remark}{#1}}{\qed}

\newcounter{algorithm}
\def\openalg#1#2{\refstepcounter{algorithm}\bigskip
{\noindent \it \bfseries#1~\thealgorithm\if#2!{. }\else{ (#2). }\fi}}
\newenvironment{algorithm}[1][!]{\openalg{Algorithm}{#1}}{\qed}

\def\R{\mathbb{R}}

\def\Z{\mathbb{Z}}

\def\dx{\,{\rm d}x}

\def\<{\langle}
\def\>{\rangle}

\def\argmin{{\rm argmin}}
\def\sign{{\rm sgn}}
\def\conv{{\rm conv}}

\def\eps{\varepsilon}

\def\Ps{\mathcal{P}}
\def\Ys{\mathcal{Y}}
\def\Us{\mathcal{U}}

\def\As{\mathcal{A}}
\def\Cs{\mathcal{C}}
\def\Is{\mathcal{I}}
\def\Ls{\mathcal{L}}
\def\Ns{\mathcal{N}}
\def\Ss{\mathcal{S}}

\def\mrLs{\mathring{\mathcal{L}}}

\def\Ks{\mathscr{K}}

\def\Ts{\mathscr{T}}

\def\tom{\widetilde{\omega}}
\def\th{\widetilde{h}}

\def\tN{\tilde{N}}

\def\Es{\mathcal{E}}

\def\Ys{\mathcal{Y}}

\def\gb{\boldsymbol{g}}

\def\ac{{\rm ac}}
\def\cp{{\rm cp}}
\def\a{{\rm a}}

\def\c{{\rm c}}
\def\i{{\rm i}}
\def\NN{{\rm NN}}
\def\NNN{{\rm NNN}}

\def\M{{\rm M}}
\def\LJ{{\rm LJ}}

\def\mo{{\rm mo}}
\def\cg{{\rm cg}}

\def\leftSup{{\rm left}}
\def\rightSup{{\rm right}}
\def\interface{{\rm int}}

\def\z{{\rm z}}

\def\hybrid{{\rm hybrid}}

\def\Os{\mathcal{O}}

\renewcommand{\cases}[1]{\left\{ \begin{array}{rl} #1 \end{array} \right.}

\numberwithin{equation}{section}

\begin{document}

\title[Analysis of a Posteriori Error Estimators for an A/C Coupling Method]{Analysis of the Residual  Type and the Recovery Type a Posteriori Error Estimators for a Consistent Atomistic-to-Continuum Coupling Method in 1D}

\author{Hao Wang}
\address{Hao Wang\\
College of Mathematics\\
Sichuan University\\
No.24 South Section One, Yihuan Road\\
Chengdu, 610065\\
China}
\email{wangh@scu.edu.cn}

\author{Siyao Yang}
\address{Siyao Yang\\
Department of Mathematics, Faculty of Science\\
National University of Singapore\\
21 Lower Kent Ridge Road, 119077\\
Singpaore}
\email{siyao\_yang@u.nus.edu}
\date{\today}

\begin{abstract}
We consider the \emph{a posteriori} error estimation for an atomistic-to-continuum coupling scheme for a generic one-dimensional many-body next-nearest-neighbour interaction model in 1D. We derive and rigorously prove the efficiency of the residual type estimator. We prove the equivalence between the residual type and the gradient recovery type estimator in the continuum region and propose a (novel) hybrid a \emph{a posteriori} error estimator by combining the two types of estimators. Our numerical experiments illustrate the optimal convergence rate of the adaptive algorithms using these estimators whose efficiency factors are also presented.
\end{abstract}


\thanks{H.W. was supported by NSFC grant 11501389, 11471214 and Sichuan University Starting Up Research Funding No. 2082204194117. S.Y. was supported by the President's Graduate Fellowship of National University of Singapore.}

\maketitle

\section{Introduction}
\label{Sec: Intro}
Atomistic-to-continuum coupling methods (a/c methods) are a class of multiscale methods for coupling an atomistic description of a solid to a matching continuum elasticity model. Such methods combine the accuracy of the atomistic simulation and the efficiency of the continuum model in the way that the atomistic model is applied in the region where localized crystal defects, such as vacancies, dislocations and defects, may happen and the continuum model is used in the elastic far field together with finite element discretization to reduce the degree of freedom. We refer to \cite{MO_TEB_First_QC_1996, Shimokawa_QNL_Original_2004, E_Lu_Yang_Geo_Consitent_2006, Li_Luskin_QNL_2012, Shapeev_Consisten_A_C_1D_2D,  Shapeev_ACC_3D, CO_LZ_QC_2D_2011, CO_LZ_Finite_GRAC_2014} for the construction of such methods and  \cite{MRE_TEB_QCReview_2003, MRE_TDB_QC_Review_2009, CO_ML_ACCoupling_2013} for the reviews.

Numerical analysis for the a/c methods has been an important research field in the computational mathematics community and considerable effort has been given to the \emph{a priori} analysis \cite{Lin_1D_LJ_QC_Ana_2003, Lin_QCP_2D_2007, CO_ES_1D_QC_2008, MD_ML_Opt_Order_2009, PM_ZY_QNL_2009, CO_QNL1D_2011, CO_HW_QC_1D_A_Priori_2011, CO_2D_Patch_Test, CO_AS_ACC_2D_Analysis} where the issue of model error which is committed by the artificial coupling interface between different models has been extensively discussed. However, the problems of \emph{a posteriori} error control and adaptivity for the coupling methods has attracted comparatively little attention. First noticable results on the a posteriori error analysis were given in \cite{AM_A_Post_ML_FK_Model_2007, AM_A_Post_ML_FK_Model_2008_Model, AM_A_Post_ML_FK_Model_2008_FEM, Prudhomme_Bauman_Adaptive_A_to_C_2009} where the \emph{goal-oriented} a posteriori error estimator have been derived for the original energy-based quasicontinuum method \cite{MO_TEB_First_QC_1996}. Such approach requires the use of the solutions of dual problems which may cause additional cost of computation. Moreover, since the original energy-based quasicontinuum method exhibits large model error on the coupling interface, the size of the atomistic region tends to be larger than needed and thus further increase the computational effort. The residual based a posteriori error estimate, which is the first approach employed in the current work, is first derived in \cite{CO_ES_1D_QC_2008} for a coarse-grained atomistic model. However, since no coupling of different models occurs, the coarse-grained scheme in \cite{CO_ES_1D_QC_2008} is essentially different from the a/c methods we analyze. A recent advance in this direction is the a posteriori error estimates for consistent energy-based coupling methods in one- and two-dimensional settings \cite{CO_HW_ACC_1D_A_Posteriori} and \cite{HW_ML_PL_LZ_2D_A_Post} in which a posteriori error estimators are derived through a residual based approach which is similar to \cite{CO_ES_1D_QC_2008} with the coupling feature of a/c method being kept.

In the computational material science community, adaptivity has always been used with a/c coupling methods since the methods were developed. One of the most commonly used a posteriori error estimators is the gradient recovery error estimator which was first developed for second order elliptic equations in \cite{Zienkiewicz_Zhu_Error_Estimator_1987} and appeared in a/c literature in \cite{MO_EBT_QC_Adaptive_99}. The advantage of the gradient average error estimator lies in its simplicity and consequently low cost of computation especially in higher dimensions. However, no rigorous justification of using gradient recovery estimator for a/c method has been given.

The present work attempts to bridge the gap between theoretical analysis and computational practice.

The first goal of the present work can be considered as an extension of the analysis in \cite{CO_HW_ACC_1D_A_Posteriori} where we derive the residual based error estimator for the geometry reconstruction based atomistic-to-continuum coupling (GRAC) method \cite{CO_LZ_QC_2D_2011,CO_LZ_Finite_GRAC_2014} which can be applied to many-body finite range interaction models. In addition to that, we prove the efficiency of the residual based error estimator which essentially shows that such error estimator, up to a computable constant, provides a local lower bound for the true error.

The second goal of the present work is to construct and analyze the gradient recovery type a posteriori error estimator for a/c method which is much more widely used in computational practice. We prove the equivalence between the gradient recovery type error estimator and the residual based error estimator in the continuum region. In addition, to fully reflect the influence of the interface, we formulate the so called hybrid type a posteriori error estimator by combing the recovery type and the residual type error estimators.

%Because of the efficiency of the residual based error estimator, the efficiency of the gradient recovery error estimator holds consequently. We note that the formulation of the gradient recovery error estimator differs from that of second order elliptic equations as a result of the coupling model we consider.

We restrict our analysis to a 1D periodic many-body next-nearest-neighbour interaction atomistic system in order to present our idea in a simple setting. Although in a 1D setting, the technicality and subtlety of the analysis makes the present work nontrivial which we consider as a valuable step towards the analysis in higher dimensional cases.

% Outline 
\subsection{Outline}
\label{Sec: Outline}
In \ref{Sec: Model}, we formulate the atomistic model and its GRAC approximation. We define the weak problems according to the formulations of the two different models. We also introduce the notation that is used throughout the derivation and the analysis.

In \ref{Sec: Residual Based EE}, we establish the residual based a posteriori error estimator for the GRAC method and in Section \ref{Sec: Efficiency}, we prove the efficiency of the residual based a posteriori error estimator.

In \ref{Sec: Gradient Recover EE}, we construct and analyze the gradient recovery a posteriori error estimator  \cite{VerfurthAPosteriori} for the a/c method we consider and formulate the hybrid a posteriori error estimator which preserves certain good properties.

In \ref{Sec: Numeric Exp}, we describe the mesh refinement algorithms according to different a posteriori error estimators and present numerical examples to illustrate the convergence of these algorithms.

%%%%%%%%%%%%%%%%% Model Set Up %%%%%%%%%%%%%%%%%%%%%%%%%%%%%%%%%%

\section{The Atomistic Model and its GRAC approximation}
\label{Sec: Model}

%%%%%%%%%%%%%%%%% Atomistic Model %%%%%%%%%%%%%%%%%%%%%%%%%%%%%%%%%%

\subsection{Atomistic model}
\label{Sec: A Model}
We consider a model problem in the domain $\Omega = (-1, 1]$ containing $2N$ atoms, with the set of indices of the lattice sites $\Ls := \{-N+1, -N+2, \ldots, N-1, N \}$ and the periodic extentions $\Omega_\sharp = \R$ and $\Ls_\sharp = \Z$. Let $\varepsilon=1/N$ be the lattice spacing, and let $F>0$ be a macroscopic deformation gradient. Following previous works \cite{MD_ML_Opt_Order_2009, CO_QNL1D_2011, CO_HW_QC_1D_A_Priori_2011, CO_HW_ACC_1D_A_Posteriori} and to avoid unnecessary diffculties with boundaries, we impose periodic and mean zero boundary conditions on the space of displacements and define
\begin{align}
\label{def: atom displacement space}
&\Us^{\eps}:=\{u:\Ls_\sharp \rightarrow\mathbb{R}^{\Z}:v_{\ell+2N}=v_{\ell}~and~\sum_{\ell=-N+1}^{N}v_{\ell}=0\}.
\end{align}
The corresponding admissible set of deformations is defined by
\begin{align}
\label{def: atom deformation set}
&\Ys^{\eps}:=\left\{y:\Ls_\sharp \rightarrow\mathbb{R}^{\Z}:y_{\ell}=\varepsilon F\ell+v_{\ell},v\in\Us^\eps \right\}.
\end{align}

For a map $y \in \Ys^{\eps}$, we define the finite difference operators
\begin{align}
&D_{1}y_\ell:=(y_{\ell+1} - y_{\ell})/{\eps}, D_{2}y_\ell :=(y_{\ell+2} - y_{\ell})/\eps, \nonumber \\
&D_{-1}y_\ell:=(y_{\ell-1}-y_{\ell})/\eps, D_{-2}y_\ell:=(y_{\ell-2} - y_{\ell})/\eps,
\label{def: 1 2 finite differences}
\end{align}
and $D y_\ell:=(D_1 y_\ell, D_2 y_\ell, D_{-1}y_\ell, D_{-2}y_\ell)$. Note $D_{\pm1} y_\ell$ are essentially the forward and backward finite differences and denote the rescaled nearest-neghbour distances whereas $D_{\pm2} y_\ell$ denote the rescaled next-nearest-neighbour distances.

We assume that the atomistic system is modeled by a next-nearest-neighbor many-body site energy, which include models such as the Embedded Atom Method (EAM) \cite{EAM_Daw_Baskes_1984, Tadmor_Miller_MM_2011}, and the internal energy in one period under $y \in \Ys^{\eps}$ is
\begin{equation}
\mathscr{E}_{\a}(y):=\varepsilon \sum_{\ell \in \Ls}V(D y_\ell),
\label{eq: atomistic stored energy}
\end{equation}
where $V \in C^3((0, +\infty]^4;{\R})$.

We include a periodic dead load $f \in \Us^\eps$ in the model to create a nontrivial deformation that mimics the influence of a dislocation or a defect which may appear in higher dimensional cases. The \emph{external energy} (per period) under a given deformation is defined by $-\<f,y\>_{\varepsilon}:=-\varepsilon  \sum_{\ell \in \Ls} f_{\ell}y_{\ell}$ and the \emph{total energy} (per period) under a deformation $y\in\Ys$ is then given by
\begin{equation}
E_{\a}(y):=\mathscr{E}_{\a}(y) - \<f,y\>_{\varepsilon},
\label{eq: atom total energy}
\end{equation}
and the solution we seek for the atomistic problem is
\begin{equation}
\label{eq: A_problem}
y_{\a} \in \argmin E_{\a}(\Ys).
\end{equation}
The following proposition characterizes the first optimality condition of the atomistic problem \eqref{eq: A_problem}. 
\begin{proposition}
Let $y_{\a}$ be a solution to the atomistic problem \eqref{eq: A_problem} and assume $\min_{\ell} Dy^{\a}_\ell>0$. Suppose further that $V$ is differentiable at $y_{\a}$. Then $y_{\a}$ satisfies the following variational problem
\begin{equation}
  \label{A_firstV}
\<\delta\mathscr{E}_{\a}(y),v\>: =\varepsilon\sum_{\ell\in\mathscr{L}}\sigma^{\a}_{\ell}(y) v'_{\ell} = \varepsilon\sum_{\ell\in\mathscr{L}} f_\ell v_\ell = \<f,v\>_{\varepsilon} \ \forall v \in \Us^{\eps},
\end{equation}
where $v_\ell' := D_1 v_\ell$ and the atomistic stress tensor $\sigma^{\a}_{\ell}(y_{\a})$ is given by
\begin{align}
\sigma^{\a}_{\ell}(y):=&\partial_{1}V(Dy_{\ell-1})- \partial_{-1}V(Dy_{\ell}) \notag \\
&+\partial_{2}V(Dy_{\ell-2})+\partial_{2}V(Dy_{\ell-1})
 -\partial_{-2}V(Dy_{\ell})-\partial_{-2}V(Dy_{\ell+1}),
\label{eq: Q^a_l}
\end{align}
where $\partial_{i} V(\gb) := \partial V(\gb) / \partial (g_i)$ for $\gb := (g_1, g_2, g_{-1}, g_{-2})   \in {\R^{+}}^4$.
\end{proposition}

%%%%%%%%%%%%%%%%% Atomistic Model %%%%%%%%%%%%%%%%%%%%%%%%%%%%%%%%%%

%%%%%%%%%%%%%%%%% GRAC Model %%%%%%%%%%%%%%%%%%%%%%%%%%%%%%%%%%

\subsection{Atomistic-to-continuum Coupling}
\label{Sec: A/C Coupling}
We adopt the geometry reconstruction based atomistic-to-continuum coupling (GRAC) method as our coupling model. This coupling method was first proposed in \cite{E_Lu_Yang_Geo_Consitent_2006} for 2D  many-body system with flat coupling interfaces and was extended in \cite{CO_LZ_QC_2D_2011, CO_LZ_Finite_GRAC_2014} for general interfaces. We use the same idea to construct our A/C coupling model for our 1D system.

\subsubsection{The interface energy}
\label{Sec: Interface Energy}
To formulate the coupling method, we first decompose the lattice $\Ls$ into $\As, \Is$ and $\Cs$, where $\As$ denotes the set of lattice sites inside which full atomistic accuracy is required, $\Is$ denotes the set of interface lattice sites such that
\begin{equation}
\Is := \{l \in \Ls \setminus \As| l + j \in \As, j = 1,2,-1,-2 \},
\label{def: interface atoms indices}
\end{equation}
and $\Cs:= \Ls \setminus (\As \cup \Is) $ then denotes the remaining lattice sites. The coupling energy in one period (without coarse-graining) for a $y \in \Ys^\eps$ is then given by
\begin{align}
 \label{eq: stored coupling energy point summation}
\mathscr{E}_{\ac}(y)&=\eps \sum_{\ell \in \As} V(Dy_\ell)+\eps\sum_{\ell \in \Is} V^{\i}_\ell (Dy_\ell)
+\eps\sum_{\ell \in \Cs}V^{\c} (Dy_\ell),
\end{align}
where $V^{\c}$ is the continuum site energy which is obtained by Cauchy-Born rule \cite{MB_KH_Crystal_Latices, Blanc_From_A_to_C_2002} and in our case is defined by
\begin{equation}
V^{\c} (Dy_\ell) = V(D_1 y_\ell, 2D_1 y_\ell, D_{-1} y_\ell, 2 D_{-1} y_\ell).
\label{def: cont. site energy}
\end{equation}
$V^{\i}$ is the reconstructed interface site energy such that the so called \emph{patch test} conditions are satisfied, i.e., $\forall \ell \in \Is$ and $F \in \R$
\begin{equation}
\label{eq: consistency}
V^{\i}_{\ell}(Dy^{F}_\ell)=V(Dy^{F}_\ell), \   \text{ and }
f^{\ac}_\ell(y^F) := \frac{\partial \mathscr{E}_{\ac}}{\partial y_\ell} (y^F) = 0 ,
\end{equation}
where $y^F_\ell = \eps F \ell$ is a uniform deformation under the deformation gradient $F$. The interface site energy we use in the present work is defined by
\begin{align}
\label{eq: interface site energy}
V_{\ell}^{\i}(Dy) =
\left\{
\begin{array}{l l}
V(D_{1}y_\ell,D_{2}y_\ell,D_{-1}y_\ell,2D_{-1}y_\ell),  &\text{ if } \ell +j \in \As \  j = 1, 2, \\
V(D_{1}y_\ell,2D_{1}y_\ell,D_{-1}y_\ell,D_{-2}y_\ell),  &\text{ if } \ell +j \in \As \  j = -1, -2,
\end{array} \right.
\end{align}
and one may easily check that \eqref{eq: interface site energy} satisfies \eqref{eq: consistency}.

\begin{remark}
The construction of $V^\i$ is not unique. The general form of an interface energy is $V_{\ell}^{\i}(Dy) = V(Dy_\ell C_\ell)$, where the coefficient matrix $C_\ell$ is given by 
\begin{equation*}
C_\ell =
\begin{pmatrix}
C_{\ell,1,1}&C_{\ell,2,1}&C_{\ell,3,1}&C_{\ell,4,1} \\
C_{\ell,1,2}&C_{\ell,2,2}&C_{\ell,3,2}&C_{\ell,4,2} \\
C_{\ell,1,-1}&C_{\ell,2,-1}&C_{\ell,3,-1}&C_{\ell,4,-1} \\
C_{\ell,1,-2}&C_{\ell,2,-2}&C_{\ell,3,-2}&C_{\ell,4,-2} 
\end{pmatrix},
%V_{\ell}^{\i}(Dy)=V\Big(\sum_{j=\pm1,\pm2}C_{\ell,1,j}D_{j}y,\sum_{j=\pm1,\pm2}C_{\ell,2,j}D_{j}y,\sum_{j=\pm1,\pm2}C_{\ell,-1,j}D_{j}y,\sum_{j=\pm1,\pm2}C_{\ell,-2,j}D_{j}y\Big),
\end{equation*}
where the $C_{\ell,i,j}$ are the \emph{reconstruction parameters} which are chosen so that \eqref{eq: consistency} is satisfied. We adopt the current form of $V^\i$ simply because it reduces to the QNL method \cite{Shimokawa_QNL_Original_2004, CO_HW_QC_1D_A_Priori_2011} if $V$  consists only pair potentials.  However, we need to note that our analysis \emph{does not} depend on the method we choose. We refer to \cite{E_Lu_Yang_Geo_Consitent_2006, CO_LZ_QC_2D_2011,CO_LZ_Finite_GRAC_2014, CO_2D_Patch_Test} for detail of the geometry reconstruction-based atomistic-to-continuum coupling methods and \cite{Shapeev_Consisten_A_C_1D_2D, Shapeev_ACC_3D, PL_AS_QCP_2009, Li_Luskin_QNL_2012} for different approaches.
\end{remark}

%%%%%%%%%%%%%%%%%%%%%%%%%%%%%%%%%%%%%%%%%%%%%%%%%%%%%%

\subsubsection{The coupling energy with coarse-graining}
We proceed with the decomposition of the computational domain $\Omega$ into the atomistic region $\Omega_{\a}$, the interface region $\Omega_{\i}$ and the continuum region $\Omega_{\c}$ according to $\As$, $\Is$ and $\Cs$ respectively and apply a continuum model to transform the coupling energy from a pointwise summation rule to an integral form in the continuum region and coarse grain the continuum region by the finite element method to further reduce the number of degrees of freedom.

To make the above statement rigorous, we partition $\Omega$ by choosing a small number, say $K$, lattice sites as the finite element nodes and constructing the mesh $\Ts_h = \{T_k\}_{k=1}^K$ on $\Omega$ with the following properties.
\begin{itemize}
\item[(T1)]  With slight abuse of notation, the indices of the nodes are identified with the indices of the lattice sites by $\ell: \{1, \ldots, K\} \rightarrow \Ls$ such that $\ell_k := \ell(k)$ is the index of the lattice site which is also the $k$'th node in $\Ts_h$. We thus have $T_k := [x_{k-1}, x_k]$ and $x_k = \eps \ell_k$ for all $k = 1, \ldots K$. The length of $T_k$ is given by $h_{T_k} := |T_k| = x_k - x_{k-1}$. The number of atoms in a given element $T_k$ is represented by $N_{T_k}$ where $N_{T_k} = \frac{h_{T_k}}{\eps}$.
\item[(T2)]  Only one atomistic region $\Omega_{\a}$ exists in $\Omega$ which is given by $\Omega_{\a}:=(\eps \ell_{K_{1}},\eps \ell_{K_{2}})$ for some $1 < K_1 < K_2 < K$  and $\As = \{\ell_{K_1+1}, \ell_{K_1+2}, \ldots, \ell_{K_2-1}\}$ which implies that $\Ts_h$ has the atomistic resolution in $\Omega_{\a}$, i.e., every lattice site in $\As$ is a finite element node in $\Ts_h$.
\item[(T3)]  The interface region is defined to be $\Omega_{\i}:=[\eps \ell_{K_{1}-1},\eps \ell_{K_{1}}]\cup[\eps \ell_{K_{2}},\eps \ell_{K_{2}+1}]$ and $\Is = \{\ell_{K_1-1},\ell_{K_1},\ell_{K_2}, \ell_{K_2+1}\}$.
\item[(T4)]  The continuum region $\Omega_{\c} = \Omega \backslash (\Omega_{\a} \cup \Omega_{\i})$ is defined by $\Omega_{\c}:=[-\eps N,\eps \ell_{K_{1}-1})\cup (\eps \ell_{K_{2}+1},\eps N ]$ and $\Cs = \{1, \ldots, \ell_{K_{1}-2}, \ell_{K_{2}+2}, \ldots, \ell_K\}$.
\item[(T5)] The first element adjacent to the interface in the continuum region has length $\eps$ which implies that the first atom outside the atomistic and interface region is a node of $\Ts_h$.
\end{itemize}

The structure of the mesh is illustrated in Fig. \ref{Fig_mesh}.
\begin{figure}[h!]
\centering
\includegraphics[width=\textwidth]{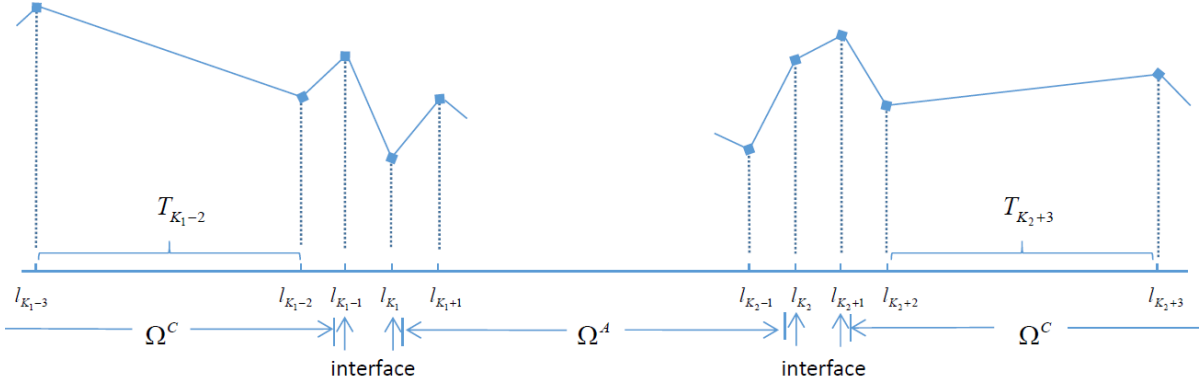}
\caption{mesh strategy for the model problem}
\label{Fig_mesh}
\end{figure}

We also define $\Ts_h^{\sharp}$ and  $\Omega^{\sharp}_{\ac}$ to be periodic extensions of $\Ts_h$ and $\Omega_{\ac}$ for $\ac \in \{\a, \i, \c\}$ such that  $x_{k+K} = 2 + x_k$ for all $k \in \Z$ and $\Omega^{\sharp}_{\ac} : = (\Omega_{\ac} +2 \Z)$. The set of the indices of the nodes in $\Ts_h$ in different regions are defined by $\mathscr{K}^{\a}:=\{K_{1}+1,...,K_{2}-1\}$, $\mathscr{K}^{\i}:=\{K_{1}-1,K_{1},K_{2},K_{2}+1\}$ and $\mathscr{K}^{\c}:=\{1, \ldots, K_{1}-2,K_{2}+2, \ldots,K\}$ respectively and we define $\mathscr{K}^{\ac} :=\mathscr{K}^{\c}\cup\mathscr{K}^{\a}\cup\mathscr{K}^{\i}$.

%%%%%%%%%%%%%%%%%%%%%%%%%%%%%%%%%%%%%%%%%%%%%%%%%%%%%%

The coarse-grained space of displacement is defined to be
\begin{equation}
\Us^{h}:=\left\{u_h \in \Ps_1(\Ts_h):   u_h(x+1) = u_h(x) \text{ and } \int_\Omega u_h(x) \dx = 0 \right\},
\end{equation}
and the coarse-grained admissible set of deformation
\begin{equation}
\Ys^{h}:=\left\{ y_h  \in \Ps_1(\Ts_h): y_h - Fx \in \Us_h \right\},
\end{equation}
where $\Ps_1(\Ts_h)$ denotes the space of continuous piecewise affine functions with respect to $\Ts_h$.

\begin{remark}
\label{remark: solution spaces}
Note that we have changed our solution sets from sets of pointwise defined functions to those of continuous  piecewise affine functions. We emphasize that the two definitions are equivalent given the values of the function on the nodes and we therefore take the two point of views liberally for functions in $\Us^{\cp}$ and $\Ys^{\cp}$ for $\cp \in \{\eps, h\}$. Another observation is that $\Us^h \subset \Us^\eps$ and $\Ys^{h} \subset \Ys^{\eps}$ which results from the constrain that all nodes in $\Ts_h$ are on the lattice sites. Such property will exclude the nonconformity of the solutions spaces to enable us to keep the presentation simple and focus on the main issues.
\end{remark}

Having the finite element discretization, we are able to transform the coupling energy from a summation rule \eqref{eq: stored coupling energy point summation} to an integral form. We first define the \emph{Cauchy-Born energy functional} for a given $y \in W^{1,\infty}(\Omega;\R)$
\begin{equation}
\Es_{\c}(y) := \int_{\Omega} W(\nabla y) \dx,
\end{equation}
where $W(F) := V(F,2F,-F,-2F) \in C^3(0,+\infty)$ is the \emph{Cauchy-Born stored energy density}\cite{CO_LZ_QC_2D_2011}.

Let $vor(\ell)$ be the Voronoi cell (see \cite{CO_LZ_Finite_GRAC_2014}) associated with $\ell$ (obviously $|vor(\ell)|\equiv\varepsilon$). For $\ell\in\Omega_{\i}$, we choose a modified interface site potential $V^{\i}_{\ell}$ and the effective cell $v^{\i}_{\ell}= vor(\ell)$ associated with $\ell$  and define the effective volume associated with $\ell$ as $\omega_{\ell}^{\i}:=|v^{\i}_{\ell}|$. In addition, for each element $T\in\mathscr{T}_{h}$ we define the effective volume $\omega_{T}:=|T\backslash\bigcup_{\ell\in\Omega^{\i}}v^{\i}_{\ell}|$. Letting $y^h_\ell := y_h(\eps \ell)$ we redefine the A/C coupling energy for $y_h \in \Ys^h$ to be
\begin{align}
\mathscr{E}_{\ac}(y)&=\varepsilon\sum_{\ell\in \As}V(Dy^h_\ell)+\varepsilon\sum_{\ell \in \Is}\omega_{\ell}^{\i}V_{\ell}^{\i}(Dy^h_\ell)
+\sum_{T\in\mathscr{T}_{h}}\omega_{T}W(\nabla {y^h}|_{T})\notag\\
&=\varepsilon\sum_{\ell\in \As}V(Dy^h_\ell)+\varepsilon\sum_{\ell \in \Is}V_{\ell}^{i}(Dy^h_\ell)
+\int_{\Omega_{\c} \backslash\bigcup_{\ell\in\Omega_{\i}}v^{\i}_{\ell}}W(\nabla {y^h}|_{T})dx.
 \label{eq: coupling energy integral}
\end{align}

\begin{remark}
In a pointwise summation rule \eqref{eq: stored coupling energy point summation}, the energy is associated with the Voronoi cell of an atom whereas in an integral form the energy is locally defined. Since the interface energy is associated with the interface atoms, certain amount of energy should be subtracted from the energy of the adjacent continuum element and hence the effective volumes appear in the formulation \eqref{eq: coupling energy integral} as an correction to keep the total energy consistent.
By (T5) we let the lengths of elements adjacent to the interface to be $\varepsilon$, i.e., $\omega_{T_{K_{1}-1}}=\omega_{T_{K_{2}+2}}=\frac{1}{2}\varepsilon$, which keeps the correction local and simplifies our analysis.
\end{remark}

\subsubsection{Total energy and its variation}
Given $f \in \Us^\eps$ and $y_h \in \Ys^h$, we define the external energy to be $-\< f, y_h\>_\eps : = -\eps \sum_{\ell = -N+1}^{N} f_\ell y^h_\ell$. Upon defining the set of indices of lattice sites inside and on the right boundary of the element $T_k$ by
\begin{equation}
\Ls_{T_k} := \{\ell_{k-1} + 1, \ldots, \ell_k\},
\end{equation}
and the indication function  $\chi(T;\ell)$ such that
\begin{equation}
\chi(T;\ell) =
\left\{
\begin{array}{l l}
1, & \ell \in \Ls_T, \\
0, & \text{ otherwise},
\end{array} \right.
\end{equation}
we are able to associate the external energy with the nodal values of $y_h$ (see the proof in Appendix \ref{Proof_2})
\begin{equation}
 \label{projected_force}
-\<f,y_h\>_{\varepsilon}=-\sum_{k\in\mathscr{K}^{\ac}}\eps \bar{f}_{k}y^h_{\ell_{k}},
\end{equation}
where the \emph{nodewise force} $\bar{f}_{k}$ is defined by
\begin{align*}
\bar{f}_{k}=
\left\{
\begin{array}{l l}
f_{\ell_{k}},  & k\in\mathscr{K}^{\a}\cup\mathscr{K}^{\i},\\
%%%%%%%%%%%%%
\sum_{\ell \in \Ls_{T_k}} \frac{\ell-\ell_{k-1}}{\ell_k - \ell_{k-1}} f_{\ell},  & k=K_{1}-2,\\
%%%%%%%%%%%%%
\sum_{\ell \in \Ls_{T_k}\cup\{K_2+2\}} (1-\frac{\ell-\ell_{k}}{\ell_{k+1} - \ell_k})f_{\ell},  & k=K_{2}+2,\\
%%%%%%%%%%%%%
\sum_{\ell \in \Ls_{T_k} \cup \Ls_{T_{k+1}}}[\frac{\ell-\ell_{k-1}}{\ell_k - \ell_{k-1}}\chi(T_{k};\ell)+(\frac{\ell_{k+1}-\ell}{\ell_{k+1} - \ell_k})\chi(T_{k+1};\ell)]f_{\ell},  & otherwise.
\end{array} \right.
\end{align*}

The total energy for the coupling model with coarse graining is given by
\begin{equation}
E_{\ac}(y_h):=\mathscr{E}_{\ac}(y_h)-\<f,y_h\>_{\varepsilon},
\end{equation}
and we wish to compute
\begin{equation}
\label{eq: QC_problem}
y_{\ac}\in \argmin E_{\ac}(\Ys_{h}).
\end{equation}

The following proposition characterizes the first optimality condition of the a/c coupling problem \eqref{eq: QC_problem}. 
%the first order optimality condition for \eqref{eq: QC_problem} in the variational form is

\begin{proposition}
  \label{QC_firstV}
Let $y_{\ac}$ be a solution to the a/c coupling problem \eqref{eq: QC_problem} and assume $\min_k (\nabla y_h|_{T_k}) >0$. Suppose further that $V$ is differentiable at $y_{\ac}$. Then there exists a unique elementwise a/c coupling stress tensor $\bar{\sigma}^{ac}_{k}(y_h)$ whose detailed formulation is given in \ref{app: coupling model stress tensor}, such that $y_{\ac}$ satisfies the following variational problem
 \begin{equation}
\<\delta\mathscr{E}_{\ac}(y_{h}),v_h\>: =\sum_{k \in \Ks^{ac}}h_{T_k} \bar{\sigma}^{\ac}_{k}(y_h) \nabla {v_h}|_{T_k} =
\sum_{{k}\in\mathscr{K}^{\ac}}\bar{f}_{{k}}v^h_{\ell_{k}} =\<f,v_h\>_{\varepsilon} \ \forall v_h \in \Us^{h}.
 \end{equation}
\end{proposition}Moreover, using the identity which is a consequence of the 1D setting of our problem
\begin{equation}
\label{eq: gradient identity}
\nabla v_h|_{T} = \frac{\varepsilon}{h_{T}}\sum_{\ell \in \Ls_T} Dv_{\ell},
\end{equation}
we have the equivalent form for the first variation of the coupling model associated with the lattice
\begin{equation}
\label{eq: variational equiv}
%\label{eq: 1D AC equivalence}
\<\delta\mathscr{E}_{\ac}(y_h),v_h\> =\sum_{k \in \Ks^{ac}}h_{T_k}\bar{\sigma}^{\ac}_{k}(y_h) \nabla v_h|_{T_k} = \varepsilon\sum_{\ell\in\mathscr{L}}\sigma^{\ac}_{\ell}(y_h)Dv^h_{\ell}  \ \forall v_h\in\Us^h,
\end{equation}
where
\begin{align}
\label{eq: stress tensor ac atom wise}
\sigma^{\ac}_{\ell}(y_h)=
\left\{
\begin{array}{l l}
\bar{\sigma}^{\ac}_{k}(y_h),   &k \in \Ks^a \cup \Ks^i \cup \{K_2 + 2\} \text{ and } \ell = \ell_k,\\
W'(\nabla y_{h}|_{T_{k}}),   &k \in\Ks^c \setminus \{K_{2}+2\} \text{ and } \ell\in \Ls_{T_{k}}. \\
\end{array} \right.
\end{align}
%which will be frequently used in the subsequent analysis.

\begin{remark}
We do not approximate the external energy by a quadrature rule to avoid substantial technical difficulty for the analysis of the efficiency. However, We note that the use of a quadrature rule (for example the trapezium rule where $\<f, y_h\>_\eps$ is approximated by $\sum_{k \in \Ks^{ac}} f_{\ell_k} y^h_{\ell_k}$) has only marginal effect in the error estimates which is negligible in computation. We refer to Section 3.3 and 3.4 of \cite{CO_HW_ACC_1D_A_Posteriori} for a thorough discussion.
%of the error estimates for the external energy using quadrature rule, .
\end{remark}

\subsection{Notation and Assumptions}
\label{Sec: Notation}
Before we give the detailed analysis, we fix some notation and list the assumptions that will be commonly used in the rest of the paper. Further notation will be defined as the analysis proceeds.

\subsubsection{Notation for lattice functions} Let $\mathcal{D}$ be a subset of $\Z$. For a vector $v \in \R^\Z$, we define
\begin{align*}
\Vert v \Vert_{\ell^{p}_{\eps}(\mathcal{D})} :=  \left\{
\begin{array}{l l}
\Big(\sum_{\ell \in \mathcal{D}} \eps |v_\ell|^p
\Big)^{1/p},  & 1 \le p < \infty,\\
\max_{\ell \in \mathcal{D}} |v_\ell|,  & p = \infty.
\end{array} \right.
\end{align*}
If the label $\mathcal{D}$ is omitted, we understand this to mean
$\mathcal{D} = \{-N+1, \dots, N\}$.

We define the first order discrete derivatives $v_\ell' := (v_\ell -
v_{\ell-1}) / \eps$ for $v \in \Us^{\eps}$and equip the space $\Us^\eps$ with the discrete Sobolev norm
\begin{displaymath}
  \|v\|_{\Us^{1,2}} := \Vert v' \Vert _{\ell^p_\eps} \quad \text{ for }
  v \in \Us^\eps.
\end{displaymath}
The norm on the dual space $(\Us^\eps)^*$ is defined by
\begin{displaymath}
  \| T \|_{\Us^{-1,2}} := \sup_{\substack{v \in \Us^\eps \\ \| v
    \|_{\Us^{1,2}}=1 }} T[v].
\end{displaymath}

\subsubsection{Sets of indices of nodes and elements} We define the set which contains the indices of the elements in the continuum region that are not adjacent to the interface region as
\begin{align}
\Ks^c_{\Ts_h} : = \Ks^c \setminus \{K_2 + 2\};
\end{align}
the set which contains the indices of elements in the continuum region that are 'further inside' the continuum region as
\begin{align}
\mathring{\Ks}^c_{\Ts_h}  := \Ks^c_{\Ts_h}  \setminus \{K_1-2, K_2+3\};
\end{align}
and the set which contains the indices of the nodes that are not adjacent to the interface nodes as
\begin{align}
\mathring{\Ks}^c : = \Ks^c \setminus \{K_1-2, K_2+2\}.
\end{align}

\subsubsection{Assumptions on the interaction potential}
We make the following assumptions on the second order partial derivatives of the interaction potential $V$. Such assumptions play important roles in proving the equivalence of the error estimator based on residuals and that based on gradient recovery and they essentially reflect the feature of nearest neighbour dominating.

Let $E \subset \R$ be the closed interval such that
\begin{equation}
\inf_{k \in \Ks^c \cup \Ks^i \cup \{ K_1+1\}} \nabla y_{\ac}|_{T_k} = \inf E \text{ and } \sup_{k \in \Ks^c \cup \Ks^i \cup \{ K_1+1\}} \nabla y_{\ac}|_{T_k} = \sup E.
\end{equation}
Let $\gb := (g_1, g_2, g_3, g_4)$ for $g_i \in \R$ and $E^{\times 4} : = E \times 2E \times -E \times -2E$. Upon defining 
\begin{align}
\Ss^{\NNN_1}: =& \{\pm(1,-1),\pm(2,2)\}, \label{def: sets of interactions 1} \\
 \Ss^{\NNN_2}: =& \{\pm(1,2),\pm(2,1),\pm(-2,1),\pm(-1,2)\}, \label{def: sets of interactions 2} \\ 
\Ss^{\NNN_3}:=& \{\pm(2,-2) \} \label{def: sets of interactions 3},
\end{align} 
we make the following assumption which is tested a posteriorily  for some typical potentials in Section \ref{Sec: numerical results} and illustrated in Figure \ref{Fig:derivative_ratios}.

\begin{assumption}
\label{ass: assump_on_derivs_V}
The second order derivatives of $V$ satisfy
\begin{align}
\sup_{\substack{\gb \in E^{\times 4},\\  (i,j) \in \Ss^{\NNN_1} }} |\partial_{ij}V(\gb)| \le \frac{1}{50} \inf_{\substack{\gb \in E^{\times 4}, \\ i = \pm1}} |\partial_{ii}V(\gb)|:= & \frac{1}{50} m_2^{\NN} \nonumber \\
\le & \frac{1}{50} \sup_{\substack{\gb \in E^{\times 4}, \\i = \pm1}} |\partial_{ii}V(\gb)|:= \frac{1}{50} M_2^{NN},
\label{eq: assump_on_derivs_V_1}
\end{align}
\begin{align}
\sup_{\substack{ \gb \in E^{\times 4}, \\ (i,j) \in \Ss^{\NNN_2}} } |\partial_{ij}V(\gb)|
\le &
\frac{1}{10}\inf_{\substack{ \gb \in E^{\times 4},\\ (i,j) \in \Ss^{\NNN_1} }} |\partial_{ii}V(\gb)|, \notag  \\
\text{and~}\sup_{\substack{ \gb \in E^{\times 4}, \\ (i,j) \in \Ss^{\NNN_3}} } |\partial_{ij}V(\gb)|
\le &
\frac{1}{10} \inf_{\substack{ \gb \in E^{\times 4}, \\ (i,j) \in \Ss^{\NNN_2}} } |\partial_{ij}V(\gb)|.
\label{eq: assump_on_derivs_V_2}
\end{align}
We also assume that
\begin{equation}
\partial_{ii}V(\gb) > 0, i = \pm1, \text{~and~} \partial_{jj}V(\gb) < 0, j=\pm2~ \forall  \gb\in E^{\times 4}.
\label{eq: assump_on_derivs_V_3}
\end{equation}
\end{assumption}

%%%%%%%%%%%%%%%%% GRAC Model %%%%%%%%%%%%%%%%%%%%%%%%%%%%%%%%%

%%%%%%%%%%%%%%%%% Residual Based Error Estimator %%%%%%%%%%%%%%%%%%%%%%%%
\section{Residual based error estimator}
\label{Sec: Residual Based EE}
In this section, we will derive the residual based a posteriori error estimator for the GRAC method for the many-body next-nearest-neighbour system. Such error estimator has been derived for the QNL and ACC method for pair potential systems in \cite{CO_QNL1D_2011} and  \cite{CO_HW_ACC_1D_A_Posteriori} respectively. Though the analysis is similar, we nevertheless include it here for the completeness and will quote related results in the previous works when necessary.

We first insert the QC solution $y_{\ac}$ into the weak formulation of the atomistic problem to obtain the residual. Using the identity \eqref{QC_firstV} and letting $v_h := I_h v \in \Us^h$ to be the pointwise interpolant of $v \in \Us^\eps$ such that $v^h_{\ell_k} = v_{\ell_k}$, we obtain the \emph{residual operator} $R \in (\Us^\eps)^*$ such that
\begin{align}
R[v]:=&\<\delta\mathscr{E}_{\a}(y_{\ac}),v\> - \<f,v\>_{\varepsilon}\notag\\
%%%%%%%%%%%%%%%%%%%
=&\big[ \<\delta\mathscr{E}_{\a}(y_{\ac}),v\> - \<f,v\> \big] - \big[ \<\delta\mathscr{E}_{\ac}(y_{\ac}),v_h\> - \<f,v_h\> \big]\notag \\
%%%%%%%%%%%%%%%%%%%
=&\big[\<\delta\mathscr{E}_{\a}(y_{\ac}),v\> - \<\delta\mathscr{E}_{\ac}(y_{\ac}),v\>\big]\notag\\
%%%%%%%%%%%%%%%%%%%
&+\big[\<\delta\mathscr{E}_{\ac}(y_{\ac}),v\>-\<\delta\mathscr{E}_{\ac}(y_{\ac}),v_h\>+\<f,v\>_{\varepsilon}-\<f,v_h\>_{\varepsilon}\big], \notag\\
%%%%%%%%%%%%%%%%%%%
=&: R_{\mo}[v] + R_{\cg}[v],
\end{align}
where we separate the residual operator into $R_{mo}$ and $R_{cg}$ which correspond to the model residual and the corse-graining residual respectively. We then estimate $R_{\mo}$ and $R_{\cg}$ separately.

\subsection{Model Residual}
\label{Sec: 2.1}
We begin our analysis for the model residual by defining
\begin{equation}
\label{def: node wise residual R}
R^{\mo}_{\ell}:=\sigma^{\a}_{\ell}(y_{\ac})-\sigma^{\ac}_{\ell}(y_{\ac}),
\end{equation}
which corresponds to the discrepancy of the stress tensors of different models. We estimate the model residual in the following theorem.

\begin{theorem}
Let $y_h \in \Ys^h$ such that $\min_k \nabla y_h|_{T_k} >0$. With the assumption that the size of the element whose index is in $\Ks^c_{\Ts_h}$ to be larger than or equal to $6\eps$, which is purely for the sake of the simplicity of presentation, the model residual is estimated by
\begin{equation}
\label{eq: model err est}
\|R_{\mo}\|_{\Us^{-1,2}} \le \Big\{\sum_{k\in\Ks^c}(\eta^{\mo}_{{k}})^{2}\Big\}^{\frac{1}{2}} \equiv \Big\{\sum_{k\in\Ks^c_{\Ts_h}}(\eta^{\mo}_{T_{k}})^{2}\Big\}^{\frac{1}{2}} =: \eta^{\mo},
\end{equation}
where the \emph{nodewise} upper bound of the model residual is given by
\begin{align}
\label{eq: node wise err est}
\eta^{\mo}_{{k}}=
\left\{
\begin{array}{l l}
\Big\{ \eps \sum_{\ell=\ell_{k}-2}^{\ell_{k}+3}\big(R^{\mo}_{\ell}\big)^{2}\Big\}^{\frac{1}{2}},&k \in \mathring{\Ks}^c, \\
%%%%%New Line%%%%%%%%%%
\Big\{ \eps \sum_{\ell=\ell_{k}-2}^{\ell_{k}+4}(R^{\mo}_{\ell})^{2} \Big\}^{\frac{1}{2}},&k=K_{1}-2,\\
%%%%%New Line%%%%%%%%%%
\Big\{\eps \sum_{\ell=\ell_{k}-3}^{\ell_{k}+3}(R^{\mo}_{\ell})^{2}  \Big\}^{\frac{1}{2}},&k=K_{2}+2,\\
 \end{array} \right.
\end{align}
and the \emph{elementwise} upper bound of the model residual is given by
\begin{align}
\label{eq: ele wise err est}
\eta^{\mo}_{T_{k}}=
\left\{
\begin{array}{l l}
%%%%%%%%%%%%%%%%%%%
\Bigg\{ \frac{1}{2}\eps \Big[\sum_{\ell=\ell_{k-1}-2}^{\ell_{k-1}+3}(R^{\mo}_{\ell})^{2}+\sum_{\ell=\ell_{k}-2}^{\ell_{k}+3}(R^{\mo}_{\ell})^{2}\Big] \Bigg\}^{\frac{1}{2}}, &k\in \mathring{\Ks}^c_{\Ts_h} ,\\
%%%%%%%%%%%%%%%%%%%%
\Bigg\{ \frac{1}{2} \eps \Big[\sum_{\ell=\ell_{k-1}-2}^{\ell_{k-1}+3}(R^{\mo}_{\ell})^{2}\Big]
+ (\eta^{\mo}_k)^2 \Bigg\}^{\frac{1}{2}} ,&k=K_{1}-2,\\
%%%%%%%%%%%%%%%%%%%%
\Bigg\{ (\eta^{\mo}_{k-1})^2+\frac{1}{2} \eps \Big[ \sum_{\ell=\ell_{k}-2}^{\ell_{k}+3}(R^{\mo}_{\ell})^{2}\Big] \Bigg\}^{\frac{1}{2}}, &k=K_{2}+3.
\end{array} \right.
\end{align}

\end{theorem}

\begin{remark}
\begin{enumerate}
\item $\eta^{\mo}$ is a reminiscent of the flux (or stress) jump terms that occur in the classical residual based error estimator for elliptic equations, but has a different origin: it results from the model approximation rather than just the finite element discretization.
\item $\eta^{\mo}_k$'s are often used in the analysis, whereas $\eta^{\mo}_{T_{k}}$'s are used in computation for the adaptivity of the mesh. We also note that the residual on the interface is included in the second element in the continuum region as a result of our mesh structure and not being able to further refine the elements whose sizes are equal to the lattice spacing $\eps$.
\end{enumerate}
\end{remark}

\begin{proof}
By \ref{A_firstV} and \ref{eq: variational equiv}, we have:
\begin{equation}
R_{\mo}[v]=\varepsilon\sum_{\ell\in\Ls}(\sigma^{\a}_{\ell}(y_{\ac})-\sigma^{\ac}_{\ell}(y_{\ac}))v_{\ell}'= \varepsilon\sum_{\ell\in\Ls}R^{\mo}_{\ell}v_{\ell}'.
\end{equation}

We notice that $R^{\mo}_\ell \equiv 0$ in the 'central atomistic region' when $\ell = \ell_{K_1}+3, \ldots, \ell _{K_2} - 2$ and inside each finite element $T_{k}$ when $\ell \in \Ls_{T_{k}} \setminus \{\ell_{k-1}+1,\ell_{k-1}+2,\ell_{k-1}+3\}\cup\{\ell_{k}-2,\ell_{k}-1,\ell_{k}\}$. Therefore, by Cauchy-Schwarz inequality, we have
\begin{align}
R_{\mo}[v]=& \sum_{k\in \mathring{\Ks}^c } \eps \sum_{\ell=\ell_{k}-2}^{\ell_{k}+3} R^{\mo}_{\ell}v_{\ell}' + \eps \sum_{\ell=\ell_{K_{1}-2}-2}^{\ell_{K_{1}+2}}R^{\mo}_{\ell}v_{\ell}' + \varepsilon\sum_{\ell=\ell_{K_{2}-1}}^{\ell_{K_{2}+2}+3}R^{\mo}_{\ell}v_{\ell}' \nonumber  \\
\le &\Big\{ \sum_{k\in\mathring{\Ks}^c} \eps \sum_{\ell=\ell_{k}-2}^{\ell_{k}+3} (R^{\mo}_{\ell})^2
+  \eps\sum_{\ell=\ell_{K_1-2}-2}^{\ell_{K_1}+2} (R^{\mo}_{\ell})^{2} +  \eps \sum_{\ell=\ell_{K_{2}-1}}^{\ell_{K_{2}+2}+3}(R^{\mo}_{\ell})^{2}  \Big\}^{\frac{1}{2}} \|v'\|_{\ell_{\varepsilon}^{2}}.
\end{align}

Regrouping the residuals with respect to nodes and elements, we obtain the stated results.
\end{proof}

\subsection{Coarse-graining Residual}
We then consider the coarse-graining residual
\begin{equation}
R_{\cg}[v]=\<\delta\mathscr{E}_{\ac}(y_{\ac}),v\>-\<\delta\mathscr{E}_{\ac}(y_{\ac}),v_h\>+\<f,v\>_{\varepsilon}-\<f,v_h\>_{\varepsilon},
\end{equation}
whose estimate is given in the following theorem.
\begin{theorem}
Let $y_h \in \Ys^h$ such that $\min_k \nabla y_h|_{T_k} >0$; then
\begin{align}
\|R_{\cg}\|_{\Us^{-1,2}} \le& \big\{ \sum_{k\in\Ks^c_{\Ts_h}} (\eta^{\cg}_{T_k})^2  + \frac{1}{2}\sum_{k\in\Ks^c_{\Ts_h}}  h_{T_k}^2(\|f - \bar{f}_{T_k} \|_{\ell_{\varepsilon}^{2}(\Ls_{T_{k}})})^2 \big\}^{\frac{1}{2}}\notag\\
= & \big\{ \sum_{k\in\Ks^c} (\eta^{\cg}_{k})^2  + \frac{1}{2}\sum_{k\in\Ks^c_{\Ts_h}}  h_{T_k}^2(\|f - \bar{f}_{T_k} \|_{\ell_{\varepsilon}^{2}(\Ls_{T_{k}})})^2 \big\}^{\frac{1}{2}}
%\nonumber \\
%=: & \big\{ (\eta^{\cg})^2  + \frac{1}{2}\sum_{k\in\Ks^c_{\Ts_h}}  h_{T_k}^2(\|f - \bar{f}_{T_k} \|_{\ell_{\varepsilon}^{2}(\Ls_{T_{k}})})^2 \big\}^{\frac{1}{2}},
\label{eq: cg residual}
\end{align}
where $\bar{f}_{T_{k}} \in \R$ is a certain average of $f$ on $T_k$ and
\begin{align}
(\eta^{\cg}_{T_k})^2:=(\frac{1}{\sqrt{2}}h_{T_k}\|\bar{f}_{T_k} \|_{\ell_{\varepsilon}^{2}(\Ls_{T_{k}})})^2 =  \frac{1}{2}h_{T_k}^2 \eps \sum_{\ell \in \Ls_{T_k}} (\bar{f}_{T_{k}})^2,\notag\\
\label{eq: cg residual element}
\end{align}
and
\begin{align}
\eta^{\cg}_{k}=
\left\{
\begin{array}{l l}
%%%%%%%%%%%%%%%%%%%
\frac{\sqrt{2}}{2}\sqrt{(\eta^{\cg}_{T_k})^{2} + (\eta^{\cg}_{T_{k+1}})^{2}}, & k\in \mathring{\Ks}^c,\\
\frac{\sqrt{2}}{2}\eta^{\cg}_{T_{K_{1}-2}}, & k = K_{1}-2, \\
\frac{\sqrt{2}}{2}\eta^{\cg}_{T_{K_{2}+3}}, & k = K_{2}+2.
\end{array} \right.
\label{eq: cg residual node}
\end{align}
We also define $(\eta^{\cg})^2 := \sum_{k\in\Ks^c_{\Ts_h}} (\eta^{\cg}_{T_k})^2 = \sum_{k\in\Ks^c} (\eta^{\cg}_{k})^2$ for later usage.
\end{theorem}

\begin{proof}
By the identity in \eqref{eq: variational equiv} we have
\begin{equation}
\<\delta\mathscr{E}_{\ac}(y_{\ac}),v\>=\<\delta\mathscr{E}_{\ac}(y_{\ac}),v_h\>.
\end{equation}
We thus only have to analyze the coarse-graining residual of the external force
\begin{align}
R_{\cg}[v]= \varepsilon \sum_{\ell\in\Cs}f_{\ell}(v_{\ell}-v^{h}_{\ell})
\leq& \sum_{k\in\Ks^c_{\Ts_h}}\big\{\varepsilon\sum_{\ell\in\Ls_{T_{k}}}f_{\ell}^{2}\big\}^{\frac{1}{2}}\big\{\varepsilon\sum_{\ell\in\Ls_{T_{k}}}(v_{\ell}-v^h_{\ell})^{2}\big\}^{\frac{1}{2}} \nonumber \\
\le &\bigg\{\sum_{\Ks^c_{\Ts_h}}(\frac{1}{2}h_{T_k}\|f\|_{\ell_{\varepsilon}^{2}(\Ls_{T_{k}})})^2 \bigg\}^{\frac{1}{2}}\|v'\|_{\ell_{\varepsilon}^{2}},
\end{align}
where the discrete Poincar\'{e} inequality $\|v-I_{\varepsilon}v_{h}\|_{\ell_{\varepsilon}^{2}(\Ls_{T_{k}})}\leq\frac{1}{2}h_{T_k}\|v'\|_{\ell_{\varepsilon}^{2}(\Ls_{T_{k}})}$ has been applied (c.f. \cite{CO_ES_1D_QC_2008}). Upon introducing $\bar{f}_{T_k}$ and applying the triangle inequality and the inequality of arithmetic means, we obtain the stated result.
\end{proof}

\begin{remark}
We postpone our choice of $\bar{f}_{T_k}$ to Section \ref{Sec: Numeric Exp} so that it adapts to the external load we apply in our numerical experiments to make the \emph{data oscillation} $h_{T_k}\|f - \bar{f}_{T_k} \|_{\ell_{\varepsilon}(\Ls_{T_{k}})}$ be a higher order term compare with $\eta_k$ which will be proved in Appendix \ref{sec: higher order of the data oscillation} (see also Remark 1.7 of \cite{VerfurthAPosteriori}).
\end{remark}

\subsection{Stability and error estimate}
We need an a posteriori stability condition to give the residual based a posteriori error estimator. However, such condition has been derived and discussed in depth in \cite{CO_HW_ACC_1D_A_Posteriori, JH_XL_HW_a_Post_Contr_Mo_Adap_17, HW_ML_PL_LZ_2D_A_Post} whose detailed formulation is of little relevance to the problem we consider. Therefore, here we just assume there exists an a posteriori stability constant $c_{\a}$ which depends on the computed solution $y_h$ such that
\begin{align}
c_{\a}(y_h) \| y'_{\a} - y'_{\ac}\|^2_{\ell_{\varepsilon}^{2}} \le &\<\delta\mathscr{E}_{\a}(y_{\ac}),y_{\ac} - y_{\a}\> -  \<\delta\mathscr{E}_{\a}(y_{\a}),y_{\ac} - y_{\a}\> \nonumber \\
=  & \<\delta\mathscr{E}_{\a}(y_{\ac}),y_{\ac}  - y_{\a}\> - \<f,y_{\ac}  - y_{\a}\>_{\eps} \nonumber \\
=  & R[y_{\ac}  - y_{\a}].
\end{align}
Consequently, we have the a posteriori error estimate
\begin{align}
 \label{eq: total error est}
\|y'_{\a} -  y'_{\ac}\|_{\ell_{\varepsilon}^{2}}\le & \frac{1}{c_{\a}(y_h)}\big\{(\eta^{\mo})^2 + (\eta^{\cg})^2+\frac{1}{2}\sum_{k\in\Ks^c_{\Ts_h}} h_{T_k}^2(\|f - \bar{f}_{T_k} \|_{\ell_{\varepsilon}^{2}(\Ls_{T_{k}})})^2 \big]\big\}^{\frac{1}{2}}.
\end{align}

%%%%%%%%%%%%%%%%% Residual Based Error Estimator %%%%%%%%%%%%%%%%%%%%%%%%

%%%%%%%%%%%%%%%%% Efficiency%%%%%%%%%%% %%%%%%%%%%%%%%%%%%%%%%%%%

\section{Efficiency of The Residual Based Error Estimator}
\label{Sec: Efficiency}
In this section, we will show the residual based error estimator, up to a constant and data oscillation, provides a lower bound for the true error locally.

\subsection{Efficiency of the coarse-graining residual}
\label{Sec: Eff_EF}
We begin with the efficiency of the coarse-graining residual. The analysis closely follows that for the efficiency of the residual based error estimator for Poisson equation (c.f. \cite[Chapter 1.2]{VerfurthAPosteriori}). However, we need to make certain modifications and assumptions due to the discrete and the nonlocal features of our problem.

We first consider the elements whose sizes are greater than or equal to $8\varepsilon$. We define the \emph{discrete element bubble function} $b^{T_k} \in \Us^{\eps}$ (c.f. Chapter 1.1 in \cite{VerfurthAPosteriori}) by
\begin{align}
\label{def: element bubble function}
b^{T_{k}}_\ell:=
\left\{
\begin{array}{l l}
4\lambda^{T_{k},1}_\ell \lambda^{T_{k},2}_\ell, &\ell\in \{\ell_{k-1}+4,\ldots, \ell_{k}-3\},\\
0, &otherwise,
\end{array} \right.
\end{align}
where $\lambda^{T_{k},1}_\ell=\frac{\ell_{k}-3-\ell}{\ell_{k}-\ell_{k-1}-6}$ and $\lambda^{T_{k},2}_\ell=\frac{\ell-\ell_{k-1}-3}{\ell_{k}- \ell_{k-1}-6}$. We note that the support of $b^{T_k}$ is only on 'shrunken' $T_k$ that contains the set of atoms
\begin{equation}
\mrLs_{T_k}:=\{\ell_{k-1}+4,\ldots, \ell_{k}-3\},
\end{equation}
and such retraction of the bubble function guarantees the efficiency estimate holds precisely in $T_k$ which will commented after the proof of \ref{thm: local efficiency of the cg residual}. We first introduce the properties of $b^{T_k}$ whose proofs are given in Appendix \ref{Proof Pro}.
\begin{proposition}
\label{thm: property of elem bubble func}
The following estimates hold for be the discrete element bubble function $b^{T_k}$ defined in \eqref{def: element bubble function} :
 \begin{gather}
 |b^{T_k}_\ell | \leq 1, \forall \ell \in \mathring{\Ls}_{T_k}, \label{Pro_1_1}\\
 \|(b^{T_k})'\|_{\ell_{\varepsilon}^{2}(\mrLs_{T_k})} = C_1\sqrt{1-(\frac{\eps}{\mathring{h}_{T_k}})^2} (\mathring{h}_{T_k})^{-\frac{1}{2}} \le C_1 (\mathring{h}_{T_k})^{-\frac{1}{2}}   ,\label{Pro_1_2}\\
\|b^{T_k}\|_{\ell_{\varepsilon}^{2}(\mrLs_{T_k})}= C_2 \sqrt{1-(\frac{\eps}{\mathring{h}_{T_k}})^4}(\mathring{h}_{T_k})^{\frac{1}{2}} \le C_2  (\mathring{h}_{T_k})^{\frac{1}{2}}, \label{Pro_1_3} \\
\varepsilon\sum_{\ell \in \mrLs_{T_k}} b^{T_k}_\ell = \frac{2}{3} [1 - (\frac{\eps}{\mathring{h}_{T_k}})^2]\mathring{h}_{T_k} \geq \frac{1}{2} \mathring{h}_{T_k}, \label{Pro_1_4}
 \end{gather}
where $\mathring{h}_{T_k} = h_{T_k} - 6\eps$,  $C_1 =\frac{4}{\sqrt{3}} $ and $C_2 =\frac{4}{\sqrt{30}}$.
\end{proposition}

We then obtain the local efficiency estimates of the coarse-graining residual using the properties of $b^{T}$ in the following theorem.

\begin{theorem}
\label{thm: local efficiency of the cg residual}
Suppose the length of the element $T_{k}$ is greater than or equal to $8\eps$, i.e., $h_{T_k} \ge 8\eps$. We have the following efficiency estimate
\begin{align}
  \label{Eff_EF}
{\eta}^{\cg}_{T_k}= \frac{h_{T_k}}{\sqrt{2}} \|\bar{f}_{T_{k}}\|_{\ell_{\varepsilon}^{2}(\Ls_{T_k})}
\leq  C^{\cg}_1 D_{T_k}^{\frac{3}{2}}\| y'_{\a} - y'_{\ac} \|_{\ell_{\varepsilon}^{2}(\Ls_{T_{k}})}+C^{\cg}_2 D_{T_k}h_{T_k}\|f-\bar{f}_{T_{k}}\|_{\ell_{\varepsilon}^{2}(\Ls_{T_k})},
\end{align}
where
\begin{equation}
\label{eq: cg total eff const.}
D_{T_k} := \frac{h_{T_k}}{\mathring{h}_{T_k}}, \ \  C^{\cg}_1 = 3\sqrt{2}C_1 M^{\NN}_2  =4\sqrt{6} M^{\NN}_2  \text{ and } \ C^{\cg}_2 = \sqrt{2}C_2 = \frac{4}{\sqrt{15}},
\end{equation}
in which $M_2^{\NN}$ is defined in \eqref{eq: assump_on_derivs_V_1}.
\end{theorem}

\begin{proof}
Let $w^{T_{k}}=\bar{f}_{T_{k}}b^{T_{k}} \in \Us^{\eps}$ be the specifically constructed test function. Multiplying $\bar{f}_{T_{k}}$ by $w^{T_{k}}$ and sum over $\Ls_{T_k}$, we have
\begin{align}
&\varepsilon\sum_{\ell=\ell_{k-1}+1}^{\ell_{k}}\bar{f}_{T_{k}}w^{T_{k}}_\ell \notag \\
=&\varepsilon\sum_{\ell=\ell_{k-1}+1}^{\ell_{k}}f_{\ell}w^{T_{k}}_\ell+\varepsilon\sum_{\ell=\ell_{k-1}+1}^{\ell_{k}}(\bar{f}_{T_{k}}-f_{\ell})w^{T_{k}}_\ell \notag\\
=&\varepsilon\sum_{\ell=\ell_{k-1}+1}^{\ell_{k}}f_{\ell}w^{T_{k}}_\ell
-\varepsilon\sum_{\ell=\ell_{k-1}+1}^{\ell_{k}}\sigma^{\ac}_{\ell}(y_{\ac})({w^{T_{k}}})'_\ell
+\varepsilon\sum_{\ell=\ell_{k-1}+1}^{\ell_{k}}(\bar{f}_{T_{k}}-f_{\ell})w^{T_{k}}_\ell,
\label{eq: eff force step 1}
\end{align}
where we have used the property of $w^{T_{k}}$ that it vanishes near the element boundary and $\sigma^{\ac}_{\ell}(y_{\ac})$ does not change inside each element so that
\begin{equation}
\sum_{\ell=\ell_{k-1}+1}^{\ell_{k}}\sigma^{\ac}_{\ell}(y_{\ac})({w^{T_{k}}})'_\ell=\bar{\sigma}^{\ac}_{k}(y_{\ac})({w^{T_{k}}_{\ell_k-3}} - {w^{T_{k}}_{\ell_{k-1}+4}})=0 \notag.
\end{equation}
Applying the weak formulation of the atomistic problem \eqref{A_firstV} to the first term on the right hand side of \eqref{eq: eff force step 1} and using \ref{thm: property of elem bubble func} and the fact that $\mrLs_{T_k} \subset \Ls_{T_k}$, we obtain by Cauchy-Schwarz inequality
\begin{align}
&\varepsilon\sum_{\ell=\ell_{k-1}+1}^{\ell_{k}}\bar{f}_{T_{k}}w^{T_{k}}\notag \\
=&\varepsilon\sum_{\ell=\ell_{k-1}+1}^{\ell_{k}}[\sigma^{\a}_{\ell}(y_{\a})-\sigma^{\ac}_{\ell}(y_{\ac})]({w^{T_{k}}})'_\ell+\varepsilon\sum_{\ell=\ell_{k-1}+1}^{\ell_{k}}(\bar{f}_{T_{k}}-f_{\ell})w^{T_{k}}_\ell \nonumber \\
\leq &|\bar{f}_{T_{k}}| \{\|\sigma^{\a}(y_{\a})-\sigma^{\ac}(y_{\ac})\|_{\ell_{\varepsilon}^{2}(\mrLs_{T_k})}
\|{b^{T_{k}}}'\|_{\ell_{\varepsilon}^{2}(\mrLs_{T_k})}  +\|f-\bar{f}_{T_{k}}\|_{\ell_{\varepsilon}^{2}(\mrLs_{T_k})}  \|b^{T_{k}}\|_{\ell_{\varepsilon}^{2}(\mrLs_{T_k})} \} \nonumber \\
 \le & \|\bar{f}_{T_{k}}\|_{\ell_{\varepsilon}^{2}(\Ls_{T_{k}})}  \big\{C_1 (\mathring{h}_{T_k} h_{T_k})^{-\frac{1}{2}}\|\sigma^{\a}(y_{\a})-\sigma^{\ac}(y_{\ac})\|_{\ell_{\varepsilon}^{2}(\mrLs_{T_k})} +C_2 \|f-\bar{f}_{T_{k}}\|_{\ell_{\varepsilon}^{2}(\Ls_{T_k})}\big\}.
\label{eq: eff force step 2}
\end{align}

On the other hand, by \eqref{Pro_1_4} we have the stability of the left hand side of \eqref{eq: eff force step 1} such that
\begin{equation}
 \label{App_Proof_1}
\varepsilon\sum_{\ell=\ell_{k-1}+1}^{\ell_{k}}\bar{f}_{T_{k}}w^{T_{k}}_\ell=|\bar{f}_{T_{k}}|^{2}\varepsilon\sum_{\ell=\ell_{k-1}+1}^{\ell_{k}}b_{T_{k}}(\ell)\geq\frac{\mathring{h}_{T_k}}{2h_{T_k}}\|\bar{f}_{T_{k}}\|^{2}_{\ell_{\varepsilon}^{2}(\Ls_{T_{k}})}.
\end{equation}
Consequently
\begin{equation}
 \label{eq: force_efficiency_1}
\|\bar{f}_{T_{k}}\|_{\ell_{\varepsilon}^{2}(\Ls_{T_{k}})} \leq 2C_1D_{T_k}^{\frac{1}{2}} \mathring{h}^{-1}_{T_k}\|\sigma^{\a}(y_{\a})-\sigma^{\ac}(y_{\ac})\|_{\ell_{\varepsilon}^{2}(\mrLs_{T_k})}+2C_2D_{T_k}\|f-\bar{f}_{T_{k}}\|_{\ell_{\varepsilon}^{2}(\Ls_{T_k})}.
\end{equation}

Fianlly, by the definitions of $\sigma^a_\ell(y)$ in \eqref{eq: Q^a_l} and $\sigma^{\ac}_\ell(y_{\ac})$ in \eqref{eq: stress tensor ac atom wise} and \ref{ass: assump_on_derivs_V} for the interaction potential, we obtain by algebraic manipulations that
\begin{equation}
 \label{eq: force_efficiency_2}
 \|\sigma^{\a}(y_{\a})-\sigma^{\ac}(y_{\ac})\|_{\ell^{2}_{\eps}(\mrLs_{T_k})} \leq 3M^{\NN}_2\|y'-y'_{h}\|_{\ell^{2}_{\eps}(\Ls_{T_{k}})}.
\end{equation}
Combining \eqref{eq: force_efficiency_1} and \eqref{eq: force_efficiency_2}, we obtain the results stated in the theorem.
\end{proof}

The final steps of the proof implies why $b^{T_k}$ has a 'shrunken' support. Suppose $b^{T_k}$ has support on whole $\Ls_{T_k}$. We will then have $\|\sigma^{\a}(y_{\a})-\sigma^{\ac}(y_{\ac})\|_{\ell^{2}_{\eps}(\Ls_{T_k})}$ on the right hand side of \eqref{eq: force_efficiency_1}. Since the definition of $\sigma^a(y)$ is nonlocal, we will inevitably encouter the error terms $y'_\ell - y'_h|_{T_k}$ in the estimate comparable to \eqref{eq: force_efficiency_2}, where $\ell \in \{\ell_{k-1} -j\}_{j = 1}^{4}\cup \{\ell_{k} +j\}_{j = 1}^{3}$ is outside $\Ls_{T_k}$. Such error is of no interest to us but we will not be able to get rid of it unless making the assumption of the closeness between the ${y'_{\a}}$ around $\Ls_{T_k}$ and $y'_{\ac}|_{T_k}$ which may not hold especially if the element is large.

To complete the efficiency estimate of the coarse-graining residual, we need to consider the 'small' elements (the elements whose sizes are smaller than $8\eps$) which typically gather around the atomistic region. The idea for tacking this issue is simple: we just glue several 'small' elements together to make a whole piece whose size is large enough to carry out similar analysis as in the proof of Theorem \ref{thm: local efficiency of the cg residual}. In that case, the efficiency holds in the form that
\begin{align}
  \label{Eff_EF_1}
(\sum_{T}({\eta}^{\cg}_{T_k})^2)^\frac{1}{2}
\lesssim  \big(\sum_{T} \| y'_{\a} - y'_{\ac} \|_{\ell_{\varepsilon}^{2}(\Ls_{T_{k}})}^2+\sum_{T} h_{T_k}^2\|f-\bar{f}_{T_{k}}\|_{\ell_{\varepsilon}^{2}(\Ls_{T_k})}^2 \big)^{\frac{1}{2}}.
\end{align}
However, we need to note that such estimate is no longer elementwise local as opposed to similar estimate for Poisson equation. We also note that the requirement for the minimum length of an element should be $4\eps$, which should not be difficult for the mesh generation, to include the error contribution on that element. We will not pursue the precise formulation further to limit the length of the current work but move on the the discussion for the model residual.

%%%%%%%%%%%%%%%%%%%%%%%%Efficiency CG%%%%%%%%%%%%%%%%%%%%%%

%%%%%%%%%%%%%%%%%%%%%%%%Efficiency Model%%%%%%%%%%%%%%%%%%%%%%
\subsection{Efficiency of the model residual}
\label{Sec: Eff_GJ}
We proceed with the analysis for the efficiency of the model residual. Because of the complexity of the interface, we will analyze the model residuals on the nodes in $\mathring{\Ks}^c$ and those on the nodes in $\Ks^c \setminus \mathring{\Ks}^c$ separately.

\subsubsection{Away from the interface}
\label{Sec: model residual efficiency in C}

We define $\tom_k:= T_k \cup T_{k+1}$ be the union of the elements on either side of the $k$'th node and  $|\mathcal{D}|$ refers to the cardinality of a given countable set $\mathcal{D}$.
The following sets are also defined for later use:
\begin{equation}
\mrLs_{\tom_k} : = \{\ell_{k-1}+4,...,\ell_{k+1}-3\} \text{ and } \Ls_{\Lambda^{c}_{k}} := \{\ell_{k}-3,...,\ell_k+3\},
\end{equation}
which contain the indices of lattice in the 'centre' of $\tom_k$ and the indices near the boundary of $T_k$ and $T_{k+1}$. We then have the following estimate for the efficiency of the model residual.

\begin{lemma}
\label{lemma: mo eff continuum 1}
Let  $N_T := |\Ls_T|$. For $k \in\mathring{\Ks}^c$, we have
\begin{align}
\eps[(R^{\mo}_{\ell_{k}})^{2}+(R^{\mo}_{\ell_{k}+1})^{2}]  \le \eta^{\mo}_k \Big\{   \sum_{T\subset\tom_k}& C^{\mo}_{1}(T) \|y_{\a}'- y_{\ac}'\|_{\ell_{\varepsilon}^{2}(\Ls_{T})} \notag \\
+& \eps  \sum_{T \subset \tom_k} C^{\mo}_{2}(T) \|f-\bar{f}_{T}\|_{\ell_{\varepsilon}^{2}(\Ls_{T})} \Big\},
  \label{eq: Eff_GJ_1}
\end{align}
where
\begin{align}
C^{\mo}_{1}(T)=& \frac{4\sqrt{2}C^{\cg}_1}{N_T^{3/2}}D^\frac{3}{2}_T + 6M^{\NN}_2  = (\frac{32\sqrt{3}}{N_T^{3/2}}D^\frac{3}{2}_T + 6)M^{\NN}_2, \notag \\
C^{\mo}_{2}(T) =& (\frac{4\sqrt{2}C^{\cg}_2}{N_T^{1/2}}D_T+2)= \frac{32}{\sqrt{30}N_T^{1/2}}D_T +2.
\label{eq: mo eff. const. 1}
\end{align}
\end{lemma}

\begin{proof}
For any $w \in\Us^{\varepsilon}$ whose support is $\Ls_{\Lambda^{c}_{k}}$, we have by Abel transform that
\begin{align}
\varepsilon\sum_{\ell\in\Ls_{\Lambda^{c}_{k}}}\sigma^{\a}_{\ell}(y_{\ac})w'_\ell&= \sum_{\ell=\ell_{k}-3}^{\ell_{k}-1}(R^{\mo}_{\ell}-R^{\mo}_{\ell+1})w_\ell+(\sigma^{\a}_{\ell_{k}}(y_{\ac})-\sigma^{\a}_{\ell_{k}+1}(y_{\ac}))w_{\ell_{k}}\notag\\
&+\sum_{\ell=\ell_{k}+1}^{\ell_{k}+3}(R^{\mo}_{\ell}-R^{\mo}_{\ell+1})w_\ell.
\label{eq: Abel Trans 1}
\end{align}
In particular, we define the \emph{edge test function} $w^{E_k}\in\Us^{\eps}$ by
\begin{align*}
w^{E_k}_\ell=
\left\{
\begin{array}{l l}
-\eps (R^{\mo}_{\ell}+R^{\mo}_{\ell+1}), &\ell_{k}-3 \leq \ell \leq \ell_{k}-1,\\
\eps (R^{\mo}_{\ell}+R^{\mo}_{\ell+1}), & \ell_{k}+1 \le \ell  \le \ell_{k}+3,\\
0, &\text{ otherwise }. \\
\end{array} \right.
\end{align*}
Recall the definition of the model residual $R^{\mo}$ that
\begin{align*}
R^{\mo}_{\ell}=
\left\{
\begin{array}{l l}
\sigma^{\a}_{\ell}(y_{\ac})- \sigma^{\ac}_{\ell}(y_{\ac})= \sigma^{\a}_{\ell}(y_{\ac})- W'(\nabla y_{\ac}|_{T_{k}}), &\ell=\ell_{k}-2,\ell_{k}-1,\ell_{k},\\
\sigma^{\a}_{\ell}(y_{\ac})- \sigma^{\ac}_{\ell}(y_{\ac}) =\sigma^{\a}_{\ell}(y_{\ac})- W'(\nabla y_{\ac}|_{T_{k+1}}), &\ell=\ell_{k}+1,\ell_{k},\ell_{k}+2,\\
0, & \text{ otherwise },
\end{array} \right.
\end{align*}
we obtain by telescoping that
 \begin{equation}
\varepsilon\sum_{\ell\in\Ls_{\Lambda^{c}_{k}}} \sigma^{\a}_{\ell}(y_{\ac}) (w^{E_k})'_\ell=\eps [(R^{\mo}_{\ell_{k}})^{2}+(R^{\mo}_{\ell_{k}+1})^{2}].
\end{equation}

We then add and subtract $\eps \sum_{\ell \in \Ls} f_{\ell}$ and apply the weak formulation of the atomistic problem \eqref{A_firstV} to obtain
\begin{align}
\label{eq: model eff step 2}
\eps [(R^{\mo}_{\ell_{k}})^{2}+(R^{\mo}_{\ell_{k}+1})^{2}]
&=\varepsilon\sum_{\ell\in\Ls_{\Lambda^{c}_{k}}}\sigma^{\a}_{\ell}(y_{\ac}) (w^{E_k})'_\ell-\varepsilon\sum_{\ell\in\Ls} f_{\ell}w^{E_k}_\ell+\varepsilon\sum_{\ell\in\Ls_{\Lambda^{c}_{k}}}f_{\ell}w^{E_k}_\ell\notag\\
&=\varepsilon\sum_{\ell\in\Ls_{\Lambda^{c}_{k}}}f_{\ell}w^{E_k}_\ell+\varepsilon\sum_{\ell\in\Ls_{\Lambda^{c}_{k}}}[\sigma^{\a}_{\ell}(y_{\ac})-\sigma^{\a}_{\ell}(y_{\a})] (w^{E_k})'_\ell.
\end{align}

For the first term of the right hand side of \eqref{eq: model eff step 2}, by Cauchy-Schwarz inequality we have
\begin{align}
\label{eq: model eff step 2 term1}
\varepsilon\sum_{\ell\in\Ls_{\Lambda^{c}_{k}}}f_{\ell}w^{E_k}_\ell&\leq\|f\|_{\ell_{\varepsilon}^{2}(\Ls_{\Lambda^{c}_{k}})}\|w^{E_k}\|_{\ell_{\varepsilon}^{2}(\Ls_{\Lambda^{c}_{k}})}\notag\\
&\leq\sum_{T\subset\tom_k}(\|\bar{f}_{T}\|_{\ell_{\varepsilon}^{2}(\Ls_{\Lambda^{c}_{k}}\cap\Ls_T)}+\|f-\bar{f}_{T}\|_{\ell_{\varepsilon}^{2}(\Ls_{\Lambda^{c}_{k}}\cap\Ls_T)})\|w^{E_k}\|_{\ell_{\varepsilon}^{2}(\Ls_{\Lambda^{c}_{k}})}.
\end{align}
Using $R_\ell = 0 $ when $\ell\in\mrLs_{T_{k}}$ and the inequality of arithmetic means, we can further estimate $\|w^{E_k}\|_{\ell_{\varepsilon}^{2}(\Ls_{\Lambda^{c}_{k}})}$ by
\begin{align}
\|w^{E_k}\|_{\ell_{\varepsilon}^{2}(\Ls_{\Lambda^{c}_{k}} )}
&=\Big[\varepsilon\sum_{\ell\in\Ls_{\Lambda^{c}_{k}}  \setminus \{\ell_k\} }\eps^2(R^{\mo}_{\ell}+R^{\mo}_{\ell+1})^{2}\Big]^{\frac{1}{2}} \leq 2\eps  \eta^{\mo}_k.
\label{eq: edge test func prop. 1}
\end{align}

For the terms inside the bracket of \eqref{eq: model eff step 2 term1}, we apply \eqref{Eff_EF} to obtain
\begin{align}
\label{eq: model eff step 3-1}
&\|\bar{f}_{T}\|_{\ell_{\varepsilon}^{2}(\Ls_{\Lambda^{c}_{k}}\cap\Ls_T)}+\|f-\bar{f}_{T}\|_{\ell_{\varepsilon}^{2}(\Ls_{\Lambda^{c}_{k}}\cap\Ls_T)} \notag \\
\le& \big(\frac{|\Ls_{\Lambda^{c}_{k}}\cap\Ls_T|}{|\Ls_T|}\big)^{\frac{1}{2}}[\frac{\sqrt{2}C^{\cg}_1}{h_T}D_T^{\frac{3}{2}}\| y_{\a}'- y_{\ac}'\|_{\ell_{\varepsilon}^{2}(\Ls_{T})} + \sqrt{2}C^{\cg}_2D_T\|f-\bar{f}_{T}\|_{\ell_{\varepsilon}^{2}(\Ls_{T})}] \notag 
\\&+ \|f-\bar{f}_{T}\|_{\ell_{\varepsilon}^{2}(\Ls_{\Lambda^{c}_{k}}\cap\Ls_T)} \notag \\
\le& \frac{2\sqrt{2}C^{\cg}_1}{\sqrt{N_T}h_T}D^{\frac{3}{2}}_T\| y_{\a}'- y_{\ac}'\|_{\ell_{\varepsilon}^{2}(\Ls_{T})} + (\frac{2\sqrt{2}C^{\cg}_2}{\sqrt{N_T}}D_T+1)\|f-\bar{f}_{T}\|_{\ell_{\varepsilon}^{2}(\Ls_{T})},
\end{align}
where $C^{\cg}_1$ and $C^{\cg}_2$ are defined in \eqref{eq: cg total eff const.}. Note that in the last step of the analysis above, we have used the fact that $|\Ls_{\Lambda^{c}_{k}}\cap\Ls_T| \le 4$, and we have overestimated the high-order term $\|f-\bar{f}_{T}\|_{\ell_{\varepsilon}^{2}(\Ls_{\Lambda^{c}_{k}}\cap\Ls_T)}$ by $\|f-\bar{f}_{T}\|_{\ell_{\varepsilon}^{2}(\Ls_T)}$ for the simplification of presentation.

For the second term on the right hand side of \eqref{eq: model eff step 2}, we can use similar analysis as in \eqref{eq: force_efficiency_2} and the definition of $w^{E_k}$to obtain
\begin{align}
\label{eq: model eff step 3-2}
\varepsilon\sum_{\ell\in\Ls_{\Lambda^{c}_{k}}}[\sigma^{\a}_{\ell}(y_{\ac})-\sigma^{\a}_{\ell}(y_{\a})]w^{E_k}_\ell
\leq&  \|\sigma^{\a}(y_{\ac})-\sigma^{\a}(y_{\a})\|_{\ell_{\varepsilon}^{2}(\Ls_{\Lambda^{c}_{k}})} \|(w^{E_k})'\|_{\ell_{\varepsilon}^{2}(\Ls_{\Lambda^{c}_{k}})} \nonumber \\
\le &  3M^{\NN}_{2} \| y_{\a}' -  y_{\ac}' \|_{\ell_{\varepsilon}^{2}(\Ls_{\Lambda^{c}_{k}})}  \|(w^{E_k})'\|_{\ell_{\varepsilon}^{2}(\Ls_{\Lambda^{c}_{k}})}, \nonumber \\
\le & 6 M^{\NN}_{2} \| y_{\a}' -  y_{\ac}' \|_{\ell_{\varepsilon}^{2}(\Ls_{\tom_k})} \eta^{\mo}_k
\end{align}
Again, we overestimate $\| y_{\a}' -  y_{\ac}' \|_{\ell_{\varepsilon}^{2}(\Ls_{\Lambda^{c}_{k}})}$ by $ \| y_{\a}' -  y_{\ac}' \|_{\ell_{\varepsilon}^{2}(\Ls_{\tom_k})}$ in the last step in order to keep the presentation consistent. Combining \eqref{eq: edge test func prop. 1} \eqref{eq: model eff step 3-1} and \eqref{eq: model eff step 3-2}, we obtain the stated results.
\end{proof}

Similar to \ref{lemma: mo eff continuum 1} we have the following results:
\begin{lemma}
\label{lemma: mo eff continuum 2}
For each $k \in\mathring{\Ks}^c$, we have
\begin{align}
\eps[(R^{\mo}_{\ell_{k}-1})^{2}+&(R^{\mo}_{\ell_{k}+2})^{2}]  \notag \\
&\le   \eta^{\mo}_k \Big\{   \sum_{T\subset\tom_k}C^{\mo}_{1}(T) \|y_{\a}'- y_{\ac}'\|_{\ell_{\varepsilon}^{2}(\Ls_{T})}+ \eps  \sum_{T \subset \tom_k}C^{\mo}_{2}(T) \|f-\bar{f}_{T}\|_{\ell_{\varepsilon}^{2}(\Ls_{T})} \Big\},
   \label{eq: Eff_GJ_2}
\end{align}
and
\begin{align}
\eps[(R^{\mo}_{\ell_{k}-2})^{2}+&(R^{\mo}_{\ell_{k}+3})^{2}] \notag \\
&\le \eta^{\mo}_k \Big\{   \sum_{T\subset\tom_k}C^{\mo}_{1}(T) \|y_{\a}'- y_{\ac}'\|_{\ell_{\varepsilon}^{2}(\Ls_{T})}+ \eps \sum_{T \subset \tom_k} C^{\mo}_{2}(T) \|f-\bar{f}_{T}\|_{\ell_{\varepsilon}^{2}(\Ls_{T})} \Big\}.
   \label{eq: Eff_GJ_3}
\end{align}
where $C_1^{\mo}(T)$ and $C_2^{\mo}(T)$ are defined in \eqref{eq: mo eff. const. 1}.
\end{lemma}

\begin{proof}
The proof of \ref{lemma: mo eff continuum 2} is essentially an analogy to that of \ref{lemma: mo eff continuum 1} with the only modifications that the \emph{edge test function} $w^{E_k}$ is defined as
\begin{align*}
w^{E_k}_\ell=
\left\{
\begin{array}{l l}
-\eps (R^{\mo}_{\ell}+R^{\mo}_{\ell+1}), &\ell_{k}-3 \leq \ell \leq \ell_{k}-2,\\
\eps (R^{\mo}_{\ell}+R^{\mo}_{\ell+1}), & \ell_{k}+2 \le \ell  \le \ell_{k}+3,\\
0, &\text{ otherwise }, \\
\end{array} \right.
\end{align*}
for the proof of \eqref{eq: Eff_GJ_2} and
\begin{align*}
w^{E_k}_\ell =
\left\{
\begin{array}{l l}
-\eps (R^{\mo}_{\ell}+R^{\mo}_{\ell+1}), &\ell = \ell_{k}-3,\\
\eps (R^{\mo}_{\ell}+R^{\mo}_{\ell+1}), &\ell  = \ell_{k}+3,\\
0, &\text{ otherwise }, \\
\end{array} \right.
\end{align*}
for the proof of \eqref{eq: Eff_GJ_3}.
\end{proof}

The local efficiency of the model residual is then given by the following theorem.
\begin{theorem}
\label{thm: mo eff continuum}
For each $k \in\mathring{\Ks}^c$, we have
\begin{equation}
\label{eq: mo efficiency}
\eta^{\mo}_k \le 3 \sum_{T\subset\tom_k}C^{\mo}_{1}(T)\| y_{\a}'- y_{\ac}'\|_{\ell_{\varepsilon}^{2}(\Ls_{T})} + 3 \eps \sum_{T \subset \tom_k} C^{\mo}_{2}(T) \|f-\bar{f}_{T}\|_{\ell_{\varepsilon}^{2}(\Ls_{T})},
\end{equation}
where $C^{\mo}_{1}(T)$ and $C^{\mo}_2(T)$ are defined in \eqref{eq: mo eff. const. 1}.
\end{theorem}

\begin{proof}
By the identity $\eta^{\mo}_k =  \Big\{ \eps \sum_{\ell=\ell_{k}-2}^{\ell_{k+3}}\big(R^{\mo}_{\ell}\big)^{2}\Big\}^{\frac{1}{2}}$, \ref{lemma: mo eff continuum 1} and \ref{lemma: mo eff continuum 2}. The following theorem is then immediately obtained. 
\end{proof}

\subsubsection{Near the interface}
The efficiency estimate for the model residual near interfaces $\eta^{\mo}_{K_{1}-2}$ and $\eta^{\mo}_{K_{2}+2}$ is different from that inside the continuum region due to both the complexity of the formulation of $R^{\mo}_{\ell}$ at the interfaces and the wider support of $\eta^{\mo}_{K_{1}-2}$ and $\eta^{\mo}_{K_{2}+2}$. We hence give a special treatment to $\eta^{\mo}_{K_{1}-2}$ and $\eta^{\mo}_{K_{2}+2}$. For simplicity we only give the analysis in detail to $\eta^{\mo}_{K_{1}-2}$ and the analysis for $\eta^{\mo}_{K_{2}+2}$ is analogous.

We begin by separating $\eta^{\mo}_{K_{1}-2}$ into two parts:
\begin{equation}
\eta^{\mo}_{K_{1}-2} = \{(\eta^{\mo}_{\Lambda^{\c,\i}})^{2} + (\eta^{\mo}_{\Lambda^{\a,\i}})^{2}\}^{\frac{1}{2}},
\end{equation}
where
\begin{equation*}
\eta^{\mo}_{\Lambda^{\c,\i}}=(\eps \sum_{\ell=\ell_{K_{1}-2}-2}^{\ell_{K_{1}-2}}(R^{\mo}_{\ell})^{2})^{\frac{1}{2}}
\text{ and }
\eta^{\mo}_{\Lambda^{\a,\i}}=(\eps\sum_{\ell = \ell _{K_{1-1}}}^{\ell_{K_{1+2}}} (R^{\mo}_{\ell})^{2})^{\frac{1}{2}}.
\end{equation*}

The efficiency of $\eta^{\mo}_{\Lambda^{\c,\i}}$ is presented in the following theorem whose proof is the same as that of $\eta^{\mo}_k$ inside the continuum region with the only modification of $w^{E_k}$ which has support only on the left of the interface atom $\ell_{K_1-2}$, and is thus omitted.
\begin{theorem}
\label{thm: mo eff interface 1}
With the  definitions of sets of lattice indices
\begin{equation}
\Ls^{\leftSup}_{\Lambda^{\a,\i}}:=\{\ell_{k}\}_{k=K_{1}-1}^{K_{1}+2} ~\text{ and }~ \ \Ls^{\leftSup}_{\Lambda^{\i}}:=\Ls_{T_{K_{1}-2}}\cup\Ls^{\leftSup}_{\Lambda^{\a,\i}},   \nonumber \\
\end{equation}
we have
\begin{equation}
 \label{eq: eff int step 1}
\eta^{\mo}_{\Lambda^{\c,\i}} \leq 3C^{\mo}_{1}(T_{K_1-2})\| y_{\a}'- y_{\ac}'\|_{\ell_{\varepsilon}^{2}(\Ls^{\leftSup}_{\Lambda^{\i}})}+ 3 \eps C^{\mo}_{2}(T_{K_1-2}) \|f-f_{T_{K_{1}-2}}\|_{\ell_{\varepsilon}^{2}(\Ls^{\leftSup}_{\Lambda^{\i}})}.
\end{equation}
\end{theorem}

We then turn our attention to $\eta^{\mo}_{\Lambda^{\a,\i}}$, whose efficiency is given by the following theorem.
\begin{theorem}
\label{thm: mo eff interface 2}
The efficiency of the model residual on the interface $\eta^{\mo}_{\Lambda^{\a,\i}}$ is given by
\begin{equation}
 \label{eq: eff int step 2}
\eta^{\mo}_{\Lambda^{\a,\i}}\leq 3M^{\NN}_{2} \|y_{\a}' - y_{\ac}'\|_{\ell_{\varepsilon}^{2}(\Ls^{\leftSup}_{\Lambda^{\i}})}.
\end{equation}
\end{theorem}

\begin{proof}
We first construct the interface test function $w^{\interface}\in\Us^{\eps}$ satisfying
\begin{align*}
w^{\interface}_\ell=
\left\{
\begin{array}{l l}
\eps \sum_{j=\ell_{K_{1}-1}}^{\ell}R^{\mo}_{j}, &\ell\in\{\ell_{K_{1}-1},...,\ell_{K_{1}+2}\},\\
0, &\text{ otherwise }.
\end{array} \right.
\end{align*}
Noticing that
\begin{align*}
(w^{\interface})_\ell'=
\left\{
\begin{array}{l l}
R^{\mo}_{\ell}, &\ell\in\{\ell_{K_{1}-1},...,\ell_{K_{1}+2}\},\\
-\sum_{\ell=\ell_{K_{1}-1}}^{\ell_{K_{1}+2}}R^{\mo}_{\ell}, &\ell=\ell_{K_{1}+3},\\
0, &\text{ otherwise },
\end{array} \right.
\end{align*}
we consequently  have
\begin{equation}
 \label{eq: int int eff part 1}
\varepsilon\sum_{\ell=\ell_{K_{1}-1}}^{\ell_{K_{1}+2}}R^{\mo}_{\ell} (w^{\interface})_\ell' = \varepsilon\sum_{\ell=\ell_{K_{1}-1}}^{\ell_{K_{1}+2}}(R^{\mo}_{\ell})^{2}=(\eta^{\mo}_{\Lambda^{\a,\i}})^{2}.
\end{equation}

The key observation is that $\Ts_h$ and $\Ts_\eps$ coinside at the interface and atomistic region which implies that $\bar{f}_{T_k}=f_{\ell_{k}}$ when $k \in \{K_{1}-1,...,K_{1}+3\}$. Together with the definition of $R^{\mo}$, the following identity holds
\begin{align}
&\varepsilon\sum_{\ell=\ell_{K_{1}-1}}^{\ell_{K_{1}+2}}R^{\mo}_{\ell} (w^{\interface})_\ell' \notag \\
=&\varepsilon\sum_{\ell\in\Ls}R^{\mo}_{\ell} (w^{\interface})_\ell'\notag\\
=&\varepsilon\sum_{k\in\Ks^{ac}}(\sigma^{\a}_{\ell_{k}}(y_{\ac})-\bar{\sigma}^{ac}_{k}(y_{\ac})) (w^{\interface})_\ell'
+\eps \sum_{k=K_{1}-1}^{K_{1}+3}(\bar{f}_{T_k}-f_{\ell_{k}})w^{\interface}_{\ell_k}\notag\\
=&\varepsilon\sum_{ \ell \in \Ls}\big(\sigma^{\a}_{\ell_{k}}(y_{\ac}) (w^{\interface})_\ell' - f_\ell w^{\interface}_\ell \big)
+ \varepsilon\sum_{k\in\Ks^{ac}} \big(\bar{f}_{T_k}w^{\interface}_{\ell_{k}} -\bar{\sigma}^{ac}_{k}(y_{\ac}) (w^{\interface})'_{\ell_{k}} \big) ,
\label{eq: model res. eff. interface}
\end{align}
where $\bar{\sigma}^{ac}_{k}(y_{\ac})$ is defined in \eqref{eq: stress tensor ac elem wise}. By the weak formula of a/c coupling model problem \eqref{QC_firstV}, the second term in \eqref{eq: model res. eff. interface} vanishes, and by the atomistic problem  \eqref{A_firstV} and \eqref{eq: int int eff part 1},
\begin{align}
(\eta^{\mo}_{\Lambda^{\a,\i}})^{2}
=\varepsilon\sum_{\ell=\ell_{K_{1}-1}}^{\ell_{K_{1}+2}}(\sigma^{\a}_{\ell_{k}}(y_{\ac})-\sigma^{\a}_{\ell}(y_{\a}))
\leq &\|\sigma^{\a}(y_{\ac})-\sigma^{\a}(y_{\a})\|_{\ell_{\varepsilon}^{2}(\Ls^{\leftSup}_{\Lambda^{\i}})}\| (w^{\interface})'\|_{\ell_{\varepsilon}^{2}(\Ls^{\leftSup}_{\Lambda^{\i}})}\notag\\
\leq &3M^{\NN}_{2} \eta^{\mo}_{\Lambda^{\a,\i}} \| y_{\a}' -  y_{\ac}'\|_{\ell_{\varepsilon}^{2}(\Ls^{\leftSup}_{\Lambda^{\i}})},
 \label{eq: int int eff part 2}
\end{align}
which reveals the stated result by dividing both sides by $\eta^{\mo}_{\Lambda^{\a,\i}}$.
\end{proof}

Combining \eqref{eq: eff int step 1} and \eqref{eq: eff int step 2}, the efficiency of $\eta^{\mo}_{K_{1}-2}$ is given by
\begin{align}
\eta^{\mo}_{K_{1}-2}\leq 3(C^{\mo}_{1}(T_{K_1-2})+M^{\NN}_{2})& \| y_{\a}' -  y_{\ac}'\|_{\ell_{\varepsilon}^{2}(\Ls^{\leftSup}_{\Lambda^{\i}})} \notag \\
+&3\eps C^{\mo}_{2}(T_{K_1-2}) \|f-f_{T_{K_{1}-2}}\|_{\ell_{\varepsilon}^{2}(\Ls^{\leftSup}_{\Lambda^{\i}})}.
 \label{eq: eff int total}
\end{align}

We also present the efficiency estimate for $\eta^{\mo}_{K_{2}+2}$ as
\begin{align}
\eta^{\mo}_{K_{2}+2}\leq 3(C^{\mo}_{1}(T_{K_2+3})+ M^{\NN}_{2}) &\|y_{\a}' -  y_{\ac}'\|_{\ell_{\varepsilon}^{2}(\Ls^{\rightSup}_{\Lambda^{\i}})} \notag \\
+&3\eps C^{\mo}_{2}(T_{K_2+3})\|f-f_{T_{K_{2}+3}}\|_{\ell_{\varepsilon}^{2}(\Ls^{\rightSup}_{\Lambda^{\i}})},
\end{align}
where $\Ls^{\rightSup}_{\Lambda^{\i}}:=\Ls_{T_{K_{2}+3}}\cup\{\ell_{k}\}_{k=K_{2}-1}^{K_{2}+2}$, for the completeness of our analysis.

\begin{remark}
The proof of the efficiency of the model residual is subtle and is novel to the best knowledge of the authors. It is different from that of the gradient jump residual in Poisson equation (c.f. \cite[Lemma 1.3 and Equation 1.24]{VerfurthAPosteriori}), which is because of the different origins of the two residuals. The key of the proofs is the construction of the \emph{continuum edge test function} $w^{E_k}$ and the \emph{interface edge test function} $w^{\interface}$ which essentially incorporate the change of models as well as the discreteness of the underlying problems.
\end{remark}

\subsection{Comments for the Efficiency of the Residual Based A Posteriori Error Estimator}
\label{Sec: Discussion Efficiency}
Having the local efficiency of the residual is given in \ref{thm: local efficiency of the cg residual}, \ref{thm: mo eff continuum} \ref{thm: mo eff interface 1} and \ref{thm: mo eff interface 2}, the following comments can be made to help better understand our results.

We note that to establish the efficiency of the error estimator, we need to divide the stability constant $c_a(y_{\ac})$ on both sides of the estimate. By a detailed a posteriori stability estimate, c.f. \cite{CO_HW_ACC_1D_A_Posteriori, JH_XL_HW_a_Post_Contr_Mo_Adap_17}, we have that $M_2^{\NN}/c_a(y_{\ac})$ is of $\Os(1)$. Combining the above theorems, we conclude that the residual based error estimator locally provides a lower bound for the true error in the sense that
\begin{equation}
\frac{1}{c_a(y_{\ac})}\big\{(\eta^{\mo}_{T_k})^2 + (\eta^{\cg}_{T_k})^2\big\}^{\frac{1}{2}}
\le \overline{C} \Big\{   \sum_{j= k-1}^{k+1} \|y_{\a}'- y_{\ac}'\|_{\ell_{\varepsilon}^{2}(\Ls_{T_j})}+  \sum_{j= k-1}^{k+1} h_{T_j} \|f-\bar{f}_{T}\|_{\ell_{\varepsilon}^{2}(\Ls_{T_j})} \Big\},
\end{equation}
with certain modification at the a/c interface and $\overline{C}$ only depends on the mesh regularity but (almost) not on $y_{\ac}$. The estimates of the constants may be sharper if we specify the interaction potential. However, we decide to keep the generic formulation of $V$ so that our analysis can be applied to a large class of energy-based a/c method as long as it preserves the basic structure of the variational formulation.

We do not expect the so-called asymptotic exactness to hold (which means the error estimator asymptotically equals to the true error as mesh size tends to zero \cite{Babuska_Rheinboldt_81})  in our problem. This is due to the generic formulation of our atomistic model which is nonlinear, nonlocal and discrete and introduces even larger discrepancy with the coupling problem for the stress tensors as the mesh is refined towards the underline lattice. However, since the model adaptivity is imposed when the mesh size becomes small, we observe an efficiency factor (error estimator divided by the actual error) being almost constant in our numerical experiments, c.f. Fig. \ref{Fig:Efficiency_Factor_Gradient} in Section \ref{Sec: Numeric Exp}.

%%%%%%%%%%%%%%%%% Efficiency %%%%%%%%%%%%%%%%%%%%%%%%%%%%%%%%%%%%%

%%%%%%%%%%%%%%%%% Gradient Recovery%%%%%%%%%%%%%%%%%%%%%%%%%%%%%

\section{The hybrid error estimator for A/C coupling method}
\label{Sec: Gradient Recover EE}
%Having established the efficiency of the residual based error estimator, we turn our attention to the \emph{gradient recovery} type of a posteriori error estimator which is popular in the engineering and scientific computing community because of its simplicity of implementation. The gradient recovery estimator was first introduced in \cite{Zienkiewicz_Zhu_Error_Estimator_1987} for the adaptive solution for the 2D poisson equation using the finite element method. The application of the gradient recovery estimator in a/c coupling problem dates back to \cite{MO_EBT_QC_Adaptive_99} . 

Having established the efficiency of the residual based error estimator, we turn our attention to the \emph{gradient recovery} type of a posteriori error estimator which is popular in the engineering and scientific computing community. Such popularity is due to the simplicity of implementation of the gradient recovery estimator which only depends on the computed solution but not any \emph{a priori} knowledge of the external load. The gradient recovery estimator was first introduced in \cite{Zienkiewicz_Zhu_Error_Estimator_1987} for the adaptive finite element solution for Poisson equation in 2D and the application of the gradient recovery estimator in a/c coupling problem dates back to \cite{MO_EBT_QC_Adaptive_99}.

In the present section, we first derive the classical gradient recovery error estimator with ajustment to the underline coupling method and prove its equivalence with the coarse-graining residual and the model residual respectively in the continuum region. We then combine the classical gradient recovery error estimator in the continuum region and our model residual on the interface together to give a new a posteriori error estimator which only depends the computed solution $y_{\ac}$ (or $u_{\ac}$).

\subsection{Construction of the classical gradient recovery error estimator}
\label{Sec: Constr. GR EE}

To derive the gradient recovery error estimator, we first define a mesh-dependent scalar product $(\cdot,\cdot)_{h}$ on $\Ps_1(\Ts_h)$ by
\begin{align}
(v_h,w_h)_{h}&:= \sum_{k \in \Ks^{c} \cup \{K_{1}-1\}} \frac{h_{T_k}}{2} \big\{\sum_{\ell\in\mathcal{N}_{T_k}}v_{h}|_{T_k}(\eps\ell)w_{h}|_{T_k}(\eps\ell)\big\}, \ \forall v_h, w_h \in \Ps_1(\Ts_h),
  \label{inner_elementwise}
\end{align}
where $\psi|_{T_k}(x) := \lim_{t \rightarrow x; t \in T_k}\psi(t)$ and $\Ns_T$ denotes the indices of the two nodes associated with the element $T$.
With the definition $h_{\tilde{\omega}_k}:= \frac{h_{T_k}+h_{T_{k+1}}}{2}$, we can rewrite \eqref{inner_elementwise} in the nodewise form
\begin{align}
(v_h,w_h)_{h}&:=\sum_{k \in \Ks^c} h_{\tilde{\omega}_k} v^{h}_{\ell_k}w^{h}_{\ell_k} + \frac{\eps}{2}\sum_{k \in\{K_{1}-1,K_{2}+1\}}v^h_{\ell_k} w^h_{\ell_k}, \ \forall v_h, w_h \in \Ps_1(\Ts_h).
%~\forall v_h,w_h\in\Vs_{h}
 \label{inner_nodewise}
\end{align}

We define $Gu_{h} \in \Us^h$ by
\begin{equation}
 \label{projection_property}
(Gu_{h},v_h)_{h}=(\nabla u_{h},v_h)_{h},~\forall v_h \in \Ps_1(\Ts_h).
\end{equation}
By \eqref{inner_nodewise} and \eqref{projection_property}, the nodal values of $Gu_{\ac}$ are given by
\begin{align}
\label{eq: expression of Gu_h}
Gu^{\ac}_{\ell_{k}}=
\left\{
\begin{array}{l l}
\sum_{p \in \{k, k+1\}} \frac{h_p}{h_p+h_{p+1}} \nabla u_{\ac}|_{{T_p}}, &k \in  \Ks^c,\\
{\nabla u_{\ac}}|_{T_{k}},
&k ={K_{1}-1},\\
{\nabla u_{\ac}}|_{T_{k+1}},
&k={K_{2}+1}.\\
\end{array} \right.
\end{align}
The operator $G$ can be extended to $y \in \Ys_h$ with the same definition as in \eqref{eq: expression of Gu_h}.

The gradient recovery error estimator is then defined by
\begin{equation}
(\eta^{\z})^2:=\sum_{k \in \Ks^c_{\Ts_h} }(\eta^{\z}_{T_k})^{2},
\label{eq: zz est definition}
\end{equation}
with the \emph{elementwise} contribution
\begin{align}
\label{eq: zz est elementwise}
\eta^{\z}_{T_k}=
\left\{
\begin{array}{l l}
\|Gu_{\ac}-\nabla u_{\ac}\|_{\ell_{\varepsilon}^{2}(\Ls_{T_k})} \equiv \|Gy_{\ac}-\nabla y_{\ac}\|_{\ell_{\varepsilon}^{2}(\Ls_{T_k})},
&k \in \mathring{\Ks}^c_{\Ts_h},\\
\|Gu_{\ac}-\nabla u_{\ac}\|_{\ell_{\varepsilon}^{2}(\cup_{j = 0}^{4}\Ls_{T_{k+i}})} \equiv \|Gy_{\ac}-\nabla y_{\ac}\|_{\ell_{\varepsilon}^{2}(\cup_{j = 0}^{1}\Ls_{T_{k+i}})},
&k={K_{1}-2},\\
\|Gu_{\ac}-\nabla u_{\ac}\|_{\ell_{\varepsilon}^{2}(\cup_{j = 0}^{4}\Ls_{T_{k-i}})} \equiv \|Gy_{\ac}-\nabla y_{\ac}\|_{\ell_{\varepsilon}^{2}(\cup_{j = 0}^{1}\Ls_{T_{k-i}})},
&k={K_{2}+3}.\\
\end{array} \right.
\end{align}
The identity holds by $y_{\ac} - y_{\a} = u_{\ac} - u_\a$ as $u_{\cp}(x) = y_{\cp}(x) - Fx, \cp \in \{\a, \ac\}$.

Using \eqref{projection_property} \eqref{inner_elementwise} and \eqref{inner_nodewise}, we can derive an equivalent \emph{nodewise formulation} of the gradient recovery estimator, which will be used in the analysis, is given by
\begin{align}
&(G u_{\ac}- \nabla u_{\ac},Gu_{\ac}-\nabla u_{\ac})_{h}\nonumber \\
=& (\nabla u_{\ac},\nabla u_{\ac})_{h}-(G u_{\ac},Gu_{\ac})_{h} \nonumber \\
=& \sum_{k \in \Ks^c}\sum_{T\subset\widetilde{\omega}_{k}}\frac{|T|}{2}(\nabla u_{\ac}|_{T})^2  + \frac{\eps}{2}(\nabla u_{\ac}|_{T_{K_{1}-1}})^2 + \frac{\eps}{2}(\nabla u_{\ac}|_{T_{K_{2}+2}})^2 \nonumber \\
  &- \sum_{k \in \Ks^c}\frac{1}{4h_{\tilde{\omega}_k}} \Big(\sum_{T\subset\widetilde{\omega}_{k}}|T| \nabla u_{\ac}|_{T}\Big)^{2} + \frac{\eps}{2}(\nabla u_{\ac}|_{T_{K_{1}-1}})^2 + \frac{\eps}{2}(\nabla u_{\ac}|_{T_{K_{2}+2}})^2 \nonumber \\
=&\sum_{k \in \Ks^c}  \Big[ (\frac{h_{T_k}h_{T_{k+1}}}{4h_{\tilde{\omega}_k}})^{\frac{1}{2}} \big| \nabla u_{\ac}|_{T_{k+1}}  - \nabla u_{\ac}|_{T_k} \big| \Big]^2 \nonumber \\
\equiv& \sum_{k \in \Ks^c} (\frac{h_{T_k}h_{T_{k+1}}}{4h_{\tilde{\omega}_k}}) \big|  \nabla y_{\ac}|_{T_{k+1}} - \nabla y_{\ac}|_{T_k} \big|^2
=:\sum_{k \in \Ks^c} (\eta^{\z}_{k})^2,
\label{eq: zz est nodewise}
\end{align}
It can be shown that (see Appendix \ref{Proof_3})
\begin{equation}
\frac{1}{2} (\eta^{\z})^2 \le \sum_{k \in \Ks^c} (\eta^{\z}_{k})^2 \le 3  (\eta^{\z})^2.
\label{eq: zz node-element equiv}
\end{equation}

\begin{remark}
\begin{enumerate}
\item The definition of the gradient recovery operator $G$ is identical to the that for Poisson equation in the continuum region (c.f. Chapter 1.5 \cite{VerfurthAPosteriori}) but is modified near the interface since the solution in the atomistic region does not contribute to the residual based error estimator.
\item Since the values of $Gu_{\ac}$ (or $Gy_{\ac}$) is not specified in the atomistic region and the interface region, we simply understand $Gu_{\ac}$ (or $Gy_{\ac}$) as one of the elements in $\Us^h$ (or $\Ys_h$) that satisfy \eqref{eq: expression of Gu_h}.
\item As we did for the residual based estimator, we include all the interface influence in the second elements in the continuum region as shown in the last two cases in \eqref{eq: zz est elementwise}.
\end{enumerate}
\end{remark}

\subsection{Equivalence of the coarse graining residual and the gradient recovery error estimator}
\label{Sec: Equivalence}
We prove the equivalence of the gradient recovery error estimator and the coarse-graining residual. With the help of the definition of the \emph{nodewise contribution} of the gradient recovery estimator, we first present the following lemma showing the equivalence of the jumps of the stress tensor and the coarse-graining residual:

\begin{lemma}
\label{lemma: cg equivalence Q^ac}
Let $\bar{f}_{\tom_k}$ be an weighted average of the external force on $\tom_k$ defined by
\begin{equation}
\bar{f}_{\tom_k} =
\left\{
\begin{array}{l l}
 \sign\Big(\frac{h_{T_k} \bar{f}_{T_k} + h_{T_{k+1}} \bar{f}_{T_{k+1}}}{h_{T_k} + h_{T_{k+1}}} \Big)
\frac{h_{T_k} |\bar{f}_{T_k}| + h_{T_{k+1}} |\bar{f}_{T_{k+1}}|}{h_{T_k} + h_{T_{k+1}}}, &k \in \mathring{\Ks}^c,\\
\bar{f}_{T_k}, & k = K_1-2,\\
\bar{f}_{T_{k+1}}, & k = K_2+2. % \text{~\commentsy{change K1-1 into K1-2 in every later version}~}
\end{array} \right.
\end{equation}
Assume the mesh is regular such that there exists a $\kappa\in (\frac{1}{2},1)$ satisfying
\begin{equation}
\kappa \leq \frac{h_{\tilde{\omega}_k}}{h_{T_k}} \text{ and } \kappa \leq \frac{h_{\tilde{\omega}_k}}{h_{T_{k+1}}}, \quad \forall k \in \Ks^c.
\label{eq: mesh regular}
\end{equation}
Suppose further that the data oscillation satisfies
\begin{equation}
\label{eq: assump cg equivalence Q^ac}
h_{\tilde{\omega}_k}^{2} \|f- \bar{f}_{\tom_k} \|^{2}_{\ell^{2}_{\eps}(\Ls_{\widetilde{\omega}_k})} \le \min\big(\frac{\kappa}{4}, \frac{\kappa^2}{(2\kappa-1)^2}\big)(\eta^{\cg}_{k})^{2}, \ k \in \Ks^c.
\end{equation}

By the definition of $\eta^{\cg}_k$ in \eqref{eq: cg residual node}, the following equivalence holds:
 \begin{align}
\frac{\kappa}{4}(\eta^{\cg}_{k})^{2}
\le h_{\tilde{\omega}_k}  |\sigma^{\ac}_k(y_{\ac}) - \sigma^{\ac}_{k+1}(y_{\ac})|^{2}
\le \frac{9\kappa^{2}}{(2\kappa - 1)^{2}}  (\eta^{\cg}_{k})^{2}, \ \forall k \in\Ks^c.
\label{eq: cg equivalence Q^ac}
\end{align}

\end{lemma}

\begin{proof}
We first construct the \emph{discrete edge bubble function} $b^{E_k} \in \Us^h$ such that
\begin{align}
b^{E_k}_\ell=
\left\{
\begin{array}{l l}
\frac{\ell - \ell_{k-1}}{\ell_{k} - \ell_{k-1}}, &\ell \in \Ls_{T_k},\\
\frac{\ell_{k+1}-\ell}{\ell_{k+1}-\ell_{k}}, &\ell \in \Ls_{T_{k+1}},\\
0, & \text{ otherwise },
\end{array} \right.
\end{align}
whose support is $\widetilde{\omega}_k$. By the weak formulation of the coupling problem \eqref{QC_firstV} and the definition of $b^{E_k}$, we have
\begin{equation}
\sigma^{\ac}_k(y_{\ac}) - \sigma^{\ac}_{k+1}(y_{\ac})  =  \sum_{k \in \Ks^{ac}}h_{T_k}\sigma^{\ac}_{k}(y_{\ac}) \nabla b^{E_k}|_{T_k} =  \eps \sum_{\ell \in \Ls_{\widetilde{\omega}_k}} f_{\ell} b^{E_{k}}_\ell.
\label{eq: cg equivalence Q^ac anal 1}
\end{equation}

We first prove the lower bound estimate in \eqref{eq: cg equivalence Q^ac}. We rewrite \eqref{eq: cg equivalence Q^ac anal 1} as
\begin{equation}
\eps \sum_{\ell \in \Ls_{\widetilde{\omega}_k}} \bar{f}_{\tom_k}b^{E_{k}}_\ell = \sigma^{\ac}_k(y_{\ac}) - \sigma^{\ac}_{k+1}(y_{\ac})  + \eps \sum_{\ell \in \Ls_{\widetilde{\omega}_k}} (\bar{f}_{\tom_k} - f_{\ell})b^{E_{k}}_\ell.
\label{eq: cg equivalence Q^ac anal 2}
\end{equation}
For $k \in \mathring{\Ks}^c$, the square of the left hand side of \eqref{eq: cg equivalence Q^ac anal 2} times $h_{\tilde{\omega}_k}$ can be estimated by
\begin{align}
 \label{eq: cg equivalence Q^ac anal 3}
&h_{\tilde{\omega}_k} \big(\eps \sum_{\ell \in \Ls_{\widetilde{\omega}_k}} \bar{f}_{\tom_k}b_{k}(\ell)\big)^2 \notag \\
= &h_{\tilde{\omega}_k} \big( \frac{h_{T_k} |\bar{f}_{T_k}| + h_{T_{k+1}} |\bar{f}_{T_{k+1}}|}{h_{T_k} + h_{T_{k+1}}}  h_{\tilde{\omega}_k}\big)^2 \nonumber \\
\ge &\frac{1}{2}\Big(\frac{h_{\tilde{\omega}_k}}{h_{T_k}} h_{T_k}^2 \| \bar{f}_{T_k}\|^2_{\ell^{2}_{\eps}(\Ls_{T_k})} + \frac{h_{\tilde{\omega}_k}}{h_{T_k}} h_{T_{k+1}}^2 \| \bar{f}_{T_{k+1}}\|^2_{\ell^{2}_{\eps}(\Ls_{T_{k+1}})} \Big) \ge \kappa (\eta^{\cg}_k)^2,
\end{align}
where $\eta^{\cg}_k$ is defined in \eqref{eq: cg residual node}.
Applying Cauch-Schwarz inequality and inequality of the arithmetic mean to the right hand side of \eqref{eq: cg equivalence Q^ac anal 2}, we have
\begin{align}
\kappa (\eta^{\cg}_{k})^2  \leq 2h_{\tilde{\omega}_k} \Big\{   |\sigma^{\ac}_k(y_{\ac}) - \sigma^{\ac}_{k+1}(y_{\ac}) |^{2}  + h_{\tilde{\omega}_k} \|f - \bar{f}_{\tom_k}\|^{2}_{{\ell^{2}_{\eps}}(\Ls_{\widetilde{\omega}_k})}  \Big\}.
 \label{eq: equal cg upper part 1}
\end{align}
Similar analysis applies to $k = K_1-2, K_2+2$ with a light modification according to the definition of $\eta^{\cg}_k$ at the two nodes. Using assumption \eqref{eq: assump cg equivalence Q^ac}, we obtain the lower bound.

For the upper bound in \eqref{eq: cg equivalence Q^ac}, we use Cauchy Schwarz inequility and the triangle inequality to obtain
\begin{align}
h_{\tilde{\omega}_k}|\sigma^{\ac}_k(y_{\ac}) - \sigma^{\ac}_{k+1}(y_{\ac})|^2 =& h_{\tilde{\omega}_k}|\eps \sum_{\ell \in \Ls_{\widetilde{\omega}_{k}}}f_{\ell}b^{E_{k}}_\ell|^2 \notag \\
\leq & h_{\tilde{\omega}_k}^2 \Big( \|\bar{f}_{\tom_k}\|^{2}_{\ell^{2}_{{\eps}}(\Ls_{\widetilde{\omega}_k})} + \|f - \bar{f}_{\tom_k}\|^{2}_{{\ell^{2}_{\eps}}(\Ls_{\widetilde{\omega}_k})} \Big).
 \label{eq: cg equivalence Q^ac anal 4}
\end{align}
The upper bound holds simply by the assumption \eqref{eq: assump cg equivalence Q^ac} and the mesh regularity
\begin{equation}
\label{eq: mesh regular 2}
\frac{h_{\tilde{\omega}_k}}{h_{T_k}} \leq \frac{\kappa}{2\kappa - 1},  \ \frac{h_{\tilde{\omega}_k}}{h_{T_{k+1}}} \leq \frac{\kappa}{2\kappa - 1}.
\end{equation}
\end{proof}

We then prove the equivalence of the gradient recovery error estimator and the jump of the stress tensor which is given by the following lemma.

\begin{lemma}
\label{lemma: Q^ac equivalence zz}
Suppose the gradient jumps on the interface satisfy the following inequality
\begin{equation}
\max_{\substack{K^{\leftSup} \in \{K_1-1,K_1\}\\ K^{\rightSup} \in \{K_2,K_2+1\}}} \{|\frac{\nabla y_{\ac}|_{T_{K^{\leftSup}}}-\nabla y_{\ac}|_{T_{K^{\leftSup}+1}}}{\nabla y_{\ac}|_{T_{K_1-2}}-\nabla y_{\ac}|_{T_{K_1-1}}}|,|\frac{\nabla y_{\ac}|_{T_{K^{\rightSup}}}-\nabla y_{\ac}|_{T_{K^{\rightSup}+1}}}{\nabla y_{\ac}|_{T_{K_2+2}}-\nabla y_{\ac}|_{T_{K_2+3}}}|\} \leq 3.
\label{eq: assumption on interface gradient}
\end{equation}
Then for $k \in \Ks^c$, we have the following equivalence
\begin{equation}
9\big(m_2^{\NN}\big)^2(\eta^{\z}_k)^2 \le h_{\tilde{\omega}_k} |\sigma^{\ac}_k(y_{\ac}) - \sigma^{\ac}_{k+1}(y_{\ac})|^2 \le 25\big(M_2^{\NN})^2 \frac{\kappa^2}{(2\kappa-1)^2}(\eta^{\z}_k)^2.
\label{eq: Q^ac equivalence zz}
\end{equation}
where $m_2^{\NN}$ and $M_2^{\NN}$ are defined in \eqref{eq: assump_on_derivs_V_1}.
\end{lemma}

\begin{proof}
By the definition of $\sigma^{\ac}_k$ for $k \in \mathring{\Ks}^c$ and the mean value theorem we have
\begin{equation}
\sigma^{\ac}_k(y_{\ac}) - \sigma^{\ac}_{k+1}(y_{\ac}) = W'(\nabla y_{h}|_{T_k}) - W'(\nabla y_{h}|_{T_{k+1}}) = W''(\xi_k)(\nabla y_{\ac}|_{T_k} - \nabla y_{\ac}|_{T_{k+1}}),
\label{eq: Q^ac equivalence zz anal 1}
\end{equation}
where $\xi_k \in \conv(\nabla y_{\ac}|_{T_k}, \nabla y_{\ac}|_{T_{k+1}})$ and
\begin{align}
W''(F) =& 2 \partial_{1,1} V(\vec{{F}}) + 8\partial_{2,2} V(\vec{{F}})
- 2 \partial_{1,-1} V(\vec{{F}}) + 4 \partial_{1,2} V(\vec{{F}})\notag \\
& -4 \partial_{2,-1} V(\vec{{F}}) + 4 \partial_{-1,-2} V(\vec{{F}}) -4 \partial_{1,-2} V(\vec{{F}}) - 8 \partial_{2,-2} V(\vec{{F}}),
\end{align}
with $\vec{{F}}:= (F,2F,-F,-2F)$. Here we have used the symmetry that $\partial_{1, 1}V(\vec{{F}}) = \partial_{-1,-1}V(\vec{{F}})$ and $\partial_{2,2}V(\vec{{F}})=\partial_{-2,-2}V(\vec{{F}})$, and the differentiability of $V$ so that $\partial_{ji}V(\vec{\zeta}) = \partial_{ij}V(\vec{\zeta})$. Then by the definition of $\eta^{\z}_{k}$ we obtain
\begin{align*}
4W''(\xi_k)^2(\eta^{\z}_{k})^2 \le h_{\tilde{\omega}_k}|\sigma^{\ac}_k(y_{\ac}) -& \sigma^{\ac}_{k+1}(y_{\ac})|^2 \notag \\ =& 4W''(\xi_k) \frac{h_{\tilde{\omega}_k}^2}{h_{T_k} h_{T_{k+1}}}(\eta^{\z}_{k})^2 \le \frac{4\kappa^2}{(2\kappa-1)^2}W''(\xi_k)^2(\eta^{\z}_k)^2.
\end{align*}
Applying the \ref{ass: assump_on_derivs_V} we establish the estimate \eqref{eq: Q^ac equivalence zz} for $k \in \mathring{\Ks}^c$.
The analysis for $k = K_1-2, K_2+2$ are similar but more involved because of the different formulation of $\sigma^{\ac}$ on the interface. To limit the length of the present work, we put it in Appendix \ref{appendix: lemma Q^ac equal zz interface} where \eqref{eq: assumption on interface gradient} is used.
\end{proof}

Combining \ref{lemma: cg equivalence Q^ac} with \ref{lemma: Q^ac equivalence zz} and using \eqref{eq: zz node-element equiv}, we have the following equivalence
\begin{theorem}
\label{thm: equivalence of the cg residual and the zz estimator}
The following equivalence holds for the coarse-graining residual and the gradient recovery error estimator that
\begin{equation}
\label{eq: equivalence zz and cg}
\underline{C}^{\z-\cg}\eta^{\z} \le   \eta^{\cg} \le   \overline{C}^{\z-\cg} \eta^{\z},
\end{equation}
where $\eta^{\cg}$ and $\eta^{\z}$ are defined in \eqref{eq: cg residual} and \eqref{eq: zz est definition} respectively and the constants are given by
\begin{equation}
\underline{C}^{\z-\cg} = \frac{(2\kappa-1) m_{2}^{\NN}}{\sqrt{2}\kappa}
 \ \text{ and }\
 \overline{C}^{\cg-\z} =  \frac{10\sqrt{3\kappa}M_{2}^{\NN}}{2\kappa-1} .
\end{equation}
\end{theorem}

\begin{remark}
\label{remark: oscillation term}
The proof of the equivalence between the coarse-graining residual and the gradient revery error estimator essentially follows a similar line as that in \cite{Carstensen_Verfurth_99_Edge_Residual}. In order for the estimate to hold, we expect that the data oscillation $h_{\tilde{\omega}_k} \|f - \bar{f}_{\tom_k}\|_{{\ell^{2}_{\eps}}(\Ls_{\widetilde{\omega}_k})}$ is of higher order compared with $\eta^{\cg}_k$ which is proved in Appendix \ref{sec: higher order of the data oscillation}.
\end{remark}

\subsection{Equivalence of the gradient recovery error estimator and the modified model residual}
We prove the equivalence of the gradient recovery error estimator and a modification of the model residual, which will be defined in the next theorem, in the continuum region.

\begin{theorem}
\label{thm: equivalence of the model residual and the zz estimator}
Let $k \in \mathring{\Ks}^c$ and $N_{\tilde{\omega}_k} := \frac{|\Ls_{T_k} \cup \Ls_{T_{k+1}}|}{2}$. The following equivalence holds that
\begin{equation}
\label{eq: equivalence zz and mo}
\underline{C}^{\z-\mo} \eta^\z_k \le N_{\tilde{\omega}_k}^{\frac{1}{2}} (\eta^\mo_k) \le \overline{C}^{\z-\mo} \eta^\z_k,
\end{equation}
where $\eta^\mo_k$ and $\eta^\z_k$ are defined in \eqref{eq: node wise err est} and \eqref{eq: zz est nodewise} respectively and the constants are given by
\begin{equation}
\underline{C}^{\z-\mo} =  \frac{\kappa}{2}m_2^{NNN} \text{ and }\  \overline{C}^{\z-\mo} = \frac{6\kappa}{2\kappa-1}M_2^{NNN}.
\end{equation}
\end{theorem}

\begin{proof}
By the definition of $\eta^\mo_k$ and $\eta^\z_k$, we have
\begin{equation}
N_{\tilde{\omega}_k} (\eta^{\mo}_k)^2 = h_{\tilde{\omega}_k}\sum_{\ell = \ell_{k}-2}^{\ell_k+3} (R_\ell^\mo)^2 \text{ and } (\eta^\z_k)^2 = (\frac{h_{T_k}h_{T_{k+1}}}{4h_{\tilde{\omega}_k}}) \big|  \nabla y_{\ac}|_{T_{k+1}} - \nabla y_{\ac}|_{T_k} \big|^2,
\end{equation}
from which we easily expect the equivalence of the two by the definitions of $R^\mo_\ell$ in \eqref{def: node wise residual R} and the stress tensors $\sigma^a_\ell$ and $\sigma^{\ac}_\ell$ in \eqref{eq: Q^a_l} and \eqref{eq: stress tensor ac atom wise}. However, the proof is then rather tedious which consists of a load of multi variable Taylor expansion and a subtle discussion of signs and magnitudes of second order partial derivatives of $V$. Therefore, we leave detail to Appendix \ref{appendix: proof of equivalence of zz and model residual}.
\end{proof}

\subsection{The hybrid error estimator for a/c coupling method}
Having established the equivalence of the classical gradient recovery error estimator and the coarse-graining and the model residual, we are ready to propose the hybrid a posteriori error estimator for our a/c coupling method, which, in an elementwise form, is given by
\begin{align}
\label{eq: hybrid est elementwise}
\eta^{\hybrid}_{T_k}=
\left\{
\begin{array}{l l}
\big\{ \big[C^{\z-\cg} \eta^\z_{T_k}\big]^2 + \frac{(C^{\z-\mo} )^2}{2}  \big[ \tN_{k-1}^{-1} ( \eta^{\z}_{k-1})^2 + \tN_{k}^{-1} ( \eta^{\z}_{k})^2 \big] \big\}^{\frac{1}{2}},
&k \in \mathring{\Ks}^c_{\Ts_h},\\
\big\{ \big[C^{\z-\cg} \eta^\z_{T_k}\big]^2 + \frac{1}{2} \tN_{k-1}^{-1} (C^{\z-\mo} \eta^{\z}_{k-1})^2 + (\eta^{\mo}_k)^2 \big\}^{\frac{1}{2}},
&k={K_{1}-2},\\
\big\{ \big[C^{\z-\cg} \eta^\z_{T_k}\big]^2 +(\eta^{\mo}_{k-1})^2 +  \frac{1}{2}\tN_{k}^{-1} (C^{\z-\mo} \eta^{\z}_{k})^2 \big\}^{\frac{1}{2}},
&k={K_{2}+3},\\
\end{array} \right.
\end{align}
where  $c_{\a}(y_h)$ is the a posteriori stability constant and
\begin{equation}
\label{hybrid estimator constants}
C^{\z-\cg} = \frac{\underline{C}^{\z-\cg} + \overline{C}^{\z-\cg}}{2} \ \text{ and }\  C^{\z-\mo} = \frac{\underline{C}^{\z-\mo} + \overline{C}^{\z-\mo}}{2}.
\end{equation}

There are several comments we need to give at this moment.

First of all, the reason for which we use the hybrid error estimator instead of the gradient recovery error estimator is that the gradient recovery estimator may not correctly reflect the influence of the model error at the interface which may be a more serious problem in higher dimensions \cite{HW_ML_PL_LZ_2D_A_Post}. The idea behind the hybrid estimator is that we use a certain multiple of the gradient recovery estimator to approximate the residual based error estimator in the continuum region while keeping the residual based estimator on the interface whose effectiveness and efficiency have been proved.

%unlike the common practice for which the gradient recovery error estimator is used directly, we use it with a multiplication of certain constants to approximate the residual based error estimator. The reason is that the gradient recovery error estimator may not correctly reflect the influence of the model error without a proper modification. In addition to that, the model error on the interface is kept in its residual based form because of its proved effectiveness and efficiency as well as the complexity of establishing an equivalent gradient recovery form due to the special and complicated formulation of the coupling stress tensor.

Second, the computational cost of the hybrid error estimator is only of $\Os(K)$ for any generic external load $f$ as opposed to $\Os(N)$ for the residual based counterpart (we need to first compute $\|f\|_{\ell_{\varepsilon}^{2}(\Ls_{T_k})}$ to obtain any type of average $\bar{f}_{T_k}$ which essentially increase the computational cost).

Third, the constants $\underline{C}^{\z-\cg},  \overline{C}^{\z-\cg}, \underline{C}^{\z-\mo}$ and $ \overline{C}^{\z-\mo}$ are unknown because of the generic constants $M_2^{\NN}, m_2^{\NN}, M_2^{\NNN}$ and $m_2^{\NNN}$. In practice, we estimate these generic constants a posteriorily to be
\begin{align}
M_2^{\NN} = \sup_{\substack{\ell  \in \Cs \cup \Is, \\ i = \pm1}} |\partial_{ii}V(D y^h_\ell)|, \quad &m_2^{\NN} = \inf_{\substack{\ell  \in \Cs \cup \Is, \\ i = \pm1}} |\partial_{ii}V(D y^h_\ell)|, \nonumber \\
M_2^{\NNN} = \sup_{\substack{\ell \in \Cs \cup \Is,\\  (i,j) \in \Ss^{\NNN_1} }} |\partial_{ij}V(D y^h_\ell)|, \quad &
m_2^{\NNN} = \inf_{\substack{\ell \in \Cs \cup \Is,\\  (i,j) \in \Ss^{\NNN_1} }} |\partial_{ij}V(D y^h_\ell)|.
\end{align}
where $\Ss^{\NNN_1},  \Ss^{\NNN_2}$ and $ \Ss^{\NNN_3}$ are defined by \eqref{def: sets of interactions 1}, \eqref{def: sets of interactions 2} and \eqref{def: sets of interactions 3} respectively. We note that the computation cost of these constants is again of $\Os(K)$. The constants $C^{\z-\cg}$ and $C^{\z-\mo}$ are then chosen as the average of the related constants as a result of the equivalence relations \eqref{eq: equivalence zz and cg} and \eqref{eq: equivalence zz and mo}.

%%%%%%%%%%%%%%%%% Gradient Recovery%%%%%%%%%%%%%%%%%%%%%%%%%%%%%%%%%

%%%%%%%%%%%%%%%%% Numerical Exp%%%%%%%%%%%%%%%%%%%%%%%%%%%%%%%%%%%

\section{Numerical Experiments}
\label{Sec: Numeric Exp}
In this section, we present numerical experiments to illustrate the results of our analysis. We will propose an adaptive mesh refinement algorithm using the two different error estimators derived earlier, and show numerically that both estimators lead to an optimal convergence rate in terms of the number of degrees of freedom as we expect. In addition, we show the efficiency factors of the two estimators remain in a satisfactory level which is within our estimate.

With certain adjustments, the problem we consider here follows that in \cite{CO_HW_ACC_1D_A_Posteriori} which is a typical testing case in 1D. We fix our computational domain $\Omega = [-1/2 - 10^{-4}, 1/2 + 10^{-4}]$, $F = 1$, $N = 2L$, and let $V$ be the site energy given by an EAM model as
\begin{align*}
V(Dy_{\ell})= \frac{1}{2}\sum_{i \in \{1,2\};j \in \{-1,-2\}}(&\phi(D_{i}y_{\ell})+\phi(-D_{j}y_{\ell})) \notag \\ 
&+ \widetilde{F}\Big(  \sum_{i \in \{1,2\};j \in \{-1,-2\}}\big[\psi(D_{i}y_{\ell})+\psi(-D_{j}y_{\ell})\big]  \Big),
\end{align*}
where $\phi(r) = e^{-2a(r-1)}- 2e^{-a(r-1)}$, $\psi(r) = e^{-br}$ and $\widetilde{F}(\rho) = c[(\rho-\rho_{0})^2 +(\rho-\rho_{0})^4]$, with the parameter $a = 4.4$, $b=3$, $c=5$, $\rho_{0}=6e^{-b}$. We defined the external force $f_\ell$ to be
\begin{align*}
  f_\ell=
\left\{
\begin{array}{l l}
 - 0.4 (1 + \frac{1}{2\eps \ell}), &\text{for $-L \le \ell < 0$},\\
 0.4(\frac{1}{2\eps \ell} - 1), &\text{for $0 < \ell \le L$},\\
0, & \text{for $\ell \in \{-(L+5), \ldots, -(L+1), L+1, \ldots, L+5 \}$.}\\
\end{array} \right.
\end{align*}

Note that $f$ behaves essentially like $|x|^{-1}$, which is a
typical decay rate for elastic fields generated by localized defects
in 2D/3D which may be not be created by local perturbations in our 1D model and other reasons for which the external force is such defined can be found in detail in \cite[Section 6]{CO_HW_ACC_1D_A_Posteriori}. The adjustment we make here is that we leave the force zero on either boundaries of our computational domain for a purely technical reason that, according to our mesh structure introduced immediately in \ref{sec: adaptive algorithm}, the data oscillation $h_{\tom_k}\|f - \bar{f}_{\tom_k}\|^{2}_{{\ell^{2}_{\eps}}(\Ls_{\widetilde{\omega}_k})}$ of higher order compared with $\eta^{\cg}$ which is shown in Appendix \ref{sec: higher order of the data oscillation}.

\subsection{Adaptive algorithm}
\label{sec: adaptive algorithm}
We first define the error estimators
according to which we drive the mesh refinement. The element error estimators for the residual based algorithm are
given by (cf. \eqref{eq: ele wise err est} and
\eqref{eq: cg residual element})
\begin{displaymath}
  (\rho^{\rm{res}}_{T_k})^2 :=   \frac{1}{c_{\a}(y_h)} \cases{
   (\eta^{\mo}_{T_k})^2 + (\eta^{\cg}_{T_k})^2, & \text{if
    } k \in \Ks^c_{\Ts_h}, \\
    0, & \text{otherwise}.
  }
\end{displaymath}
The element error estimators for the hybrid based algorithm are given
by (cf. \ref{eq: hybrid est elementwise})
\begin{displaymath}
  \rho^{\hybrid}_{k} :=    \frac{1}{c_{\a}(y_h)} \cases{ \eta^{\hybrid}_{T_k}, & \text{if
    } k \in \Ks^c_{\Ts_h}, \\
    0, & \text{otherwise}.
  }
\end{displaymath}
Here we define the averaging force $\bar{f}_{T_k}$ to be
\begin{equation}
\bar{f}_{T_k} := \sign(f|_{T_k})h_{T_k}^{-1}\|f\|_{L^2(T_k)}.
\end{equation}
Note that $\bar{f}_{T_k}$ is well-defined since we may assume that the sign of $f$ keeps the same on any element $T_k$ for the force in our experiment.

In the following algorithm, let $\rho_{T_k} \in \{\rho_{T_k}^{\rm{res}},
\rho_{T_k}^{\hybrid}\}$. Our algorithm is based on established ideas from the
adaptive finite element literature \cite{Dorfler_Conver_Adap_Alg_Poiss}.
\begin{figure}
  \includegraphics[width=0.8\textwidth]{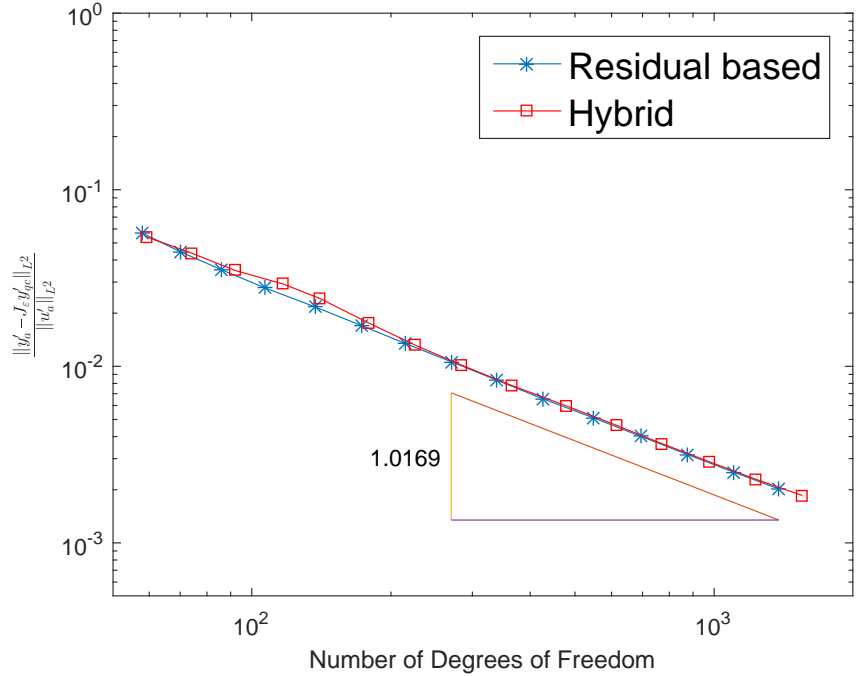}
  \caption{Relative errors in the deformation gradient
     plotted against the number of
    degrees of freedom for two types of mesh refinements.}
  \label{Fig:Relative_Error_in_Gradient}
\end{figure}

\begin{algorithm}[A posteriori mesh refinement]
\label{Algo:EDMesh}
\begin{enumerate}
\item Add the nodes $0, \pm \eps, \dots, \pm 3\eps, \pm L \eps, \pm (L+5) \eps$ to the mesh. Keep the elements $T_1 :=[-(L+5) \eps, -L\eps]$ and $T_K :=[L \eps, (L+5)\eps]$ fixed in subsequent meshes.
\item {\it Compute: } Compute the QC solution on the current mesh, compute the
estimators $\rho_{T_k}$.
\item {\it Mark:} Choose a minimal subset $\mathcal{M} \subset \{1,
  \dots, K\}$ of indices such that
  \begin{equation}
    \sum_{k \in \mathcal{M}} \rho_{T_k}^2 ~\ge~ \frac12 \sum_{k = 1}^K \rho_{T_k}^2.
  \end{equation}
\item {\it Refine: } Bisect all elements $T_k^h$ with indices
  belonging to $\mathcal{M}$. If an element that needs to be refined is adjacent to the atomistic
  region, merge this element into the atomistic region and create a
  new atomistic to continuum interface.
\item If the resulting mesh reaches a prescibed maximal number of degrees of
  freedom, stop algorithm; otherwise, go to Step (2).
\end{enumerate}
\end{algorithm}

\begin{figure}
  \includegraphics[width=0.8\textwidth]{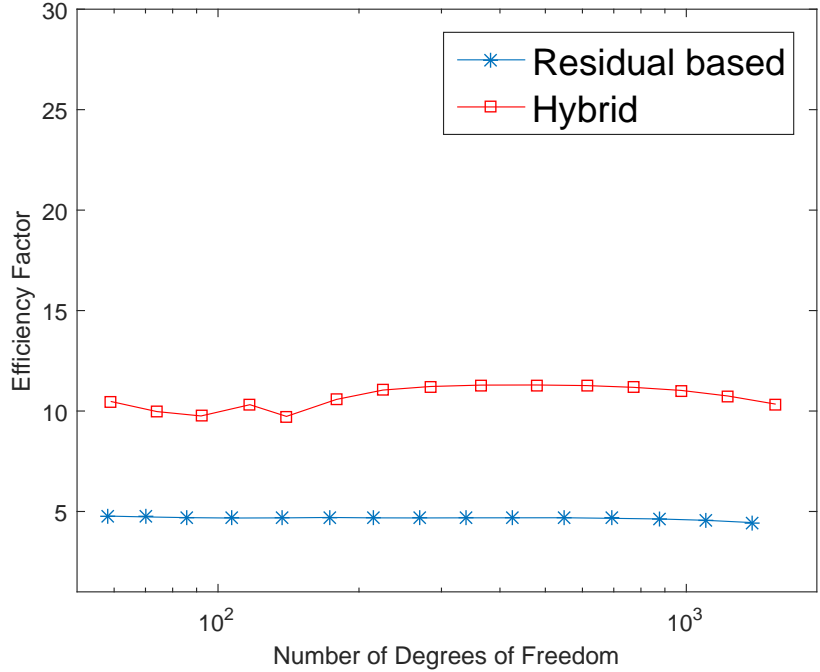}
  \caption{Efficiency factors plotted against the number of degrees of freedom
    for two types of mesh refinements.}
  \label{Fig:Efficiency_Factor_Gradient}
\end{figure}
\begin{figure}
  \includegraphics[width=0.8\textwidth]{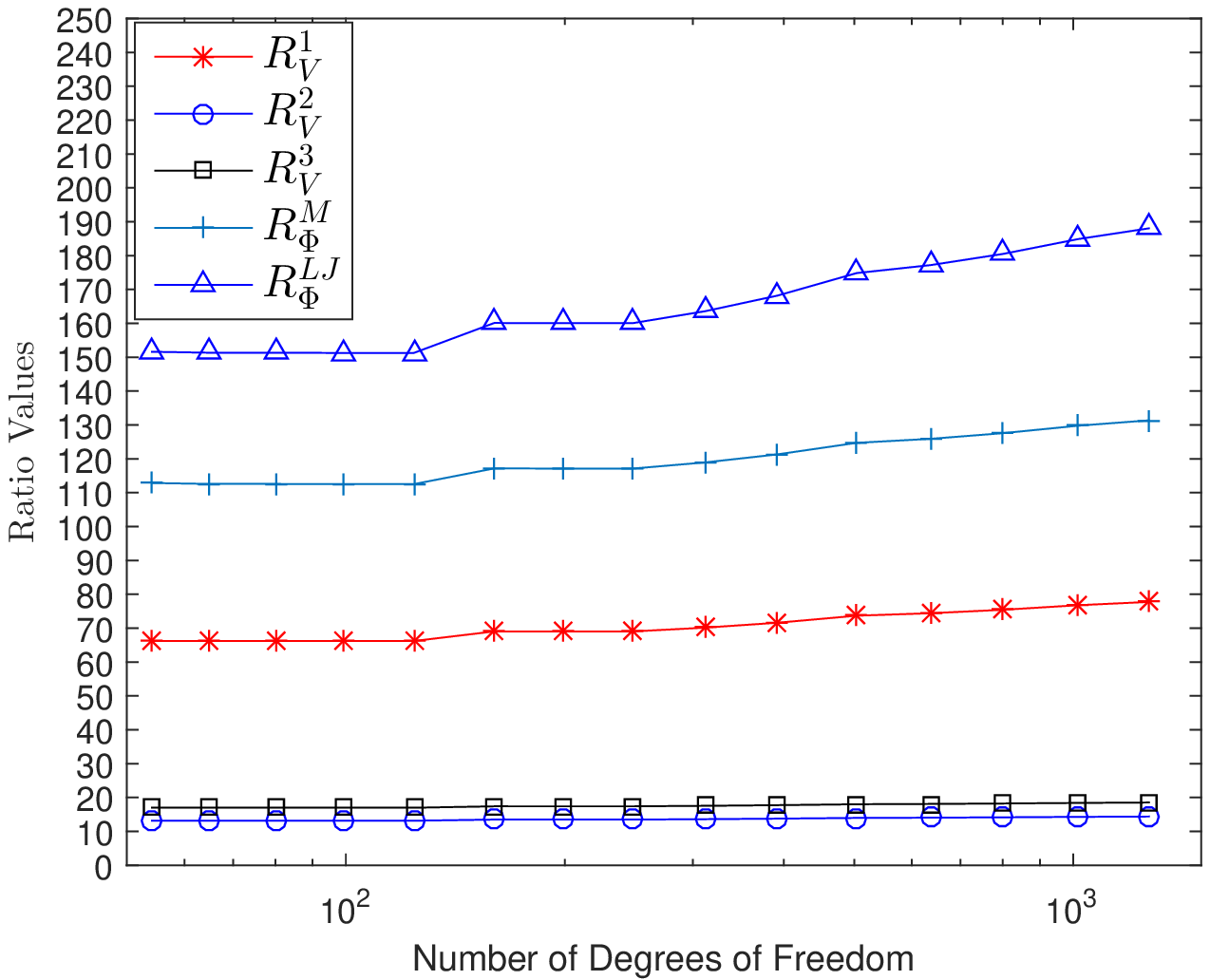}
  \caption{Second-order derivative ratios $R_V^1$, $R_V^2$ and $R_V^3$.}
  \label{Fig:derivative_ratios}
\end{figure}

\subsection{Numerical Results}
\label{Sec: numerical results}
We summarize the results of the computations with meshes generated by the adaptive algorithm with both the residual based and the hybrid error estimators. In addition, we plot the ratios between the maximum and minimum values of different groups of the second-order derivatives $\partial_{ij}V$ to support the assumptions we proposed \ref{ass: assump_on_derivs_V}.

\begin{enumerate}
\item In \ref{Fig:Relative_Error_in_Gradient} we display the
   relative errors for the two types of mesh generation algorithms. The
  differences between the results produced by the two algorithms is
  negligible. We observe the convergence rates close to $(\#{\rm DoF})^{-1}$ for both algorithms as expected.

\item The efficiency indicators (estimate divided by the actual error) are displayed in  \ref{Fig:Efficiency_Factor_Gradient}, from which we observe that both the residual based error estimator and the hybrid error estimator possess good efficiency throughout the computations. The hybrid estimator has a slightly larger efficiency factor because of the estimated constants $C^{\z-\cg}$ and $C^{\z-\mo}$ defined in \eqref{hybrid estimator constants} whose actual values are difficult (if not impossible) to track.

%Table \ref{Tab:2ed derivative} and
\item  \ref{Fig:derivative_ratios} displays the ratios of second-order derivatives $\partial_{ij}V$ and $\partial_{ij}\Phi^\a$. In particular, we define
\begin{align}
R_V^1 &:= \inf_{\substack{\ell \in \Ls, \\ i = \pm1}} |\partial_{ii}V(y^{\ac}_\ell)| / \sup_{\substack{\ell \in \Ls,\\  (i,j) \in \Ss^{\NNN_1} }} |\partial_{ij}V(y^{\ac}_\ell)|, \label{def: assumption ratio 1 many body}\\
R_V^2 &:= \inf_{\substack{\ell \in \Ls, \\  (i,j) \in \Ss^{\NNN_1} }} |\partial_{ii}V(y^{\ac}_\ell)| / \sup_{\substack{\ell \in \Ls,\\  (i,j) \in\Ss^{\NNN_2} }} |\partial_{ij}V(y^{\ac}_\ell)|, \\
R_V^3 &:= \inf_{\substack{\ell \in \Ls, \\  (i,j) \in \Ss^{\NNN_2} }} |\partial_{ii}V(y^{\ac}_\ell)| / \sup_{\substack{\ell \in \Ls,\\  (i,j) \in \Ss^{\NNN_3} }} |\partial_{ij}V(y^{\ac}_\ell)|, 
\end{align}
where where $\Ss^{\NNN_1},  \Ss^{\NNN_2}$ and $ \Ss^{\NNN_3}$ are defined by \eqref{def: sets of interactions 1}, \eqref{def: sets of interactions 2} and \eqref{def: sets of interactions 3}. We find that the nearest-neighbour derivatives $|\partial_{11}V(\gb)|$ and $|\partial_{33}V(\gb)|$ are significantly larger than other types of derivatives in terms of the absolute value which essentially reflects the nearest-neighbor dominant feature of our interaction potential.

\item In \ref{Fig:derivative_ratios}, we also test the assumption for two pair potential cases as $R^{\M}_{\Phi}$ and $R^{\LJ}_{\Phi}$, which are defined by
\begin{equation}
R_\Phi^{\M} := \inf_{\substack{\ell \in \Ls, \\ i = \pm1}} |\partial_{ii}\Phi^{\M}(y^{\ac}_\ell)| / \sup_{\substack{\ell \in \Ls,\\  (i,j) \in \Ss^{\NNN_1} }} |\partial_{ij}\Phi^{\M}(y^{\ac}_\ell)|, 
\end{equation}
and 
\begin{equation}
R_\Phi^{\LJ} := \inf_{\substack{\ell \in \Ls, \\ i = \pm1}} |\partial_{ii}\Phi^{\LJ}(y^{\ac}_\ell)| / \sup_{\substack{\ell \in \Ls,\\  (i,j) \in \Ss^{\NNN_1} }} |\partial_{ij}\Phi^{\LJ}(y^{\ac}_\ell)|, 
\end{equation}
where $\Phi^\M$ and $\Phi^\LJ$ respectively denote Morse potential and Lennard-Jones potential.  The explicit form of these two types of potential are given by
\begin{equation}
\Phi^\gamma_{\ell}(y) = \frac{1}{2}[ \phi^\gamma(y'_{\ell+1}) +  \phi^\gamma(y'_{\ell}) + \phi^\gamma(y'_{\ell+1}+y'_{\ell+2}) + \phi^\gamma(y'_{\ell-1}+ y'_{\ell} ) ], \  \gamma \in \{\M, \LJ\},
\end{equation}
and 
\begin{align}
\phi^\M(r) &= e^{-10(r-1)} - 2e^{-5(r-1)}, \\
\phi^\LJ(r)&= r^{-12} - 2r^{-6}. 
 \end{align}

We note that for pair potentials all the cross derivative terms vanish and only one ratio is related which is essentially equivalent to $R^1_V$ in \ref{def: assumption ratio 1 many body} for the many-body case.   
\end{enumerate}

We can conclude that both a posteriori error indicators can be used to select meshes that are
quasi-optimal at least for our model problem  (also c.f. \cite[Section 6 Figure 1]{CO_HW_ACC_1D_A_Posteriori} for the discussion of quasi-optimality).

%%%%%%%%%%%%%%%%% Numerical Exp%%%%%%%%%%%%%%%%%%%%%%%%%%%%%%%%%%%

%%%%%%%%%%%%%%%%% Conclusion %%%%%%%%%%%%%%%%%%%%%%%%%%%%%%%%%%%
\section{Conclusion}

We have derived and analyzed two different types of \emph{a posteriori} error estimators for the GRAC a/c coupling method in 1D. The \emph{residual based} error estimator is proved to be efficient that provides both the upper bound globally and the lower bound (up to some generic constants) locally for the true error between the solution of the coupling model and the atomistic model. Our analysis applies to generic energy-based a/c methods and interaction potentials. We then analyzed the \emph{gradient recovery} type error estimator which is easy to implement and hence is widely used in computational material science and engineering community. We proved the equivalence between the \emph{residual based} and (a modified) \emph{gradient recovery} error estimators in the continuum region. However, in order to keep the error estimator sharp on the interface which is important in the adaptive solution of our coupling model, we combine the two types of error estimators to propose a \emph{hybrid} error estimator. Our numerical experiments then indicate that both estimators give the correct convergence rate and illustrate the efficiency of the estimators.

We conclude by pointing out the merit of the extension of our analysis to higher dimensional problems. The residual based a posteriori error estimate for GRAC model in 2D has been proposed in \cite{HW_ML_PL_LZ_2D_A_Post} where the complexity of implementation is encountered. One particular difficulty is the implementation and the computation of the model error along each finite element boundary which requires the tracing of discrepancy of the geometry of the underline lattice and the coarse-grained mesh. However, our analysis of the efficiency of the residual based estimator and derivation of the hybrid estimator essentially imply that influence of the model error in the continuum region may be marginal compared with the coarse-graining error, especially on the large elements. We believe that similar phenomenon appears in higher dimensions and can be rigorously proved with careful (but maybe much more involved) consideration, for which our 1D analysis provides a valuable stepping-stone. Moreover, the \emph{hybrid} error estimator may also be extended to higher dimensions where the effect of the interface plays much more important role in adaptivity (c.f. \cite{HW_ML_PL_LZ_2D_A_Post}) and can be used for more efficient but reliable application of adaptive atomistic-to-continuum coupling methods.

%In order to obtain sharper estimates on actual errors, alternative approaches such as the goal-oriented
%approach \cite{Arndt2008b, Serge2007a} might be preferrable.

%%%%%%%%%%%%%%%%% Conclusion %%%%%%%%%%%%%%%%%%%%%%%%%%%%%%%%%%%

%%%%%%%%%%%%%%%%% Appendix%%%%%%%%%%%%%%%%%%%%%%%%%%%%%%%%%%%%%%

\appendix
\section{Detailed Formulations by Some Symbols}
\subsection{Details on the deformation gradient $Dy^h$}
We write out the specific form of $Dy^h$. Inside the continuum region where $k \in \mathring{\Ks}^c$, we define the deformation gradient $Dy_{\ell}^h$, whose support index is $\{\ell_{k-1}+2,...,\ell_{k}+1\}$, by
\begin{align*}
Dy_{\ell}^h=
\left\{
\begin{array}{l l}
(y_h\rq{}|_{T_{k}},y_h\rq{}|_{T_{k}}+y_h\rq{}|_{T_{k+1}},-y_h\rq{}|_{T_{k}},-2y_h\rq{}|_{T_{k}}),   &\ell=\ell_{k}-1,\\
(y_h\rq{}|_{T_{k+1}},2y_h\rq{}|_{T_{k+1}},-y_h\rq{}|_{T_{k}},-2y_h\rq{}|_{T_{k}}), &\ell=\ell_{k},\\
(y_h\rq{}|_{T_{k+1}},2y_h\rq{}|_{T_{k+1}},-y_h\rq{}|_{T_{k+1}},-y_h\rq{}|_{T_{k}}-y_h\rq{}|_{T_{k+1}}), &\ell=\ell_{k}+1, \\
( y_h\rq{}|_{T_{k}},2 y_h\rq{}|_{T_{k}},- y_h\rq{}|_{T_{k}},-2y_h\rq{}|_{T_{k}}),   &\text{otherwise}.
\end{array} \right.
\end{align*}

Around the interface, we have
\begin{align*}
Dy_{\ell}^h=
\left\{
\begin{array}{l l}
(y_h\rq{}|_{T_{K_1-1}},y_h\rq{}|_{T_{K_1-1}}+y_h\rq{}|_{T_{K_1}},-y_h\rq{}|_{T_{K_1-2}},-2y_h\rq{}|_{T_{K_1-2}}),   &\ell = \ell_{K_1-2},\\
(y_h\rq{}|_{T_{K_2+3}},2y_h\rq{}|_{T_{K_2+3}},-y_h\rq{}|_{T_{K_2+2}},-y_h\rq{}|_{T_{K_2+1}}-y_h\rq{}|_{T_{K_2+2}}),   &\ell=\ell_{K_2+2}.
\end{array} \right.
\end{align*}

Finally, for the atoms inside the atomistic region where $k = K_1-1,\dots,K_2+1$, the formula of $Dy^h$ is simply given by
\begin{equation}
Dy_{\ell_k}^h = (y_h\rq{}|_{T_{k+1}},y_h\rq{}|_{T_{k+1}}+y_h\rq{}|_{T_{k+2}},-y_h\rq{}|_{T_{k}},-y_h\rq{}|_{T_{k-1}}-y_h\rq{}|_{T_{k}}). \notag
\end{equation}

\subsection{Details on the stress tensors of the coupling model $\bar{\sigma}^{ac}_{k}(y_h)$}
\label{app: coupling model stress tensor}
Upon defining
\begin{align*}
&\partial_{j}V_{-}^{i}(Dy_{\ell}):=\partial_{j}V(D_{1}y_{\ell},D_{2}y_{\ell},D_{-1}y_{\ell},2D_{-1}y_{\ell}),\\
%;~~\ell=\ell_{K_{1}-1},\ell_{K_{1}}\\
&\partial_{j}V_{+}^{i}(Dy_{\ell}):=\partial_{j}V(D_{1}y_{\ell},2D_{1}y_{\ell},D_{-1}y_{\ell},D_{-2}y_{\ell}),
%;~~\ell=\ell_{K_{2}},\ell_{K_{2}+1}
\end{align*}
the elementwise stress tensors of the coupling model $\bar{\sigma}^{ac}_{k}(y_h)$ is given by
\begin{align}
\label{eq: stress tensor ac elem wise}
\bar{\sigma}^{\ac}_{k}(y_h)=
\left\{
\begin{array}{l l}
\sigma^{\a}_{\ell_{k}}(y_h),   & k\in \widetilde{\Ks}^a ,\\
%%%%%%%%%%%%%
W'(y_h\rq{}|_{T_k}),   & k\in \widetilde{\Ks}^c  ,\\
%%%%%%%%%%%%%
\frac{1}{2}W'(y_h\rq{}|_{T_{k}})-\partial_{-1}V^{i}_{-}(D y^h_{\ell_{k}})-2\partial_{-2}V^{i}_{-}(Dy^h_{\ell_{k}}), &k=K_{1}-1,\\
%%%%%%%%%%%%%
\partial_{1}V^{i}_{-}(Dy^h_{\ell_{k-1}})-\partial_{-1}V^{i}_{-}(Dy^h_{\ell_{k}})\\
 +\partial_{2}V^{i}_{-}(Dy^h_{\ell_{k-1}}) -2\partial_{-2}V^{i}_{-}(Dy^h_{\ell_{k}})-\partial_{-2}V(Dy^h_{\ell_{k+1}}),&k=K_{1},\\
%%%%%%%%%%%%%%%%%%%
\partial_{1}V^{i}_{-}(Dy^h_{\ell_{k-1}})-\partial_{-1}V(Dy^h_{\ell_{k}}) +\partial_{2}V^{i}_{-}(Dy^h_{\ell_{k-2}})\\
+\partial_{2}V^{i}_{-}(Dy^h_{\ell_{k-1}}) -\partial_{-2}V(Dy^h_{\ell_{k}})-\partial_{-2}V(Dy^h_{\ell_{k+1}}),& k=K_{1}+1,\\
%%%%%%%%%%%%%%%%%%%
\partial_{1}V(Dy^h_{\ell_{k-1}})-\partial_{-1}V(Dy^h_{\ell_{k}}) +\partial_{2}V^{i}_{-}(Dy^h_{\ell_{k-2}})\\
+ \partial_{2}V(Dy^h_{\ell_{k-1}})-\partial_{-2}V(Dy^h_{\ell_{k}})-\partial_{-2}V(Dy^h_{\ell_{k+1}}),  &k=K_{1}+2,\\
%%%%%%%%%%%%%%%%%%%
\partial_{1}V(Dy^h_{\ell_{k-1}})-\partial_{-1}V(Dy^h_{\ell_{k}}) +\partial_{2}V(Dy^h_{\ell_{k-1}})\\+
\partial_{2}V(Dy^h_{\ell_{k-2}})
-\partial_{-2}V(Dy^h_{\ell_{k}})-\partial_{-2}V^{+}_{i}(Dy^h_{\ell_{k+1}}),  &k=K_{2}-1,\\
%%%%%%%%%%%%%%%%%%%
\partial_{1}V(Dy^h_{\ell_{k-1}})-\partial_{-1}V_{+}^{i}(Dy^h_{\ell_{k}})+\partial_{2}V(Dy^h_{\ell_{k-1}})\\
+\partial_{2}V(Dy^h_{\ell_{k-2}})
-\partial_{-2}V_{+}^{i}(Dy^h_{\ell_{k}})-\partial_{-2}V_{+}^{i}(Dy^h_{\ell_{k+1}}), &k=K_{2},\\
%%%%%%%%%%%%%%%%%%%
\partial_{1}V_{+}^{i}(Dy^h_{\ell_{k-1}})-\partial_{-1}V_{+}^{i}(Dy^h_{\ell_{k}}) \\
+\partial_{2}V(Dy^h_{\ell_{k-2}}) + 2\partial_{2}V_{+}^{i}(Dy^h_{\ell_{k-1}}) -\partial_{-2}V_{+}^{i}(Dy^h_{\ell_{k}}),&k=K_{2}+1,\\
%%%%%%%%%%%%%%%%%%%
\frac{1}{2}W'(y_h\rq{}|_{T_{k}})+\partial_{1}V^{i}_{+}(D y^h_{\ell_{k-1}})+2\partial_{2}V^{i}_{+}(Dy^h_{\ell_{k-1}}), &k=K_{2}+2.
%%%%%%%%%%%%%%%%%%%
%0,&otherwise
\end{array} \right.
\end{align}
where $\vec{F} := (F, 2F, -F, -2F)$, $W'(F) := \sum_{j = \pm1} (-1)^{\frac{j-1}{2}}[\partial_j V(\vec{F})+2\partial_{2\times j} V(\vec{F})]$ and $\sigma^{\a}_{\ell_{k}}(y_h)$ is given in \eqref{eq: Q^a_l}.
%
%
%
%\subsection{$R^{\mo}_{\ell}$ inside the continuum region in Section \ref{Sec: 3.2.1.}}
%Inside the continuum region (i.e., for $\ell\in T_{k},k\in(\Ks^c$), we can elaborate the $R^{\mo}_{\ell}$ as:
%\begin{align*}
%R^{\mo}_{\ell}=
%\left\{
%\begin{array}{l l}
%\partial_{2}V_{\ell_{k}-1}(D\bar{y})-\partial_{2}V_{\ell_{k}-2}(D\bar{y}),   &\ell=\ell_{k}-2\\
%\{\partial_{1}V_{\ell_{k}-1}(D\bar{y})-\partial_{1}V_{\ell_{k}-2}(D\bar{y})\}\\
%+\{\partial_{2}V_{\ell_{k}-1}(D\bar{y})+\partial_{2}V_{\ell_{k}}(D\bar{y})-2\partial_{2}V_{\ell_{k}-2}(D\bar{y})\},&\ell=\ell_{k}-1,\\
%\{\partial_{1}V_{\ell_{k}-1}(D\bar{y})+\partial_{1}V_{\ell_{k}}(D\bar{y})-2\partial_{1}V_{\ell_{k}-2}(D\bar{y})\}\\
%+\{\partial_{2}V_{\ell_{k}-1}(D\bar{y})+\partial_{2}V_{\ell_{k}}(D\bar{y})+\partial_{2}V_{\ell_{k}+1}(D\bar{y})-3\partial_{2}V_{\ell_{k}-2}(D\bar{y})\},&\ell=\ell_{k},\\
%\{\partial_{1}V_{\ell_{k}}(D\bar{y})+\partial_{1}V_{\ell_{k}+1}(D\bar{y})-2\partial_{1}V_{\ell_{k}+2}(D\bar{y})\}\\
%+\{\partial_{2}V_{\ell_{k}-1}(D\bar{y})+\partial_{2}V_{\ell_{k}}(D\bar{y})+\partial_{2}V_{\ell_{k}+1}(D\bar{y})-3\partial_{2}V_{\ell_{k}-2}(D\bar{y})\},&\ell=\ell_{k}+1,\\
%\{\partial_{1}V_{\ell_{k}+1}(D\bar{y})-\partial_{1}V_{\ell_{k}+2}(D\bar{y})\}\\
%+\{\partial_{2}V_{\ell_{k}}(D\bar{y})+\partial_{2}V_{\ell_{k}+1}(D\bar{y})-2\partial_{2}V_{\ell_{k}+2}(D\bar{y})\},&\ell=\ell_{k}+2,\\
%\partial_{2}V_{\ell_{k}+1}(D\bar{y})-\partial_{2}V_{\ell_{k}+2}(D\bar{y}),   &\ell=\ell_{k}+3,\\
%0,& others.
%\end{array} \right.
%\end{align*}

\section{Proofs for Some Auxiliary Results}
\subsection{Proof for propositions \eqref{Pro_1_2}\eqref{Pro_1_3}\eqref{Pro_1_4}}
\label{Proof Pro}
Here, we give the proof for the two propositions \eqref{Pro_1_2}\eqref{Pro_1_3}\eqref{Pro_1_4} and estimate the values of $C_1$, $C_2$ and $C_3$.
\begin{proof}
In the following proof, we use the symbols $\ell^k_1 = \ell_{k-1}+3$ and $\ell^k_2 = \ell_{k}-3$ for simplification.

For proposition \eqref{Pro_1_2}. We can compute the discrete derivative ${b^{T_k}_\ell}'$ as
\begin{align*}
{b^{T_k}_\ell}'=
\left\{
\begin{array}{l l}
\frac{4}{h_{T_k}^2}(x_1^k+x_2^k+\eps-2\eps \ell), &\ell = \ell^k_1 +1,\dots,\ell^k_2,\\
0, &\text{ otherwise }.
\end{array} \right.
\end{align*}

Therefore, we have
\begin{equation}
\|{b^{T_k}}'\|^2_{\ell_{\eps}^{2}(\mrLs_{T_{k}})} = \frac{16\eps^3}{\mathring{h}^4_{T_k}}\sum_{\ell= \ell^k_1 +1}^{\ell^k_2}[4\ell^2 + 4(\ell^k_1 + \ell^k_2+1)\ell +  (\ell^k_1 + \ell^k_2 +1)^2].
\end{equation}

Apply the following facts:
\begin{align}
\sum_{i=1}^n i =& \frac{i(i+1)}{2}; \label{eq: 1th sum} \\
\sum_{i=1}^n i^2 =& \frac{i(i+1)(2i+1)}{6},  \label{eq: 2ed sum}
\end{align}
and then we have
\begin{equation}
\|{b^{T_k}}'\|^2_{\ell_{\eps}^{2}(\mrLs_{T_{k}})} = \frac{16}{3}[1-(\frac{\eps}{\mathring{h}_{T_k}})^2]h^{-1}_{T_k},
\end{equation}
which leads us to the result \eqref{Pro_1_2}.

For proposition \eqref{Pro_1_3}. Similarly we compute
\begin{equation}
\|b^{T_k}\|^2_{\ell_{\eps}^{2}(\mrLs_{T_{k}})} = \frac{16\eps^5}{\mathring{h}_{T_k}^4}\sum_{\ell =\ell^k_1+1}^{\ell^k_2} \{\ell^4 -2(\ell^k_1 + \ell^k_2)\ell^3 + [(\ell^k_1)^2 + 4\ell^k_1 \ell^k_2 + (\ell^k_2)^2]\ell^2 -2 \ell^k_1 \ell^k_2(\ell^k_1 + \ell^k_2)\ell + (\ell^k_1)^2 (\ell^k_2)^2  \}.
\end{equation}

By the facts
\begin{align}
\sum_{i=1}^n i^3 =& \frac{1}{4}i^2 (i+1)^2; \\
\sum_{i=1}^n i^4 =& \frac{i(i+1)(2i+1)(3i^2 +3i -1)}{30},
\end{align}
together with \eqref{eq: 1th sum} and \eqref{eq: 2ed sum}, we have
\begin{equation}
\|{b^{T_k}}'\|^2_{\ell_{\eps}^{2}(\mrLs_{T_{k}})} = \frac{8}{15}[1-(\frac{\eps}{\mathring{h}_{T_k}})^4]\mathring{h}_{T_k},
\end{equation}
and the result \eqref{Pro_1_3} can then be directly obtained by taking square root on both sides of the equation above.

For proposition \eqref{Pro_1_4}. We directly calculate that
\begin{equation}
\eps \sum_{\ell =\ell^k_1+1}^{\ell^k_2} b^{T_k}_\ell = \frac{4\eps^3}{\mathring{h}^2_{T_k}}\sum_{\ell =\ell^k_1+1}^{\ell^k_2} [-\ell^2 + (\ell^k_1 +\ell^k_2)\ell - \ell^k_1 \ell^k_2].
\end{equation}

Again by \eqref{eq: 1th sum} and \eqref{eq: 2ed sum}, we have
\begin{equation}
\eps \sum_{\ell =\ell^k_1+1}^{\ell^k_2} b^{T_k}_\ell  = \frac{2}{3}[1-(\frac{\eps}{\mathring{h}_{T_k}})^2]\mathring{h}_{T_k},
\end{equation}
which gives us proposition \eqref{Pro_1_4}.
\end{proof}

\subsection{Proof for \eqref{projected_force}}
 \label{Proof_2}

\begin{proof}
We write the inner product $\<f,y_h\>_{\varepsilon}$ as
\begin{equation}
 \label{appendix: project force proof 1}
\<f,y_h\>_{\varepsilon} = \eps \sum_{\ell \in \Cs \backslash \{\ell_{K_2+2}\}} f_{\ell}y^h_{\ell} + \eps \sum_{\ell = \ell_{K_1-1}}^{\ell_{K_2+2}} f_{\ell}y^h_{\ell},
\end{equation}
where
\begin{align}
&\varepsilon\sum_{\ell\in \Cs \backslash\{ \ell_{K_{2}+2}\}}f_{\ell}y^h_{\ell} \notag \\
=&\varepsilon\sum_{k\in\Ks^c\backslash\{K_2+2\}}\sum_{\ell \in \Ls_{T_k}}[y_{\ell_{k-1}}+\frac{y_{\ell_{k}}-y_{\ell_{k-1}}}{h_{T_k}}(\ell-\ell_{k-1})\varepsilon]f_{\ell}\notag\\
=&\varepsilon\sum_{k\in\Ks^c\backslash\{K_2+2\}}\sum_{\ell \in \Ls_{T_k}}\frac{\ell-\ell_{k-1}}{h_{T_k}}\varepsilon f_{\ell}y_{\ell_{k}}+\varepsilon\sum_{k\in \Ks^{c}\backslash\{K_2+2\}}\sum_{\ell \in \Ls_{T_{k+1}}}(1-\frac{\ell-\ell_{k}}{h_{T_{k+1}}}\varepsilon)f_{\ell}y_{\ell_{k}}\notag\\
=&\varepsilon\sum_{k\in\mathring{\Ks}^c}\big\{\varepsilon\sum_{\ell \in \Ls_{T_k}}\frac{\ell-\ell_{k-1}}{h_{T_k}}\varepsilon f_{\ell}+\sum_{\ell \in \Ls_{T_{k+1}}}(1-\frac{\ell-\ell_{k}}{h_{T_{k+1}}}\varepsilon)f_{\ell}\big\}y_{\ell_{k}}\notag \\
 +&\varepsilon\sum_{\ell \in \Ls_{T_{K_1-2}}}\frac{\ell-\ell_{K_{1}-3}}{h_{K_{1}-2}}\varepsilon f_{\ell}y_{\ell_{K_{1}-2}}+\eps \sum_{\ell \in \Ls_{T_{K_2+3}}}(1-\frac{\ell-\ell_{K_{2}+2}}{h_{K_{2}+3}}\varepsilon)f_{\ell}y_{\ell_{K_{2}+2}},
  \label{appendix: project force proof 2}
\end{align}
and
\begin{equation}
  \label{appendix: project force proof 3}
\eps \sum_{\ell = \ell_{K_1-1}}^{\ell_{K_2+2}} f_{\ell}y^h_{\ell} = \eps \sum_{k = K_1-1}^{K_2+2} f_{\ell_k}y^h_{\ell_k}.
\end{equation}

Combining \eqref{appendix: project force proof 2} and \eqref{appendix: project force proof 3}, we can write \eqref{appendix: project force proof 1} into the form of
\begin{equation*}
\<f,y_h\>_{\varepsilon}=\sum_{k\in\mathscr{K}^{ac}}\eps \bar{f}_{k}y^h_{\ell_{k}},
\end{equation*}
where the projected force $\bar{f}_{k}$ has been given in \eqref{projected_force}.

\end{proof}

\subsection{Proof for \eqref{eq: zz node-element equiv}}
 \label{Proof_3}

\begin{proof}
To prove \eqref{eq: zz node-element equiv}, we only need to show that: for any $v_{h}\in\Us^h$, we have
\begin{equation}
 \label{equal_norm_1}
(v_h,v_h)_h \sim \|v_h \|^2_{\ell^2_{\eps}(\Cs \cup \Is \backslash \{\ell_{K_2}\})}.
\end{equation}

The right part of \eqref{equal_norm_1} can be easily computed as
\begin{equation*}
\|v_h \|^2_{\ell^2_{\eps}(\Cs \cup \Is \backslash \{\ell_{K_2}\})} = A^N + B^N,
\end{equation*}
where
\begin{align}
A^N :=& \sum_{k\in \Ks_{\Ts_h}^{c}} \big\{\frac{h_{T_k}}{3}[(v^{h}_{\ell_k})^2+v^{h}_{\ell_k}v^{h}_{\ell_{k-1}}+ (v^{h}_{\ell_{k-1}})^2]
  + \frac{\eps}{2}[(v^{h}_{\ell_k})^2-(v^{h}_{\ell_{k-1}})^2]
  + \frac{\eps^2}{3h_{T_k}}[v^{h}_{\ell_k}-v^{h}_{\ell_{k-1}}]^2 \big\}, \notag \\
B^N :=&  \frac{\eps}{2}\sum_{k \in \{K_1-2,K_2+2\}}(v^{h}_{\ell_k})^2 + \eps\sum_{k \in \{K_1-1,K_2+1\}}(v^{h}_{\ell_k})^2, \notag
\end{align}

while the left part
\begin{equation*}
(v_h,v_h)_{h}= A^{IP} + B^{IP},
\end{equation*}
where
\begin{align*}
A^{IP} :=&  \frac{1}{2} \sum_{k \in \Ks_{\Ts_h}^{c}} h_{T_k} [(v^{h}_{\ell_k})^2 + (v^{h}_{\ell_{k-1}})^2], \\
B^{IP} :=& \frac{\eps}{2}\sum_{k \in \{K_1-2,K_1-1,K_2+1,K_2+2\}} (v^{h}_{\ell_k})^2.
\end{align*}

By using the fact that $h_{T_k} \geq \eps$ and the mean value inequality, we have
\begin{align}
A^N \leq& \frac{1}{2}\sum_{k\in \Ks_{\Ts_h}^{c}}(h_{T_k}+\frac{\eps^2}{3h_{T_k}})[(v^{h}_{\ell_k})^2 + (v^{h}_{\ell_{k-1}})^2] \notag \\
\le& \frac{2}{3}\sum_{k\in \Ks_{\Ts_h}^{c}} h_{T_k} [(v^{h}_{\ell_k})^2 + (v^{h}_{\ell_{k-1}})^2] = \frac{4}{3}A^{IP}
  \label{appendix: norm equal proof 1}
\end{align}
and
\begin{equation}
A^N \geq \frac{1}{6}\sum_{k\in \Ks_{\Ts_h}^{c}} h_{T_k} [(v^{h}_{\ell_k})^2 + (v^{h}_{\ell_{k-1}})^2] = \frac{1}{3}A^{IP}.
  \label{appendix: norm equal proof 2}
\end{equation}

Combining \eqref{appendix: norm equal proof 1} with \eqref{appendix: norm equal proof 2}, we have
\begin{equation*}
\frac{3}{4}A^N \le A^{IP} \le 3A^N.
\end{equation*}

On the other hand, it is straightforward that
\begin{equation*}
\frac{1}{2}B^N \le B^{IP} \le B^N.
\end{equation*}

Therefore, we obtain the equivalence as
\begin{equation*}
\frac{1}{2}\|v_h \|^2_{\ell^2_{\eps}(\Cs \cup \Is \backslash \{\ell_{K_2}\})} \leq (v_h,v_h)_h \leq 3\|v_h \|^2_{\ell^2_{\eps}(\Cs \cup \Is \backslash \{\ell_{K_2}\})}
\end{equation*}

\end{proof}

\subsection{Proof for the interface case in lemma \ref{lemma: Q^ac equivalence zz}}
\label{appendix: lemma Q^ac equal zz interface}
\begin{proof}
By the definition of $\sigma^{\ac}_{K_1-2}(y_{\ac})$ and $\sigma^{\ac}_{K_1-1}(y_{\ac})$ and the mean value theorem, we have
\begin{align}
&\sigma^{\ac}_{K_1-2}(y_{\ac}) - \sigma^{\ac}_{K_1-1}(y_{\ac}) = [\frac{1}{2}W''(\xi_{K_1-2}) + \partial_{-1,-1} V(\vec{\zeta}_{-1,-1}) - \partial_{-1,1} V(\vec{\zeta}_{-1,1})\notag \\
 &- 2\partial_{-1,2} V(\vec{\zeta}_{-1,2}) + \partial_{-1,-2} V(\vec{\zeta}_{-1,-2})
 - 2\partial_{-2,1} V(\vec{\zeta}_{-2,1}) - 4\partial_{-2,2} V(\vec{\zeta}_{-2,2}) + 2\partial_{-2,-1} V(\vec{\zeta}_{-2,-1}) \notag \\
&+ 2\partial_{-2,-2} V(\vec{\zeta}_{-2,-2})] (y_h'|_{T_{K_1-2}}-y_h'|_{T_{K_1-1}}) 
 +[-\partial_{-1,1} V(\vec{\zeta}_{-1,1})- 2\partial_{-1,2} V(\vec{\zeta}_{-1,2})\notag \\
&-2\partial_{-2,1} V(\vec{\zeta}_{-2,1})-4\partial_{-2,2} V(\vec{\zeta}_{-2,2})] (y_h'|_{T_{K_1-1}}-y_h'|_{T_{K_1}}) \notag \\
 &+[-\partial_{-1,2} V(\vec{\zeta}_{-1，2})-2\partial_{-2,2} V(\vec{\zeta}_{-2,2})] (y_h'|_{T_{K_1}}- y_h'|_{T_{K_1+1}}),
 \end{align}
where $\xi_{K_1-2}, (\zeta_{ij})_m \in \bigcup_{k \in \{K_1-2,K_1-1,K_1\}} \conv(\nabla y_h|_{T_{k}},\nabla y_h|_{T_{k+1}})$. Now we apply the assumption \ref{ass: assump_on_derivs_V} and by the similar analysis as that in the proof for lemma \ref{lemma: Q^ac equivalence zz}, we obtain
\begin{align}
\frac{9}{4}[\partial_{1,1}V(\vec{\underline{\xi}}_{K_1-2})+\partial_{3,3}V(\vec{\underline{\xi}}_{K_1-2})]^2 (\eta^{\z}_{K_1-2})^2 &\leq \th_{K_1-2}|\sigma^{\ac}_{K_1-2}(y_{\ac}) - \sigma_{K_1-1}^{ac}(y_{\ac})|^2 \notag \\
\leq& (\frac{5\kappa}{4\kappa-2})^2 [\partial_{1,1}V(\vec{\underline{\xi}}_{K_1-2})+\partial_{3,3}V(\vec{\underline{\xi}}_{K_1-2})]^2 (\eta^{\z}_{K_1-2})^2.
\end{align}

Note that here $\partial_{1,1}V(\vec{\underline{\xi}}_{K_1-2})+\partial_{3,3}V(\vec{\underline{\xi}}_{K_1-2})>0$ by \eqref{eq: assump_on_derivs_V_3}. Again we apply the assumption \ref{ass: assump_on_derivs_V}, then we observe that the equivalence \eqref{eq: Q^ac equivalence zz} still holds when $k = K_1-2$.

The analysis for the interface case $k = K_2+2$ is almost same and thus we omit its proof here.
\end{proof}
\subsection{Proof for the statement in Remark \ref{remark: oscillation term}}
\label{sec: higher order of the data oscillation}
We show that the oscillated term $h_{T_k}\|f-\bar{f}_{T_k}\|_{\ell^2_\eps(\Ls_{T_k})}$ and $\th_{T_k} \|f - \bar{f}_{\tom_k}\|_{{\ell^{2}_{\eps}}(\Ls_{\widetilde{\omega}_k})}$ are both high-order compared with $\eta_{\cg}$, where $\bar{f}_{T_k}:= \sign(f|_{T_k})\frac{\|f\|_{L^2(T_k)}}{\sqrt{h_{T_k}}}$ and the definition of $\bar{f}_{\tom_k}$ has been given in lemma \ref{lemma: cg equivalence Q^ac}.
\begin{proof}
For the term $h_{T_k}^2\|f-\bar{f}_{T_k}\|^2_{\ell^2_\eps(\Ls_{T_k})}$.

We first consider the case where $f>0$ and thus $\bar{f}_{T_k}>0$ by its definition. Upon introducing $e_{\ell} := |f_\ell -\bar{f}_{T_k}|$, we have
\begin{equation}
\label{appendix: osc 1}
\bar{f}_{T_k} e_\ell \le e_\ell (f_\ell + \bar{f}_{T_k}) = |f^2_{\ell}-f^2_{T_k}| = \frac{1}{h_{T_k}}|\int_{T_k}(f^2_\ell - f^2)dx|.
\end{equation}

Now, we expand $f(x)$ at the point $(\eps \ell_{k-1} + \eps j)$:
\begin{equation*}
f(x) = f_{\ell_{k-1}+j} + f'(\xi^{(j)})[x-(\eps \ell_{k-1} + \eps j)]\text{~for some~}\xi^{(j)}\in T_k,
\end{equation*}
which is immediately followed by
\begin{align*}
\frac{1}{h_{T_k}}|\int_{T_k}(f^2_{\ell_{k-1}+j} - f^2)dx| =& |f'(\eta_1^{(j)})f_{\ell_{k-1}+j}(h_{T_k} - 2\eps j) + \frac{1}{3}[f'(\eta_2^{(j)})]^2(h^2_k -3\eps j h_{T_k} + 3\eps^2 j^2)|  \\
\le& M_k^2 (|h_{T_k} - 2\eps j|+ \frac{1}{3}|h^2_k -3\eps j h_{T_k} + 3\eps^2 j^2|)
\end{align*}
for $j=1,2,\dots,|\Ls_{T_k}|$, where $\eta^{(j)}_1, \eta^{(j)}_2 \in T_k$ and $M_k := \max\{ \max_{y\in T_k}|f(y)|, \max_{y\in T_k}|f'(y)| \}$. It is easy to check that $|h_{T_k} - 2\eps j|\le h_{T_k}$ and $\frac{1}{4}h^2_k \le |h^2_k -3\eps j h_{T_k} + 3\eps^2 j^2| \le h^2_k$, thus we further have
\begin{equation}
\label{appendix: osc 2}
\frac{1}{h_{T_k}}|\int_{T_k}(f^2_{\ell_{k-1}+j} - f^2)dx| \le M^2_k (\frac{1}{3}h^2_k + h_{T_k}).
\end{equation}

Plugging \eqref{appendix: osc 2} into \eqref{appendix: osc 1} will give us
\begin{equation*}
e_{\ell_{k-1}+j} \le \frac{M^2_k}{\bar{f}_{T_k}}(\frac{1}{3}h^2_k + h_{T_k})
\end{equation*}
and therefore the oscillated term $h_{T_k}^2\|f-\bar{f}_{T_k}\|^2_{\ell^2_\eps(\Ls_{T_k})}$ is indeed a high-order term compared with $(\eta^\cg_{T_k})^2 = \frac{1}{2}h^3_k (\bar{f}_{T_k})^2$ since
\begin{equation*}
h_{T_k}^2\|f-\bar{f}_{T_k}\|^2_{\ell^2_\eps(\Ls_{T_k})} = h^2_k \eps \sum^{|\Ls_{T_k}|}_{j=1}e^2_{\ell_{k-1}+j}\le (\frac{M^2_k}{\bar{f}_{T_k}})^2(h_{T_k}^5 + \frac{2}{3}h^6_k + \frac{1}{9}h^7_k).
\end{equation*}

We omit the proof for the case $f<0$ where we simply apply the similar analysis, then we can reach to the conclusion that $h_{T_k}^2\|f-\bar{f}_{T_k}\|^2_{\ell^2_\eps(\Ls_{T_k})}\sim o((\eta^\cg_{T_k})^2)$.

For the term $\th_{T_k}^{2} \|f - \bar{f}_{\tom_k}\|^{2}_{{\ell^{2}_{\eps}}(\Ls_{\widetilde{\omega}_k})}$.

We consider $k\in\Ks^c\backslash\{K\}$, that is, we do not consider the node on the periodic boundary of the chain. In this way, $\sign(\bar{f}_{T_k})=\sign(\bar{f}_{T_{k+1}})=\sign(f|_{\tilde{\omega}_k})$.

Due to the fact that the formula of $\bar{f}_{\tom_k}$ inside the continuum region is different from that near the interface, we first given the proof for $k\in \mathring{\Ks}^c$. Similar as the analysis for the oscillated term $h_{T_k}^2\|f-\bar{f}_{T_k}\|^2_{\ell^2_\eps(\Ls_{T_k})}$, we first consider the case where $f|_{\tilde{\omega}_k}>0$ and therefore both $\bar{f}_{T_k}$ and $\bar{f}_{T_{k+1}}$ are positive. By the definition of $\bar{f}_{\tom_k}$ and $\eta^\cg_k$ for $k\in \mathring{\Ks}^c$,
\begin{align}
\th_{T_k}^{2} \|f - \bar{f}_{\tom_k}\|^{2}_{{\ell^{2}_{\eps}}(\Ls_{\widetilde{\omega}_k})} =& \th_{T_k}^{2} \|\frac{h_{T_k}}{2h_{\tilde{\omega}_k}}(f -\bar{f}_{T_k})+ \frac{h_{T_{k+1}}}{2h_{\tilde{\omega}_k}}(f-\bar{f}_{T_{k+1}})\|^{2}_{{\ell^{2}_{\eps}}(\Ls_{\widetilde{\omega}_k})} \notag \\
\le& 2h_{\tilde{\omega}_k}^2[\|\frac{h_{T_k}}{2h_{\tilde{\omega}_k}}(f -\bar{f}_{T_k})\|^{2}_{{\ell^{2}_{\eps}}(\Ls_{\widetilde{\omega}_k})} + \|\frac{h_{T_{k+1}}}{2h_{\tilde{\omega}_k}}(f -\bar{f}_{T_{k+1}})\|^{2}_{{\ell^{2}_{\eps}}(\Ls_{\widetilde{\omega}_k})}] \notag \\
=& \frac{1}{2}h_{T_k}^2\sum_{T \subset \tilde{\omega}_k}\|f-\bar{f}_{T_k}\|^2_{\ell^2_\eps(\Ls_{T})} + \frac{1}{2}h_{T_{k+1}}^2\sum_{T \subset \tilde{\omega}_k}\|f-\bar{f}_{T_{k+1}}\|^2_{\ell^2_\eps(\Ls_{T})} \notag \\
=& \frac{1}{2}h^2_k\|f-\bar{f}_{T_k}\|^2_{\ell^2_\eps(\Ls_{T_{k+1}})}+\frac{1}{2}h^2_{k+1}\|f-\bar{f}_{T_{k+1}}\|^2_{\ell^2_\eps(\Ls_{T_{k}})} + o((\eta^\cg_k)^2), \notag
\end{align}
where we have applied the previous conclusion $h_{T_k}^2\|f-\bar{f}_{T_k}\|^2_{\ell^2_\eps(\Ls_{T_k})}\sim o((\eta^\cg_{T_k})^2)$. Now we only need to show that
\begin{equation}
\label{appendix: osc 3}
 \frac{1}{2}h^2_k\|f-\bar{f}_{T_k}\|^2_{\ell^2_\eps(\Ls_{T_{k+1}})}+\frac{1}{2}h^2_{k+1}\|f-\bar{f}_{T_{k+1}}\|^2_{\ell^2_\eps(\Ls_{T_{k}})} \sim o((\eta^\cg_k)^2).
\end{equation}

For the term $\frac{1}{2}h^2_k\|f-\bar{f}_{T_k}\|^2_{\ell^2_\eps(\Ls_{T_{k+1}})}$, we expand $f(x)$ at the point $(\eps\ell_k + \ell j)$:
\begin{equation*}
f(x) = f_{\ell_{k}+j} + f'(\zeta^{(j)})[x-(\eps \ell_{k} + \eps j)]\text{~for some~}\zeta^{(j)}\in T_{k+1},
\end{equation*}
and then we mimic the the previous process and obtain the similar result as
\begin{equation*}
\frac{1}{h_{T_k}}|\int_{T_k}(f^2_{\ell_{k}+j} - f^2)dx| \le  {M'_k}^{2} (|h_{T_k} - 2\eps j|+ \frac{1}{3}|h^2_k -3\eps j h_{T_k} + 3\eps^2 j^2|),
\end{equation*}
for $j=1,2,\dots,|\Ls_{T_{k+1}}|$, where $M'_k := \max\{ \max_{y\in T_{k+1}}|f(y)|, \max_{y\in T_{k+1}}|f'(y)| \}$. Note that $|h_{T_k} - 2\eps j|\le h_{T_k} + h_{T_{k+1}}$ and $\frac{1}{4}h^2_k \le |h^2_k -3\eps j h_{T_k} + 3\eps^2 j^2| \le 3(h_{T_k} + h_{T_{k+1}})^2$, we obtain
\begin{equation*}
e_{\ell_{k}+j} \le \frac{{M'_k}^{2}}{\bar{f}_{T_k}}(4\th^2_k + 2h_{\tilde{\omega}_k}).
\end{equation*}

Therefore,
\begin{align}
\frac{1}{2}h_{T_k}^2\|f-\bar{f}_{T_k}\|^2_{\ell^2_\eps(\Ls_{T_{k+1}})} = \frac{1}{2}h^2_k \eps \sum^{|\Ls_{T_{k+1}}|}_{j=1}e^2_{\ell_{k}+j}\le & 2h_{T_k}^2h_{T_{k+1}}(\frac{{M'_k}^{2}}{\bar{f}_{T_k}})^2(2h_{\tilde{\omega}_k}^2 + h_{\tilde{\omega}_k})^2 \notag \\
 \le& \frac{2}{\kappa^3}(\frac{{M'_k}^{2}}{\bar{f}_{T_k}})^2(h_{\tilde{\omega}_k}^5 + 4h_{\tilde{\omega}_k}^6 + 4h_{\tilde{\omega}_k}^7).
\label{appendix: osc 4}
\end{align}

Note that the last step in \eqref{appendix: osc 4} is a direct application of the mesh regularity assumption \eqref{eq: mesh regular}. The mesh regularity further gives us
\begin{align}
(\eta^\cg_k)^2 = \frac{1}{2}[(\eta^\cg_{T_k})^2 + (\eta^\cg_{T_{k+1}})^2] =& \frac{1}{4}[h^3_k(\bar{f}_{T_k})^2 + h^3_{k+1}(\bar{f}_{T_{k+1}})^2] \notag \\
\geq& \frac{1}{4}(\frac{2\kappa -1}{\kappa})^3h_{\tilde{\omega}_k}^3[(\bar{f}_{T_k})^2 + (\bar{f}_{T_{k+1}})^2]£¬
\label{appendix: osc 5}
\end{align}
for each $k \in \mathring{\Ks}^c$. Comparing \eqref{appendix: osc 4} with \eqref{appendix: osc 5} leads us to the result $\frac{1}{2}h^2_k\|f-\bar{f}_{T_k}\|^2_{\ell^2_\eps(\Ls_{T_{k+1}})}\sim o((\eta^\cg_k)^2)$. Without the detailed proof, we also give the conclusion $\frac{1}{2}h^2_{k+1}\|f-\bar{f}_{T_{k+1}}\|^2_{\ell^2_\eps(\Ls_{T_{k}})} \sim o((\eta^\cg_k)^2)$ by a similar analysis. Therefore, we have now obtained \eqref{appendix: osc 3} and thus finish the proof for the case $f|_{\tilde{\omega}_k}>0$. Again, we omit the proof for the case $f|_{\tilde{\omega}_k}<0$ since the process is almost the same.

The analysis above contains the proof the special interface case, where the formula of $\bar{f}_{\tom_k}$ is slightly different. In order to prevent the proof from being too tedious, we do not bother with the interface case proof.
\end{proof}

\subsection{Detailed proof for Theorem \ref{thm: equivalence of the model residual and the zz estimator}}
\label{appendix: proof of equivalence of zz and model residual}
\begin{proof}

\begin{table}[h]
\centering
\begin{tabular}{c c c c c} % creating 10 columns
\hline\hline % inserting double-line
 \multicolumn{2}{c}{Type of derivatives} &Sign & $\max_{\gb \in E^{\times 4}}|\cdot|$ & $\min_{\gb \in E^{\times 4}}|\cdot|$
\\ [0.5ex]
\hline % inserts single-line
$\partial_{1,1}V(\gb)$ & $\partial_{-1,-1}V(\gb)$ & $+$ & 24.7302& 22.6129\\ % Entering row contents
$\partial_{1,-1}V(\gb)$ & $\partial_{-1,1}V(\gb)$ & $-$ & 0.27510& 0.26670\\
$\partial_{22}V(\gb)$ & $\partial_{-2,-2}V(\gb)$ & $-$ & 0.21399& 0.20367\\
$\partial_{12}V(\gb)$ & $\partial_{21}V(\gb)$ & $+$ & 0.01374& 0.01310\\
$\partial_{1,-2}V(\gb)$ & $\partial_{-2,1}V(\gb)$ & $-$ & 0.01374& 0.01310\\
$\partial_{2,-1}V(\gb)$ & $\partial_{-1,2}V(\gb)$ & $-$ & 0.01374& 0.01310\\
$\partial_{-1,-2}V(\gb)$ & $\partial_{-2,-1}V(\gb)$ & $+$ & 0.01374& 0.01310\\
$\partial_{2,-2}V(\gb)$ & $\partial_{-2,2}V(\gb)$ & $-$ & 0.00069& 0.00064\\[1ex] % [1ex] adds vertical space
\hline % inserts single-line
\end{tabular}
 \caption{Range of the absolute value of the second-order derivative $\partial_{ij}V(\gb)$ where $\gb \in E^{\times 4}$ (the definition of the set $E^{\times 4}$ can be found in Section \ref{Sec: Notation}); ``$+$" (``$-$") indicates that the corresponding values are positive (negative).}
\label{Tab:2ed derivative}
\end{table}
We first look at one of the gradient jump terms. For the term $R^{mo}_{\ell}$, applying the mean value theorem allows us to obtain that
\begin{align}
R^{\mo}_{\ell_k - 2} &= \sigma^\a_{\ell_k - 2}(y^{\ac}) -  \sigma^{\ac}_{\ell_k - 2}(y^{\ac}) = \partial_{-2} V( y_h'|_{T_{k}},2 y_h'|_{T_{k}},- y_h'|_{T_{k}},-2 y_h'|_{T_{k}}) \notag \\
-&\partial_{-2} V(y_h'|_{T_{k}}, y_h'|_{T_{k}}+ y_h'|_{T_{k+1}},- y_h'|_{T_{k}},-2 y_h'|_{T_{k}}) \notag \\
&= -\partial_{-2,2}V( y_h'|_{T_{k}},2\xi ,- y_h'|_{T_{k}},-2 y_h'|_{T_{k}})( y_{h}'|_{T_{k+1}} -  y_{h}'|_{T_{k}}),
\label{appendix: theorem proof 1}
\end{align}
for some $\xi \in \conv(\nabla y_{h}|_{T_k} , \nabla y_{h}|_{T_{k+1}})$.

In order for simplification, we let the symbol $\partial_{-2,2}V(\vec{\zeta}^{-2,2}_{\ell_{k}-2})$ denote the partial derivative in \eqref{appendix: theorem proof 1}, where $\vec{\zeta}^{-2,2}_{\ell_{k}-2}$ can be regraded as a vector $(\zeta^{-2,2}_{\ell_{k}-2,1},2\zeta^{-2,2}_{\ell_{k}-2,2},-\zeta^{-2,2}_{\ell_{k}-2,3},-2\zeta^{-2,2}_{\ell_{k}-2,4})$ and $\zeta^{-2,2}_{\ell_{k}-2,j} \in \conv(\nabla y_{h}|_{T_k} , \nabla y_{h}|_{T_{k+1}})$. Therefore, we similarly can define other $\partial_{ij}V(\vec{\zeta}^{ij}_{\ell})$ and thus by the formulas of $\sigma^\a_{\ell}(y^{\ac})$ and $\sigma^{\ac}_{\ell_k}(y^{\ac})$ respectively given in \eqref{eq: Q^a_l} and \eqref{eq: stress tensor ac atom wise} we have
\begin{align}
&R^{\mo}_{\ell_k - 2} = -\partial_{-2,2}V(\vec{\zeta}^{-2,2}_{\ell_{k}-2})( y_{h}'|_{T_{k+1}} -  y_{h}'|_{T_{k}}); \notag \\
&R^{\mo}_{\ell_k - 1} = -[\partial_{-1,2}V(\vec{\zeta}^{-1,2}_{\ell_{k}-1})+\partial_{-2,2}V(\vec{\zeta}^{-2,2}_{\ell_{k}-1})+2\partial_{-2,2}V(\vec{\xi}^{-2,2}_{\ell_{k}-1})+\partial_{-2,1}V(\vec{\zeta}^{-2,1}_{\ell_{k}-1})]( y_{h}'|_{T_{k+1}} -  y_{h}'|_{T_{k}}); \notag \\
&R^{\mo}_{\ell_k} = [\partial_{1,2}V(\vec{\zeta}^{1,2}_{\ell_{k}})+\partial_{2,2}V(\vec{\zeta}^{2,2}_{\ell_{k}})-\partial_{-1,1}V(\vec{\zeta}^{-1,1}_{\ell_{k}})-2\partial_{-1,2}V(\vec{\zeta}^{-1,2}_{\ell_{k}}) \notag \\
-&\partial_{-2,1}V(\vec{\zeta}^{-2,1}_{\ell_{k}})-\partial_{-2,1}V(\vec{\xi}^{-2,1}_{\ell_{k}}) 
-2\partial_{-2,2}V(\vec{\zeta}^{-2,2}_{\ell_{k}})-2\partial_{-2,2}V(\vec{\xi}^{-2,2}_{\ell_{k}})\notag \\
+&\partial_{-2,-1}V(\vec{\zeta}^{-2,-1}_{\ell_{k}})+\partial_{-2,-2}V(\vec{\zeta}^{-2,-2}_{\ell_{k}})]( y_{h}'|_{T_{k+1}} -  y_{h}'|_{T_{k}}); \notag \\
&R^{\mo}_{\ell_k+1} = [\partial_{1,-1}V(\vec{\zeta}^{1,-1}_{\ell_{k}+1})+2\partial_{1,-2}V(\vec{\zeta}^{1,-2}_{\ell_{k}+1})-\partial_{21}V(\vec{\zeta}^{21}_{\ell_{k}+1})-\partial_{22}V(\vec{\zeta}^{22}_{\ell_{k}+1})\notag \\
+&\partial_{2,-1}V(\vec{\zeta}^{2,-1}_{\ell_{k}+1})+\partial_{2,-1}V(\vec{\xi}^{2,-1}_{\ell_{k}+1}) 
+2\partial_{2,-2}V(\vec{\zeta}^{2,-2}_{\ell_{k}+1})+2\partial_{2,-2}V(\vec{\xi}^{2,-2}_{\ell_{k}+1})\notag \\
-&\partial_{-1,-2}V(\vec{\zeta}^{-1,-2}_{\ell_{k}+1})-\partial_{-2,-2}V(\vec{\zeta}^{-2,-2}_{\ell_{k}+1})]( y_{h}'|_{T_{k+1}} -  y_{h}'|_{T_{k}}); \notag \\
&R^{\mo}_{\ell_k +2} = [\partial_{1,-2}V(\vec{\zeta}^{1,-2}_{\ell_{k}+2})+\partial_{2,-2}V(\vec{\zeta}^{2,-2}_{\ell_{k}+2})+2\partial_{2,-2}V(\vec{\xi}^{2,-2}_{\ell_{k}+2})+\partial_{2,-1}V(\vec{\zeta}^{2,-1}_{\ell_{k}+2})]( y_{h}'|_{T_{k+1}} -  y_{h}'|_{T_{k}}); \notag \\
&R^{\mo}_{\ell_k +3} = \partial_{2,-2}V(\vec{\zeta}^{2,-2}_{\ell_{k}+3})( y_{h}'|_{T_{k+1}} -  y_{h}'|_{T_{k}}),
\end{align}
where $\vec{\eta}^{ij}_{\ell}:=(\eta^{ij}_{\ell,1},2\eta^{ij}_{\ell,2},-\eta^{ij}_{\ell,3},-2\eta^{ij}_{\ell,4})\in \R^{1\times4}$, $\eta \in \{\zeta,\xi\}$ and $\vec{\eta}^{ij}_{\ell,m} \in \conv(\nabla y_{h}|_{T_k} , \nabla y_{h}|_{T_{k+1}})$ for $m=\pm1,\pm2$. In general, a type of multi-body potential has the following property in terms of its second-order derivative:
\begin{align*}
\sign(\partial_{ij}V(\vec{\eta}^{ij}_{\ell}))=
\left\{
\begin{array}{l l}
1,   &(i,j)=\pm(1,1),\pm(1,2),\\
-1,   &\text{otherwise},
\end{array} \right.
\end{align*}
for every $\ell \in \Cs$. This property can be confirmed by Tab \ref{Tab:2ed derivative} where we have calculated the values of each second-order derivative for a certain type of multi-body potential.

Once we obtain the sign of the derivatives, we can estimate each gradient jump term $R^{\mo}_{\ell}$ with the help of the assumption \ref{ass: assump_on_derivs_V} as
\begin{align*}
0 \le  &|R^{\mo}_{\ell_k-2}|,|R^{\mo}_{\ell_k+3}| \le \frac{1}{100}m^{\NNN}_2 |\nabla y_{h}|_{T_{k+1}} - \nabla y_{h}|_{T_{k}}|; \\
 \frac{1}{5}m_2^{\NNN}|\nabla y_{h}|_{T_{k+1}} - \nabla y_{h}|_{T_{k}}| \le &|R^{\mo}_{\ell_k-1}|,|R^{\mo}_{\ell_k+1}|  \le \frac{23}{100}M_2^{\NNN}|\nabla y_{h}|_{T_{k+1}} - \nabla y_{h}|_{T_{k}}|; \\
 0 \le &|R^{\mo}_{\ell_k}|,|R^{\mo}_{\ell_k+1}| \le 2M_2^{\NNN}|\nabla y_{h}|_{T_{k+1}} - \nabla y_{h}|_{T_{k}}|.
 \end{align*}

Further calculation combining the mesh regularity assumption \eqref{eq: mesh regular} \eqref{eq: mesh regular 2} and the definition of $\eta^z_k$ in \eqref{eq: zz est nodewise} gives us the result in theorem \ref{thm: equivalence of the model residual and the zz estimator}.

\end{proof}

%%%%%%%%%%%%%%%%% Appendix%%%%%%%%%%%%%%%%%%%%%%%%%%%%%%%%%%%%%%

\bibliographystyle{plain}
\bibliography{qc1}

\end{document}